\documentclass[11pt,reqno]{amsart}
\usepackage{amsmath,amsfonts,amsthm,color,amssymb, mathrsfs}

\usepackage{pxfonts}

\makeatletter
\@tfor\@tempa:=
\alpha \beta \gamma \delta \epsilon \zeta \eta \theta \iota \kappa \lambda \mu
\nu \xi \pi \rho \sigma \tau \upsilon \phi \chi \psi \omega \varepsilon
\vartheta \varpi \varrho \varsigma \varphi  
\do {
	\count255 \numexpr\@tempa+"7000\relax
	\expandafter\mathchardef\@tempa \count255 }
\makeatother

\usepackage{xcolor}
\usepackage[T1]{fontenc}
\usepackage{wasysym}
\usepackage[normalem]{ulem}
\usepackage{stmaryrd} 

\usepackage[left=1in, right=1in, top=1.1in,bottom=1.1in]{geometry}
\usepackage{enumitem}
\usepackage{hyperref}
\hypersetup{
	colorlinks   = true,
	citecolor    = blue,
	linkcolor=blue
}
\allowdisplaybreaks
\usepackage{csquotes}
\usepackage{graphicx}
\usepackage{pstricks}
\usepackage{lmodern}

\newtheorem{theorem}{Theorem}[section]

\newtheorem{assumption}{Assumption}

\newtheorem{remark}{Remark}

\newtheorem{thm}{Theorem}[section]

\newtheorem{lem}[thm]{Lemma}
\newtheorem{prop}[thm]{Proposition}

\def \a{{\alpha}}

\def \d{{\mathrm{d}}}

\newcommand{\Om}{\Omega}

\newcommand{\cF}{\mathcal{F}}

\newcommand{\PP}{\mathbb{P}}

\newcommand{\E}{\mathbb{E}}

\newcommand{\bean}{\begin{eqnarray*}}
	\newcommand{\eean}{\end{eqnarray*}}

\newcommand{\Var}{\text{Var}}

\newcounter{bean}
\newcommand{\benuma}{\setlength{\labelwidth}{.25in}
	
	\begin{list}
		{(\alph{bean})}{\usecounter{bean}}}
	\newcommand{\eenuma}{\end{list}}

\begin{document}
	
	\title[Multilevel Monte Carlo for SVEs with fractional kernel]{Central limit theorem of Multilevel Monte Carlo Euler estimators for Stochastic Volterra equations with fractional kernels}
	\author[S. Liu]{Shanqi Liu}
	\address{School of Mathematical Science, Nanjing Normal University, Nanjing 210023, China}
	\email{shanqiliumath@126.com}
	
	\author[Y. Hu]{Yaozhong Hu}
	\address{Department of Mathematical and Statistical Sciences, University of Alberta, Edmonton T6G 2G1, Canada}
	\email{yaozhong@ualberta.ca}
	
	\author[H. Gao]{Hongjun Gao}
	\address{School of Mathematics, Southeast University, Nanjing 211189, China}
	\email{hjgao@seu.edu.cn (Corresponding author)}
	\vspace{-2cm}
	\maketitle
	\vspace{-1cm}
	\begin{abstract}
		  This paper is devoted to proving a (Lindeberg-Feller type ) central limit theorem    for the multilevel Monte Carlo estimator associated with the Euler discretization scheme  for the stochastic Volterra equations with fractional kernels $K(u)=u^{H-\frac{1}{2}}/\Gamma(H+1/2), H\in (0,1/2]$.
		
		\smallskip
		
	\end{abstract}
\tableofcontents
	\section{Introduction}
	Let $(\Omega,\mathcal{F}, \mathbb{P})$ to be a  complete   probability space 
	with a filtration  $\mathcal{F}_t$, which  is a nondecreasing family of sub-$\sigma$ fields of $\mathcal{F}$ satisfying the usual conditions. 
	Motivated by an attempt to solve the fractional order 
	stochastic differential equation,   the following stochastic Volterra equation (SVE for short) on $d$-dimensional Euclidean space $\mathbb{R}^{d}$ has been studied recently  (e.g.\cite{kilbas}):  
	\begin{align}\label{Volterra eq}
		X_{t}=X_0+\int_{0}^{t}K(t-s)b(X_{s})\d s+\int_{0}^{t}K(t-s)\sigma(X_{s})\d W_s,\quad 0\le t\le T\,, 
	\end{align}
	where $X_0\in\mathbb{R}^d,K(u)=u^{H-\frac{1}{2}}/\Gamma(H+1/2), H\in (0,1/2]$, $
	W=(W_t\,, t\ge 0)$
	is an $\mathrm{m}$-dimensional Brownian motion defined on  
	$(\Om, \cF, \PP)$, and the coefficients $b:\mathbb{R}^d\to\mathbb{R}^d$, $\sigma:\mathbb{R}^d\to\mathbb{R}^d\times \mathbb{R}^q$ are  assumed to satisfy some conditions.  For the purpose of this work, we assume that
	they satisfy the following conditions.  
%
	 
	 \begin{assumption}\label{assumption 2.1}
	 	\quad There exists a positive constant $L_1$ such that 
	 	$$
	 	|b(x)-b(y) |+ |\sigma(x)-\sigma(y) |\leq L_1|x-y|\quad \hbox{for any $x,y\in\mathbb{R}^d$}\,.			
	 	$$
	 \end{assumption}                        
	Since the explicit solution of SVEs with singular kernels  is  rarely known  one has to rely on numerical methods for simulations of the solutions  of these equations. Various time-discrete numerical approximation schemes for \eqref{Volterra eq} have been investigated  in recent years. We recall
	one that we are going to carry out some further study.  
	
	One  elementary and yet  widely used numerical method for \eqref{Volterra eq} is the following Euler-Maruyama    scheme, studied by Zhang    \cite{ZX1},    Li et al.\cite{LHH2} and Richard et al.\cite{RTY}.  To describe this scheme concisely    let us   we consider only uniform partitions of the interval $[0,T]$: $\pi_n: 0=t_0<t_1<\cdots<t_{[nT]+1}=T\wedge \frac{[nT]+1}{n}$,  where $ t_i=\frac{i}{n},i=0,1,\cdots,[nT]$ (we shall consider this type of partitions throughout the paper). For every positive integer $n$,  the Euler-Maruyama    approximation   $X ^{  e, n}$  
	is given by 
	\begin{align}\label{Euler scheme}
		X ^{  e, n}_t  =&X_0+\int_{0}^{t}K(t-\frac{[ns]}{n})b(X ^{  e, n} _{\frac{[ns]}{n}})\d s+\int_{0}^{t}K(t-\frac{[ns]}{n})\sigma(X ^{  e, n}_{\frac{[ns]}{n}})\d W_s,
	\end{align}
	where $[a]$ denotes the largest integer which is less than or equal to $a$.  	      It was proved in  \cite{LHH2} and \cite{RTY}
	that there exists a positive constant $C$, independent  of  $n$,   such that 
	\begin{equation}
		\sup_{t\in [0,T]}\E[|X_t- X^{e,n}_t|^p]^{1/p}\leq C n^{-H}. \label{e.1.3} 
	\end{equation} 
	In other word, the  Euler-Maruyama    scheme   for  \eqref{Volterra eq} 
	has  a rate of convergence $H$. 
	A variant of the  above Euler-Maruyama scheme \eqref{Euler scheme} is  the following discretized version of \eqref{Volterra eq}:
	\begin{align}\label{Euler type scheme}
		X^{n}_t=&X_0+\int_{0}^{t}K(t-s)b(X^{n}_{\frac{[ns]}{n}})\d s+\int_{0}^{t}K(t-s)\sigma(X^{n}_{\frac{[ns]}{n}})\d W_s\,. 
	\end{align}
	For this scheme it would also  expected that \eqref{e.1.3} still holds true, namely, this variant has also convergence rate $H$. 
	In fact, it  does and we have  
	\begin{equation}
		\sup_{t\in [0,T]}\E[|X_t- X^{n}_t|^p]^{1/p}\leq C n^{-H}. \nonumber
	\end{equation}   
	Moreover,  it is proved   in  \cite{FU} and \cite{NS}, independently,  that 
	if we denote $U^n=n^{H}[X^{n}_t-X_t]$,
	then    as $n$ tends to infinity the process $U^n$ converges stably in law to the solution $U=(U^1,\cdots,U^d)^{\top}$ of  the following linear  SVE:
	\begin{align}
		U_t^i= \sum_{k=1}^{d}&\int_0^t K(t-s)  \partial_k b^i(X_s)U^k_s \mathrm{d} s+\sum_{j=1}^q \sum_{k=1}^d \int_0^t K(t-s) \partial_k \sigma_j^i(X_s)U^k_s\d W^j_s\nonumber \\
		& +\frac{1}{\sqrt{\Gamma(2H+2)\sin \pi H}}\sum_{j=1}^q \sum_{k=1}^d\sum_{l=1}^q \int_0^t K(t-s) \partial_k \sigma_j^i\left(X_s\right)\sigma^k_l(X_{s})\d B^{l,j}_s,
		\label{e.1.5} 	\end{align}
	where $B$ is an $q^2$-dimensional Brownian motion independent of $W$. 
	
	Let us also mention that in the case of classical  stochastic differential equation, namely, when the Hurst parameter  $H=1/2$ the asymptotic behaviors of the
	normalized error process for Euler-Maruyama  scheme  have
	been studied in \cite{Ja2,KP0} and the references therein. We also mention that still in the absence of the Volterra kernel, when  the driven Brownian motion  is replaced by the fractional Brownian motion, there are also a number of results on the rate of convergence and asymptotic error distributions of Euler-type
	numerical schemes, we refer to \cite{DNT,FR,HLN,HLN2,LT} and the references therein.
	
With the rapid development of computing power, people are more and more relying
on numerical computations to obtain the quantities people are interested. This boosts researchers to use 	multilevel Monte Carlo (MLMC) method  to do simulations.   In the case of classical  stochastic differential equation driven by standard Brownian motion, 
	   many  researchers  have been interested in the   central limit theorems associated with  the MLMC error  (e.g.  recent works \cite{BKN,BK,DL,GLP,HK,KL}). Let us mention one result.   Giles \cite{Giles},   studied  the scheme defined by \eqref{Euler type scheme}  
	 and     it is proved  (let us emphasize that it is in  the case $H=1/2$)  in  Ben Alaya and Kebaier \cite{BK} that for any fixed positive integer $m$, as $n$ tends to infinity the  multilevel error process $\sqrt{n} (X^{mn}-X^n)$ converges stably in law to the solution $U^m =(U^{1,m} ,\cdots,U^{d,m} )^{\top}$ of  the following linear  stochastic differential equation with $m\in\mathbb{N}\backslash \{0,1\} $: 
	 \begin{align}
	 	U_t^{i,m} &= \sum_{k=1}^{d}\int_0^t\partial_k b^i(X_s)U^{k,m} _s \mathrm{d} s+\sum_{j=1}^q \sum_{k=1}^d \int_0^t \partial_k \sigma_j^i(X_s) U^{k,m}_s\d W^j_s\nonumber\\
		&\quad+\sqrt{\frac{m-1}{2m}}\sum_{j=1}^q\sum_{k=1}^{d}\sum_{l=1}^q\int_0^{t}  \partial_k \sigma_j^i\left(X_s\right) \sigma^{k}_{l}(X_{s}) \mathrm{d} B_s^{l,j},\nonumber
	 \end{align}
where $B$ is an $q^2$-dimensional Brownian motion independent of $W$.  Notice that letting  $m$ tend to infinity, we obtain formally  the Jacod and Protter’s result \cite{Ja2}
	 (e.g.  the result of  \eqref{e.1.5} in the case $H=1/2$).  
	 
	 However,   the central limit theorem of MLMC errors  for SVEs with singular kernels 
	 and/or driven by   fractional Brownian motions has not been addressed until present.  This work  aims to fill this gap.  First,  we shall obtain the central limit
of $n^H (X^{mn}-X^n)$  for   SVEs with singular kernels 
in Theorem 
	  \ref{t.main} (e.g. Equation \eqref{limit Volterra eq}).  Comparing to the
	  works \cite{FU,Ja2,LHG}, here we transfer the classical error results to the multilevel error case. Moreover,  we need to point out that in comparison to \cite{BK}, we now work on the singular kernels framework. Theorem \ref{t.main}  is an extension and improvement of Theorem 3 of \cite{BK} (see Remark \ref{rem 2.1} for more details).

 Motivated by the work  \cite{BK}  it has some advantage to use the following 
 MLMC to approximate $\E(f(X_T))$:   
\begin{align}\label{def of Q}
		Q_n=\frac{1}{N_0}\sum_{k=1}^{N_0}f(X^1_{T,k})+\sum_{l=1}^{L}\frac{1}{N_l}\sum_{k=1}^{N_l}\big(f(X^{l,m^l}_{T,k})-f(X^{l,m^{l-1}}_{T,k})\big).
	\end{align} 	  
The second main result (Theorem \ref{t.3.4})   of this work is to obtain  a  central limit theorem of Lindeberg-Feller type for the above approximation  for our general SVE \eqref{Volterra eq}. 

The challenge of these two  new tasks lies in the new correlation structures introduced by the singular kernel  which  lacks  independent increment property.   
 On one hand, we need to analysis the interaction among  singular kernels 
 of  different step sizes. On the other hand, to obtain the desired limiting process,  we need to prove various precise limit theorems for fractional integral (see Appendix \ref{limit theorems}). The relationship between different limit theorems is given in Remark \ref{rem 3.2}. It is worth mentioning that although the ideas and  tools developed in \cite{FU} and \cite{LHG} are helpful  to achieve our goals,   the  results  there are not applicable 
 in  proving   our main theorem  \ref{t.main}. We need completely new technical treatments simply because the multilevel
 Monte Carlo Euler method involves the error process $X^{mn}-X^{n}$ instead of  $X^n-X$.


	Here is the structure of  the  paper.   In Section \ref{Error analysis of U}, we analyze  the error process $U^{mn,n}$ and prove a stable law convergence theorem in Theorem \ref{t.main}. In Section \ref{CLT}, we prove the  Lindeberg-Feller central limit theorem. There are huge amount of technical computations.
To ease the reading of the paper we postpone the technical computations
to the appendix.  
                    
		\section{Error analysis of $U^{mn,n}$}\label{Error analysis of U}
In this section, we shall prove a stable law convergence theorem, for the Euler scheme error on two consecutive levels $m^{\ell-1}$ and $m^{\ell}$, of the type obtained in Fukasawa and Ugai \cite{FU}. Our result in Theorem \ref{t.main} below is an innovative contribution on the Euler scheme error that is different and more tricky than the original work by Fukasawa and Ugai \cite{FU} since it involves the error process $X^{\ell, m^{\ell}}-X^{\ell, m^{\ell-1}}$ rather than $X^{\ell, m^{\ell}}-X$. Note that the study of the error $X^{\ell, m^{\ell}}-X^{\ell, m^{\ell-1}}$ as $\ell \rightarrow \infty$ can be reduced to the study of the error $X^{m n}-X^n$ as $n \rightarrow \infty$ where $X^{m n}$ and $X^n$ stand for the Euler schemes with time steps $1 /m n$ and $1/ n$ constructed on the same Brownian path.
	
	Consider the normalized error process $U^{mn,n}=(U_t^{mn,n})_{0\leq t\leq T}$ defined by
	$$U^{mn,n}_t:=n^{H}\big(X_t^{mn}-X^{n}_t\big),\quad t\in [0,T].$$
An application of the  Taylor expansion gives  
	\begin{align}
		U^{mn,n}_t&=n^{H}\int_{0}^{t}K(t-s)\big(b(X^{mn}_\frac{[mns]}{mn})-b(X^n_\frac{[ns]}{n})\big)\d s\nonumber\\&\quad+n^{H}\int_{0}^{t}K(t-s)\big(\sigma(X^{mn}_\frac{[mns]}{mn})-\sigma(X^n_\frac{[ns]}{n})\big)\d W_s\nonumber\\&=n^{H}\int_{0}^{t}K(t-s)\nabla b(X^n_{\frac{[ns]}{n}})\big(X^{mn}_\frac{[mns]}{mn}-X^{n}_\frac{[ns]}{n}\big)\d s\nonumber\\&\quad+n^{H}\int_{0}^{t}K(t-s)\nabla\sigma(X^n_{\frac{[ns]}{n}})\big(X^{mn}_\frac{[mns]}{mn}-X^n_\frac{[ns]}{n}\big)\d W_s+\mathcal{R}_t^{mn,n}\nonumber\\&=\int_{0}^{t}K(t-s)\nabla b(X^n_{\frac{[ns]}{n}})U_s^{mn,n}\d s+\int_{0}^{t}K(t-s)\nabla\sigma(X^n_{\frac{[ns]}{n}})U_s^{mn,n}\d W_s\nonumber\\&\quad+\int_{0}^{t}K(t-s)\nabla b(X^n_{\frac{[ns]}{n}})n^{H}\big(X^{n}_s-X^{n}_\frac{[ns]}{n}+X^{mn}_\frac{[mns]}{mn}-X^{mn}_s\big)\d s\nonumber\\&\quad+\int_{0}^{t}K(t-s)\nabla\sigma(X^n_{\frac{[ns]}{n}})n^{H}\big( X^n_s-X^{n}_\frac{[ns]}{n}+X^{mn}_\frac{[mns]}{mn}-X^{mn}_s\big)\d W_s+\mathcal{R}_t^{mn,n},\nonumber
	\end{align}
where
	\begin{align}
	\mathcal{R}^{mn,n}_{t}&=n^{H}\int_{0}^{t}K(t-s)\int_{0}^{1}\big(\nabla b(X^n_{\frac{[ns]}{n}}+r\big(X^{mn}_{\frac{[mns]}{mn}}-X^n_{\frac{[ns]}{n}}\big))-\nabla b(X^n_{\frac{[ns]}{n}})\big)\d r\big(X^{mn}_{\frac{[mns]}{mn}}-X^n_{\frac{[ns]}{n}}\big)\d s\nonumber \\
	&\quad+n^{H}\int_{0}^{t}K(t-s)\int_{0}^{1}\big(\nabla \sigma(X^n_{\frac{[ns]}{n}}+r\big(X^{mn}_{\frac{[mns]}{mn}}-X^n_{\frac{[ns]}{n}}\big))-\nabla \sigma(X^n_{\frac{[ns]}{n}})\big)\d r\big(X^{mn}_{\frac{[mns]}{mn}}-X^n_{\frac{[ns]}{n}}\big)\d s .\nonumber
\end{align}
	Let $V^{mn,n}_{\cdot}=\{V^{mn,n,k,j}\}_{ 1\leq k\leq d, 1\leq j\leq q} $ be  defined by 
	\begin{align}\label{def of V}
		V^{mn,n,k,j}_{\cdot}=n^{H}\int_{0}^{\cdot}\big(X^{n,k}_s-X^{n,k}_\frac{[ns]}{n}+X^{mn,k}_\frac{[mns]}{mn}-X^{mn,k}_s\big)\d W_s^j .
	\end{align}
	We first analyze  the limit of $\langle V^{mn,n,k_1,j}, V^{mn,n,k_2,j} \rangle_t$.  Recall that for $\kappa=n$ or $mn$
	\begin{align}
		X^{\kappa,k}_s-X^{\kappa,k}_{\frac{[\kappa s]}{\kappa}}&=\int_{0}^{\frac{[\kappa s]}{\kappa}}\Big(K(s-u)-K(\frac{[\kappa s]}{\kappa}-u)\Big)b^k(X^{\kappa}_{\frac{[\kappa u]}{\kappa}})\d u\nonumber\\&\quad+b^k(X^{\kappa}_{\frac{[\kappa s]}{\kappa}})\int_{\frac{[\kappa s]}{\kappa}}^{s}K(s-u)\d u\nonumber\\&\quad+\sum_{j=1}^{q}\int_{0}^{\frac{[\kappa s]}{\kappa}}\Big(K(s-u)-K(\frac{[\kappa s]}{\kappa}-u)\Big)\sigma^k_j(X^{\kappa}_{\frac{[\kappa u]}{\kappa}})\d W^j_u\nonumber\\&\quad+\sum_{j=1}^q\sigma^k_j(X^{\kappa}_{\frac{[\kappa s]}{\kappa}})\int_{\frac{[\kappa s]}{\kappa}}^{s}K(s-u)\d W^j_u.\nonumber
	\end{align}
	In order to rewrite  $X^{n,k}_s-X^{n,k}_\frac{[ns]}{n}+X^{mn,k}_\frac{[mns]}{mn}-X^{mn,k}_s$ as a treatable form, we introduce some more notations. Let
	\begin{align}
		\Delta K(n,s,u)&:=K(s-u)-K(\frac{[n s]}{n}-u),\nonumber\\\Delta K(mn,s,u)&:=K(s-u)-K(\frac{[mn s]}{mn}-u),\nonumber\\\Delta K(mn,n,s,u)&:=K(\frac{[mn s]}{mn}-u)-K(\frac{[ns]}{n}-u)\nonumber.
	\end{align}
	Then we have
	\begin{align}\label{def of Double X n m}
		&\quad X^{n,k}_s-X^{n,k}_\frac{[ns]}{n}+X^{mn,k}_\frac{[mns]}{mn}-X^{mn,k}_s\nonumber\\&=\int_{0}^{\frac{[n s]}{n}}\Delta K(n,s,u)b^k(X^{n}_{\frac{[nu]}{n}})\d u+b^k(X^{n}_{\frac{[n s]}{n}})\int_{\frac{[n s]}{n}}^{s}K(s-u)\d u\nonumber\\&\quad+\sum_{j=1}^q\int_{0}^{\frac{[n s]}{n}} \Delta K(n,s,u)\sigma^k_j(X^{n}_{\frac{[nu]}{n}})\d W^j_u+\sum_{j=1}^q\sigma^k_j(X^{n}_{\frac{[ns]}{n}})\int_{\frac{[n s]}{n}}^{s}K(s-u)\d W^j_u\nonumber\\&\quad-\int_{0}^{\frac{[mn s]}{mn}}\Delta K(mn,s,u)b^k(X^{mn}_{\frac{[mn u]}{mn}})\d u-b^k(X^{mn}_{\frac{[mn s]}{mn}})\int_{\frac{[mn s]}{mn}}^{s}K(s-u)\d u\nonumber\\&\quad-\sum_{j=1}^q\int_{0}^{\frac{[mn s]}{mn}}\Delta K(mn,s,u)\sigma^k_j(X^{mn}_{\frac{[mn u]}{mn}})\d W^j_u-\sum_{j=1}^q\sigma^k_j(X^{mn}_{\frac{[mn s]}{mn}})\int_{\frac{[mn s]}{mn}}^{s}K(s-u)\d W^j_u\nonumber\\
		&:= \mathcal{A}^{mn,n,k}_{1,s}+\mathcal{A}^{mn,n,k}_{2,s},
	\end{align}
	where
	\begin{align}
		\mathcal{A}^{mn,n,k}_{1,s}&:=\int_{0}^{\frac{[n s]}{n}}\Delta K(n,s,u)\big(b^k(X^{n}_{\frac{[n u]}{n}})-b^k(X^{mn}_{\frac{[mn u]}{mn}})\big)\d u\nonumber\\&\quad+\int_{0}^{\frac{[n s]}{n}}\Delta K(mn,n,s,u)b^k(X^{mn}_{\frac{[mn u]}{mn}})\d u-\int_{\frac{[ns]}{n}}^{\frac{[mn s]}{mn}}\Delta K(mn,s,u)b^k(X^{mn}_{\frac{[mn u]}{mn}})\d u\nonumber\\&\quad+b^k(X^{n}_{\frac{[n s]}{n}})\int_{\frac{[n s]}{n}}^{\frac{[mn s]}{mn}} K(s-u)\d u+\big(b^k(X^{n}_{\frac{[n s]}{n}})-b^k(X^{mn}_{\frac{[mn s]}{mn}})\big)\int_{\frac{[mns]}{mn}}^{s}K(s-u)\d u,\nonumber\\\mathcal{A}^{mn,n,k}_{2,s}&:=\sum_{j=1}^q\int_{0}^{\frac{[n s]}{n}} \Delta K(n,s,u)\sigma^k_j(X^{n}_{\frac{[nu]}{n}})\d W^j_u+\sum_{j=1}^q\sigma^k_j(X^{n}_{\frac{[ns]}{n}})\int_{\frac{[n s]}{n}}^{s}K(s-u)\d W^j_u\nonumber\\&\quad-\sum_{j=1}^q\int_{0}^{\frac{[mn s]}{mn}}\Delta K(mn,s,u)\sigma^k_j(X^{mn}_{\frac{[mn u]}{mn}})\d W^j_u-\sum_{j=1}^q\sigma^k_j(X^{mn}_{\frac{[mn s]}{mn}})\int_{\frac{[mn s]}{mn}}^{s}K(s-u)\d W^j_u.\nonumber
	\end{align}
	The quadratic variation of $V$ is computed easily:  
	\begin{align}\label{cov of V}
		&\langle V^{mn,n,k_1,j}, V^{mn,n,k_2,j} \rangle_t\nonumber\\&=n^{2H}\int_{0}^{t}\big(X^{n,k_1}_s-X^{n,k_1}_\frac{[ns]}{n}+X^{mn,k_1}_\frac{[mns]}{mn}-X^{mn,k_1}_s\big)\big(X^{n,k_2}_s-X^{n,k_2}_\frac{[ns]}{n}+X^{mn,k_2}_\frac{[mns]}{mn}-X^{mn,k_2}_s\big)\d s\nonumber\\&=\int_{0}^{t}n^{2H}\mathcal{A}^{mn,n,k_1}_{1,s}\mathcal{A}^{mn,n,k_2}_{1,s}\d s+\int_{0}^{t}n^{2H}\mathcal{A}^{mn,n,k_1}_{1,s}\mathcal{A}^{mn,n,k_2}_{2,s}\d s\nonumber\\&\quad+\int_{0}^{t}n^{2H}\mathcal{A}^{mn,n,k_1}_{2,s}\mathcal{A}^{mn,n,k_2}_{1,s}\d s+\int_{0}^{t}n^{2H}\mathcal{A}^{mn,n,k_1}_{2,s}\mathcal{A}^{mn,n,k_2}_{2,s}\d s.
	\end{align}
Our main result is
the following theorem. 
\begin{theorem}\label{t.main} Assume that the derivatives of the functions $b$ and $\sigma$  are bounded (hence Assumptions \ref{assumption 2.1} holds). 
	Let    $m\in\mathbb{N}\backslash \{0,1\}=\{2, 3, \cdots\}$ and  $\epsilon \in(0, H)$. Then the process $U^{mn,n}=n^{H}(X^{mn}-X^n)$  converges stably in law in $\mathcal{C}_0^{H-\epsilon}$ 
	to the   process $U=\left(U^1, \ldots, U^d\right)^{\top}$, 
which is the unique solution 	of the following  linear stochastic Volterra equation of random coefficients:  
	\begin{align}\label{limit Volterra eq}
		U_t^i=&\sum_{k=1}^{d}\int_0^t K(t-s) \partial_k b^i(X_s)U^k_s \mathrm{d} s+\sum_{j=1}^m \sum_{k=1}^d \int_0^t K(t-s) \partial_k \sigma_j^i(X_s)U^k_s\d W^j_s\nonumber \\
		& +
		\sqrt{\frac{g^H_m}{\Gamma(2H+1)\sin\pi H}}\sum_{j=1}^q\sum_{k=1}^{d}\sum_{l=1}^q\int_0^{t} K(t-s)\partial_k \sigma_j^i\left(X_s\right) \sigma^{k}_{l}(X_{s}) \mathrm{d} B_s^{l,j}\,, 
		\\
		&\qquad\quad  t \in[0, T], i=1, \ldots, d,\nonumber
	\end{align}
	where $g_m^H:=\sum_{j=0}^{m-1}\frac{j^{2H}}{m^{2H+1}}$, $\mathcal{C}_0^{\lambda}$  denotes the set of all  $\mathbb{R}^d$-valued $\lambda$-H$\mathrm{\ddot{o}}$lder continuous functions on $[0,T]$ vanishing at $t=0$, and $B$ is an $q^2$-dimensional standard Brownian motion, independent of the original Brownian motion $W$. 
\end{theorem}
\begin{remark}\label{rem 2.1}
	Here are  some comments for the above theorem. When $H=1/2$,
\[
\sum_{j=0}^{m-1} j= m(m-1)/2\,,\quad  g_m^H=(m-1)/(2m)\,. 
\]
For general $H>0$,   it is easy to see that 
\[
\int_0^{m-1} x^{2H} \d x \le\sum_{j=0}^{m-1} j^{2H}\le 	\int_0^{m-1} x^{2H} \d x\,,
\quad \hbox{or}\quad 
\frac{(m-1)^{2H+1}}{ m^{2H+1}}\le g_m^H\le \frac{m^{2H+1}-1}{ m^{2H+1}}\,. 
\]
Thus 
\[
g_m^H\to 1\quad \hbox{as $m\to \infty$}\,.
\]
Hence we have 
	\begin{itemize}
		\item[(i)] 
		If we restrict $H=1/2$, we recover the Ben Alaya and Kebaier's   result (\cite{BK}).
		\item[(ii)] 
		If we restrict $H=1/2$ and further let formally $m$ tend to infinity, we recover the Jacod and Protter's   result (\cite{Ja2}).
	
		\item[(iii)]  If we just let formally $m$ tend to infinity, we recover the Fukasawa and Ugai's  result (\cite{FU}). 
	\end{itemize} 
\begin{center}
	\renewcommand{\arraystretch}{2}
	\begin{tabular}{|c|@{\extracolsep{2em}}c|}
		\hline
		Case& Constant of limit   equation  \\ \hline
		$H\in(0,1/2],m\in\mathbb{N}\backslash \{0,1\} $ & $\sqrt{\frac{g^H_m}{\Gamma(2H+1)\sin\pi H}}$ \\ \hline
		$H\in(0,1/2]$, $m\to\infty$ & $\frac{1}{\sqrt{\Gamma(2H+2)\sin \pi H}}$ \\\hline
		$H=1/2,m\in\mathbb{N}\backslash \{0,1\} $ & $\sqrt{\frac{m-1}{2m}}$ \\ \hline 
		$H=1/2$, $m\to\infty$ & $\frac{1}{\sqrt{2}}$ \\ \hline
	\end{tabular}
\end{center}
\end{remark}
To prove the theorem we need some lemmas. 
	\begin{lem}\label{est of A}
		For any $k=1,\cdots,d$ and any $m\in\mathbb{N}\backslash \{0,1\} $,
		\begin{itemize}
			\item[(i)] 
			$ \lim_{n\to\infty}\sup_{s\in[0,T]}\|n^{H}\mathcal{A}^{mn,n,k}_{1,s}\|_{L^2}=0$.
			\item[(ii)] 
			$\sup_{n \ge 1} \sup_{s\in[0,T]}\|n^{H}\mathcal{A}^{mn,n,k}_{2,s}\|_{L^2}<\infty.$
			\item[(iii)]  For any $k_1,k_2$ and $t\in [0,T]$  we have as ${n \rightarrow \infty}$
			\begin{align}
				n^{2H}\int_{0}^{t}\mathcal{A}^{mn,n,k_1}_{2,s}\mathcal{A}^{mn,n,k_2}_{2,s}\d s
				\stackrel{L^2}{\rightarrow } \sum_{j=1}^{q}\frac{g^H_m}{\Gamma(2H+1)\sin\pi H}\int_{0}^{t}\sigma^{k_1}_j(X_{s})\sigma^{k_2}_{j}(X_{s})\d s . \nonumber
			\end{align}
		\end{itemize} 	
	\end{lem}
\begin{proof}[Proof of Lemma \ref{est of A}-(i)]
	By the properties of fractional kernels, Assumptions \ref{assumption 2.1}, Minkowski integral inequality,Lemma \ref{bound of X n} and Lemma \ref{est-1}, we have
	\begin{align}
		\E[|\mathcal{A}^{mn,n,k}_{1,s}|^2]^{1/2}&\leq \E\Big[\Big| \int_{0}^{\frac{[n s]}{n}}\Delta K(n,s,u)\big(b^k(X^{n}_{\frac{[n u]}{n}})-b^k(X^{mn}_{\frac{[mn u]}{mn}})\big)\d u\Big|^2\Big]^{1/2}\nonumber\\&\quad+\E\Big[\Big|\int_{0}^{\frac{[n s]}{n}}\Delta K(mn,n,s,u)b^k(X^{mn}_{\frac{[mn u]}{mn}})\d u\Big|^2\Big]^{1/2}\nonumber\\&\quad+\E\Big[\Big|-\int_{\frac{[ns]}{n}}^{\frac{[mn s]}{mn}}\Delta K(mn,s,u)b^k(X^{mn}_{\frac{[mn u]}{mn}})\d u\Big|^2\Big]^{1/2}\nonumber\\&\quad+\E\Big[\Big| b^k(X^{n}_{\frac{[n s]}{n}})\int_{\frac{[n s]}{n}}^{\frac{[mn s]}{mn}} K(s-u)\d u\Big|^2\Big]^{1/2}\nonumber\\&\quad+\E\Big[\Big|\big(b^k(X^{n}_{\frac{[n s]}{n}})-b^k(X^{mn}_{\frac{[mn s]}{mn}})\big)\int_{\frac{[mns]}{mn}}^{s}K(s-u)\d u\Big|^2\Big]^{1/2}\nonumber\\&\leq \int_{0}^{\frac{[n s]}{n}}\Delta K(n,s,u)\Big\|b^k(X^{n}_{\frac{[n u]}{n}})-b^k(X^{mn}_{\frac{[mn u]}{mn}})\Big\|_{L^2}\d u\nonumber\\&\quad+\int_{0}^{\frac{[n s]}{n}}\Delta K(mn,n,s,u)\big\|b^k(X^{mn}_{\frac{[mn u]}{mn}})\big\|_{L^2}\d u\nonumber\\&\quad+\int_{\frac{[ns]}{n}}^{\frac{[mn s]}{mn}}\Delta K(mn,s,u)\big\|b^k(X^{mn}_{\frac{[mn u]}{mn}})\big\|_{L^2}\d u\nonumber\\&\quad+\big\| b^k(X^{n}_{\frac{[n s]}{n}})\big\|_{L^2}\int_{\frac{[n s]}{n}}^{\frac{[mn s]}{mn}} K(s-u)\d u\nonumber\\&\quad+\Big\|b^k(X^{n}_{\frac{[n s]}{n}})-b^k(X^{mn}_{\frac{[mn s]}{mn}})\Big\|_{L^2}\int_{\frac{[mns]}{mn}}^{s}K(s-u)\d u\nonumber\\&\leq C\Big[ \int_{0}^{\frac{[n s]}{n}}\Delta K(n,s,u)\d u+\int_{0}^{\frac{[n s]}{n}}\Delta K(mn,n,s,u)\d u+\int_{\frac{[ns]}{n}}^{\frac{[mn s]}{mn}}\Delta K(mn,s,u)\d u\nonumber\\&\quad+\int_{\frac{[n s]}{n}}^{\frac{[mn s]}{mn}} K(s-u)\d u+\int_{\frac{[mns]}{mn}}^{s}K(s-u)\d u\Big]\nonumber
		\\&\leq Cn^{H-1/2}.\nonumber
	\end{align}
	Hence $\lim_{n\to\infty}\sup_{s\in[0,T]}\|n^{H}\mathcal{A}^{mn,n,k}_{1,s}\|_{L^2}=0$.
\end{proof}
	\begin{proof}[Proof of Lemma \ref{est of A}-(ii)]
		By the properties of fractional kernels, Assumptions \ref{assumption 2.1}, Lemma \ref{basic lemma}-(iv), Lemma \ref{rate of mn,n}  and the Burkholder-Davis-Gundy (BDG) inequality, we have
		\begin{align}
			\E[|n^{H}\mathcal{A}^{mn,n,k}_{2,s}|^2]&\leq Cn^{2H}\E\Big[\Big| \int_{0}^{\frac{[n s]}{n}}\Delta K(n,s,u)\sigma_j^k(X^{n}_{\frac{[n u]}{n}})\d W^j_u\Big|^2\Big]\nonumber\\&\quad+C n^{2H}\E\Big[\Big|\sigma_j^k(X^{n}_{\frac{[n s]}{n}})\int_{\frac{[mns]}{mn}}^{s}K(s-u)\d W^j_u\Big|^2\Big]\nonumber\\&\quad+Cn^{2H}\E\Big[\Big| \int_{0}^{\frac{[mn s]}{mn}}\Delta K(mn,s,u)\sigma_j^k(X^{n}_{\frac{[mn u]}{mn}})\d W^j_u\Big|^2\Big]\nonumber\\&\quad +C n^{2H}\E\Big[\Big|\sigma_j^k(X^{mn}_{\frac{[mn s]}{mn}})\int_{\frac{[mns]}{mn}}^{s}K(s-u)\d W^j_u\Big|^2\Big]\nonumber\\&\leq Cn^{2H}\Big|s-\frac{[ns]}{n}\Big|^{2H}\sup_{s\in [0,T]}\E\Big[\big|\sigma_j^k(X_{\frac{[ns]}{n}}^{n})\big|^2\Big]\nonumber\\&\quad+ Cn^{2H}\Big|s-\frac{[mns]}{mn}\Big|^{2H}\sup_{s\in [0,T]}\E\Big[\big|\sigma_j^k(X_{\frac{[mns]}{mn}}^{mn})\big|^2\Big]\nonumber\\&\leq Cn^{-2H}.
		\end{align}
		Hence $\lim_{n\to\infty}\sup_{s\in[0,T]}\|n^{H}\mathcal{A}^{mn,n,k}_{2,s}\|_{L^2}<\infty$.
	\end{proof}
	\begin{proof}[Proof of Lemma \ref{est of A}-(iii)] 
		Recall that
		\begin{align}
			\mathcal{A}^{mn,n,k}_{2,s}&:=\sum_{j=1}^q\int_{0}^{\frac{[n s]}{n}} \Delta K(n,s,u)\sigma^k_j(X^{n}_{\frac{[nu]}{n}})\d W^j_u+\sum_{j=1}^q\sigma^k_j(X^{n}_{\frac{[ns]}{n}})\int_{\frac{[n s]}{n}}^{s}K(s-u)\d W^j_u\nonumber\\&\quad-\sum_{j=1}^q\int_{0}^{\frac{[mn s]}{mn}}\Delta K(mn,s,u)\sigma^k_j(X^{mn}_{\frac{[mn u]}{mn}})\d W^j_u-\sum_{j=1}^q\sigma^k_j(X^{mn}_{\frac{[mn s]}{mn}})\int_{\frac{[mn s]}{mn}}^{s}K(s-u)\d W^j_u\,. \nonumber 
		\end{align}
We start with the following decomposition:
		\begin{align}
			&n^{2H}\int_{0}^{t}\mathcal{A}^{mn,n,k_1}_{2,s}\mathcal{A}^{mn,n,k_2}_{2,s}\d s\nonumber\\
			&:=\sum_{j,l=1}^q\Big[\textbf{(1,1)}_{j,l}^{n,k}+\textbf{(1,2)}_{j,l}^{n,k}-\textbf{(1,3)}_{j,l}^{ mn,n,k}-\textbf{(1,4)}_{j,l}^{ mn,n,k}+\textbf{(2,1)}_{j,l}^{n,k}+\textbf{(2,2)}_{j,l}^{n,k}\nonumber\\
			&\quad-\textbf{(2,3)}_{j,l}^{ mn,n,k}-\textbf{(2,4)}_{j,l}^{ mn,n,k}-\textbf{(3,1)}_{j,l}^{ mn,n,k}-\textbf{(3,2)}_{j,l}^{ mn,n,k}+\textbf{(3,3)}_{j,l}^{ mn,n,k}+\textbf{(3,4)}_{j,l}^{ mn,n,k}\nonumber\\
			&\quad-\textbf{(4,1)}_{j,l}^{ mn,n,k}-\textbf{(4,2)}_{j,l}^{ mn,n,k}+\textbf{(4,3)}_{j,l}^{ mn,n,k}+\textbf{(4,4)}_{j,l}^{ mn,n,k}\Big],\nonumber
		\end{align}
		where
		\begin{align}
			\textbf{(1,1)}_{j,l}^{n,k}&=n^{2H}\int_{0}^{t}\Big(\int_{0}^{\frac{[n s]}{n}}\Delta K(n,s,u)\sigma_j^{k_1}(X^{mn}_{\frac{[n u]}{n}})\d W^j_u\Big)\Big(\int_{0}^{\frac{[n s]}{n}}\Delta K(n,s,u)\sigma_l^{k_2}(X^{n}_{\frac{[n u]}{n}})\d W^l_u\Big)\d s,\nonumber\\
			\textbf{(1,2)}_{j,l}^{n,k}&=n^{2H}\int_{0}^{t}\Big(\int_{0}^{\frac{[n s]}{n}}\Delta K(n,s,u)\sigma_j^{k_1}(X^{n}_{\frac{[n u]}{n}})\d W^j_u\Big)\Big(\sigma^{k_2}_j(X^{n}_{\frac{[ns]}{n}})\int_{\frac{[n s]}{n}}^{s}K(s-u)\d W^l_u\Big)\d s,\nonumber\\
			\textbf{(2,1)}_{j,l}^{n,k}&=n^{2H}\int_{0}^{t}\Big(\sigma^{k_1}_j(X^{n}_{\frac{[ns]}{n}})\int_{\frac{[n s]}{n}}^{s}K(s-u)\d W^j_u\Big)\Big(\int_{0}^{\frac{[n s]}{n}}\Delta K(n,s,u)\sigma_l^{k_2}(X^{n}_{\frac{[n u]}{n}})\d W^l_u\Big)\d s,\nonumber\\
			\textbf{(2,2)}_{j,l}^{n,k}&=n^{2H}\int_{0}^{t}\Big(\sigma^{k_1}_j(X^{n}_{\frac{[ns]}{n}})\int_{\frac{[n s]}{n}}^{s}K(s-u)\d W^j_u\Big)\Big(\sigma^{k_2}_j(X^{n}_{\frac{[ns]}{n}})\int_{\frac{[n s]}{n}}^{s}K(s-u)\d W^l_u\Big)\d s,\nonumber\\\textbf{(1,3)}_{j,l}^{mn,n,k}&=n^{2H}\int_{0}^{t}\Big(\int_{0}^{\frac{[n s]}{n}}\Delta K(n,s,u)\sigma_j^{k_1}(X^{n}_{\frac{[n u]}{n}})\d W^j_u\Big)\Big(\int_{0}^{\frac{[mn s]}{mn}}\Delta K(mn,s,u)\sigma^{k_2}_j(X^{mn}_{\frac{[mn u]}{mn}})\d W^l_u\Big)\d s,\nonumber\\
			\textbf{(1,4)}_{j,l}^{mn,n,k}&=n^{2H}\int_{0}^{t}\Big(\int_{0}^{\frac{[n s]}{n}}\Delta K(n,s,u)\sigma_j^{k_1}(X^{n}_{\frac{[n u]}{n}})\d W^j_u\Big)\Big(\sigma^{k_2}_j(X^{mn}_{\frac{[mn s]}{mn}})\int_{\frac{[mn s]}{mn}}^{s}K(s-u)\d W^l_u\Big)\d s,\nonumber\\
			\textbf{(2,3)}_{j,l}^{mn,n,k}&=n^{2H}\int_{0}^{t}\Big(\sigma^{k_1}_j(X^{n}_{\frac{[ns]}{n}})\int_{\frac{[n s]}{n}}^{s}K(s-u)\d W^j_u\Big)\Big(\int_{0}^{\frac{[mn s]}{mn}}\Delta K(mn,s,u)\sigma^{k_2}_j(X^{mn}_{\frac{[mn u]}{mn}})\d W^l_u\Big)\d s,\nonumber\\
			\textbf{(2,4)}_{j,l}^{mn,n,k}&=n^{2H}\int_{0}^{t}\Big(\sigma^{k_1}_j(X^{n}_{\frac{[ns]}{n}})\int_{\frac{[n s]}{n}}^{s}K(s-u)\d W^j_u\Big)\Big(\sigma^{k_2}_j(X^{mn}_{\frac{[mn s]}{mn}})\int_{\frac{[mn s]}{mn}}^{s}K(s-u)\d W^l_u\Big)\d s,\nonumber\\
			\textbf{(3,1)}_{j,l}^{mn,n,k}&=n^{2H}\int_{0}^{t}\Big(\int_{0}^{\frac{[mn s]}{mn}}\Delta K(mn,s,u)\sigma^{k_1}_j(X^{mn}_{\frac{[mn u]}{mn}})\d W^j_u\Big)\Big(\int_{0}^{\frac{[n s]}{n}} \Delta K(n,s,u)\sigma^{k_2}_j(X^{n}_{\frac{[nu]}{n}})\d W^l_u\Big)\d s,\nonumber\\
			\textbf{(3,2)}_{j,l}^{mn,n,k}&=n^{2H}\int_{0}^{t}\Big(\int_{0}^{\frac{[mn s]}{mn}}\Delta K(mn,s,u)\sigma^{k_1}_j(X^{mn}_{\frac{[mn u]}{mn}})\d W^j_u\Big)\Big(\sigma^{k_2}_j(X^{n}_{\frac{[ns]}{n}})\int_{\frac{[n s]}{n}}^{s}K(s-u)\d W^l_u\Big)\d s,\nonumber\\
			\textbf{(3,3)}_{j,l}^{mn,n,k}&=n^{2H}\int_{0}^{t}\Big(\int_{0}^{\frac{[mn s]}{mn}}\Delta K(mn,s,u)\sigma^{k_1}_j(X^{mn}_{\frac{[mn u]}{mn}})\d W^j_u\Big)\Big(\int_{0}^{\frac{[mn s]}{mn}}\Delta K(mn,s,u)\sigma^{k_2}_j(X^{mn}_{\frac{[mn u]}{mn}})\d W^l_u\Big)\d s,\nonumber\\
			\textbf{(3,4)}_{j,l}^{mn,n,k}&=n^{2H}\int_{0}^{t}\Big(\int_{0}^{\frac{[mn s]}{mn}}\Delta K(mn,s,u)\sigma^{k_1}_j(X^{mn}_{\frac{[mn u]}{mn}})\d W^j_u\Big)\Big(\sigma^{k_2}_j(X^{mn}_{\frac{[mn s]}{mn}})\int_{\frac{[mn s]}{mn}}^{s}K(s-u)\d W^l_u\Big)\d s,\nonumber\\
			\textbf{(4,1)}_{j,l}^{mn,n,k}&=n^{2H}\int_{0}^{t}\Big(\sigma^{k_1}_j(X^{mn}_{\frac{[mn s]}{mn}})\int_{\frac{[mn s]}{mn}}^{s}K(s-u)\d W^j_u\Big)\Big(\int_{0}^{\frac{[n s]}{n}} \Delta K(n,s,u)\sigma^{k_2}_j(X^{n}_{\frac{[nu]}{n}})\d W^l_u\Big)\d s,\nonumber\\
			\textbf{(4,2)}_{j,l}^{mn,n,k}&=n^{2H}\int_{0}^{t}\Big(\sigma^{k_1}_j(X^{mn}_{\frac{[mn s]}{mn}})\int_{\frac{[mn s]}{mn}}^{s}K(s-u)\d W^j_u\Big)\Big(\sigma^{k_2}_j(X^{n}_{\frac{[ns]}{n}})\int_{\frac{[n s]}{n}}^{s}K(s-u)\d W^l_u\Big)\d s,\nonumber\\
			\textbf{(4,3)}_{j,l}^{mn,n,k}&=n^{2H}\int_{0}^{t}\Big(\sigma^{k_1}_j(X^{mn}_{\frac{[mn s]}{mn}})\int_{\frac{[mn s]}{mn}}^{s}K(s-u)\d W^j_u\Big)\Big(\int_{0}^{\frac{[mn s]}{mn}}\Delta K(mn,s,u)\sigma^{k_2}_j(X^{mn}_{\frac{[mn u]}{mn}})\d W^l_u\Big)\d s,\nonumber\\
			\textbf{(4,4)}_{j,l}^{mn,n,k}&=n^{2H}\int_{0}^{t}\Big(\sigma^{k_1}_j(X^{mn}_{\frac{[mn s]}{mn}})\int_{\frac{[mn s]}{mn}}^{s}K(s-u)\d W^j_u\Big)\Big(\sigma^{k_2}_j(X^{mn}_{\frac{[mn s]}{mn}})\int_{\frac{[mn s]}{mn}}^{s}K(s-u)\d W^l_u\Big)\d s.\nonumber
		\end{align}
		According to Lemma $4.2$  of \cite{FU}, we have that for any $k_1,k_2$ and $t\in [0,T]$  we have as ${n \rightarrow \infty}$ 
		\[
		\begin{split} 
			\textbf{(1,1)}_{j,l}^{n,k} \rightarrow & \begin{cases}\frac{1}{(2H+1)G}\int_{0}^{\infty}|\mu(r,1)|^2\d r\int_{0}^{t}\sigma^{k_1}_j(X_{s})\sigma^{k_2}_j(X_{s})\d s & \text { if } j=j', \\ 0 & \text { if } j \neq j',\end{cases}\\
					\textbf{(1,2)}_{j,l}^{n,k} \rightarrow & 0 \\
			\textbf{(2,1)}_{j,l}^{n,k}  \rightarrow&0\\ 
			\end{split}
		\]
		and 

\[
		\begin{split} 
			\textbf{(2,2)}_{j,l}^{n,k} \rightarrow & \begin{cases}\frac{1}{2HG(2H+1)}\int_{0}^{t}\sigma^{k_1}_j(X_{s})\sigma^{k_2}_j(X_{s})\d s & \text { if } j=j', \\ 0& \text { if } j \neq j' \end{cases} 
		\end{split}
		\]
		in the sense of $L^2$, where $\mu(r,1)$ can be referred to \eqref{basic types}. In subsections  \ref{(4.2.1)}-\ref{(4.2.6)} below, we will show the following limits of remainder terms in $L^2$ as $n\to \infty$: 
		\[
		\begin{split} 
			\sum_{i=1}^{2}\textbf{(i,3)}_{j,l}^{mn,n,k} \rightarrow & \begin{cases}\frac{1}{2G}\Bigg(\Big[\frac{m^{2H}+1}{(2H+1)m^{2H}}-g_m^H\Big]\int_{0}^{\infty}\mu(r,1)^2\d r+\frac{m^{2H}-1}{2H(2H+1)m^{2H}}-\frac{g_m^H}{2H}\Bigg)\int_{0}^{t}\sigma^{k_1}_j(X_{s})\sigma^{k_2}_j(X_{s})\d s& \text { if } j=j', \\ 0 & \text { if } j \neq j',\end{cases} \\
			\textbf{(1,4)}_{j,l}^{mn,n,k} \rightarrow & 0 \\
			\textbf{(2,4)}_{j,l}^{mn,n,k}  \rightarrow&\begin{cases}\frac{1}{2HG(2H+1)m^{2H}}\int_{0}^{t}\sigma^{k_1}_j(X_{s})\sigma^{k_2}_j(X_{s})\d s & \text { if } j=j', \\ 0 & \text { if } j \neq j',\end{cases} \\ 	
			\end{split}
		\] 
\[
		\begin{split} \sum_{i=1}^{2}\textbf{(3,i)}_{j,l}^{mn,n,k} \rightarrow & \begin{cases}\frac{1}{2G}\Bigg(\Big[\frac{m^{2H}+1}{(2H+1)m^{2H}}-g_m^H\Big]\int_{0}^{\infty}\mu(r,1)^2\d r+\frac{m^{2H}-1}{2H(2H+1)m^{2H}}-\frac{g_m^H}{2H}\Bigg)\int_{0}^{t}\sigma^{k_1}_j(X_{s})\sigma^{k_2}_j(X_{s})\d s& \text { if } j=j', \\ 0 & \text { if } j \neq j',\end{cases}\\
	\end{split}
		\] 
\[
		\begin{split} 		\textbf{(3,3)}_{j,l}^{mn,n,k} \rightarrow & \begin{cases}\frac{1}{(2H+1)m^{2H}G}\int_{0}^{\infty}\mu(r,1)^2\d r\int_{0}^{t}\sigma^{k_1}_j(X_{s})\sigma^{k_2}_j(X_{s})\d s& \text { if } j=j', \\ 0& \text { if } j \neq j'.\end{cases} \\\textbf{(3,4)}_{j,l}^{mn,n,k} \rightarrow & 0\\\textbf{(4,1)}_{j,l}^{mn,n,k} \rightarrow & 0\\
		\end{split}
		\] 
\[
		\begin{split} \textbf{(4,2)}_{j,l}^{mn,n,k} \rightarrow & \begin{cases}\frac{1}{2HG(2H+1)m^{2H}}\int_{0}^{t}\sigma^{k_1}_j(X_{s})\sigma^{k_2}_j(X_{s})\d s & \text { if } j=j', \\ 0& \text { if } j \neq j'.\end{cases}\\\textbf{(4,3)}_{j,l}^{mn,n,k} \rightarrow & 0\\\textbf{(4,4)}_{j,l}^{mn,n,k} \rightarrow & \begin{cases}\frac{1}{2HG(2H+1) m^{2H}}\int_{0}^{t}\sigma^{k_1}_j(X_{s})\sigma^{k_2}_j(X_{s})\d s& \text { if } j=j', \\ 0& \text { if } j \neq j'.\end{cases}
		\end{split}
		\]
		These limits  yield 
		\begin{align}
			n^{4H}&\int_{0}^{t}\mathcal{A}^{mn,n,k_1}_{2,s}\mathcal{A}^{mn,n,k_2}_{2,s}ds\xrightarrow[\text { in } L^2]{n \rightarrow \infty} \sum_{j=1}^{q}\Big[\frac{g^H_m}{G}\int_{0}^{\infty}\mu(r,1)^2\d r+\frac{g^H_m}{2HG}\Big]\int_{0}^{t}\sigma^{k_1}_j(X_{s})\sigma^{k_2}_{j}(X_{s})\d s, \nonumber
		\end{align}
		where
		$g_m^H:=\sum_{j=0}^{m-1}\frac{j^{2H}}{m^{2H+1}}$.  Using the identity
		$$2H\int_{0}^{\infty}|\mu(r,1)|^2\d r+1=\frac{G}{\Gamma(2H)\sin\pi H},$$
		  [we refer to Mishura \cite{Mis}, Theorem 1.3.1 and Lemma A.0.1],    we can rewrite the limit  as 
		\begin{align}
			n^{4H}&\int_{0}^{t}\mathcal{A}^{mn,n,k_1}_{2,s}\mathcal{A}^{mn,n,k_2}_{2,s}ds\xrightarrow[\text { in } L^2]{n \rightarrow \infty} \sum_{j=1}^{q}\frac{g^H_m}{\Gamma(2H+1)\sin\pi H}\int_{0}^{t}\sigma^{k_1}_j(X_{s})\sigma^{k_2}_{j}(X_{s})\d s, \nonumber
		\end{align}
		and the lemma is proved.  		
		 \end{proof}
		 The following integration by parts formula will be frequently used:
		\begin{align}\label{dcp of stochastic integral}
			\Big(\int_{s}^{t}h_1(u)&\d W_u^j\Big)\Big(\int_{s}^{t}h_2(u)\d W_u^{j'}\Big)=\int_{s}^{t}\Big(\int_{s}^{u}h_1(r)\d W_r^j\Big)h_2(u)\d W_u^{j'}\nonumber\\&\quad+\int_{s}^{t}\Big(\int_{s}^{u}h_2(r)\d W_r^{j'}\Big)h_1(u)\d W_u^{j}+\int_{s}^{t}h_1(u)h_2(u)\d \langle W^{j}, W^{j'} \rangle_u,
		\end{align}
		for any progressively measurable square integrable  processes  $h_1,h_2 $.

	\begin{lem}\label{est of V-2}
		For all $t\in [0,T], m\in\mathbb{N}\backslash \{0,1\} , k\in\{1,\cdots,d\}, 1\leq j\leq q,$   we have 
		$$\langle V^{mn,n,k,j}, W^{j} \rangle_t\xrightarrow { L^1} 0.$$
	\end{lem}
	
	We delay the proof of this lemma to Appendix. 
	
\begin{lem}\label{charact of V}
	The process $V^{mn,n}$ converges stably in law in $\mathcal{C}_0$
	to the following  continuous process
	$$
	V^{k, j}=\sum_{l=1}^q\sqrt{\frac{g^H_m}{\Gamma(2H+1)\sin\pi H}}\int_{0}^{t}\sigma^{k}_l(X_{s})\d B_s^{l,j},
	$$
	where $\mathcal{C}_0$ denotes the set of all  $\mathbb{R}^d$-valued continuous functions on $[0,T]$ vanishing at $t=0$, $B$ is $q^2$-dimensional standard Brownian motion, independent of the original Brownian  motion $W$. 
\end{lem}
\begin{proof}
	By Lemmas \ref{est of A},   \ref{est of V-2},  and    \cite[Theorem 4-1]{Ja1}  we  see that $V^{mn,n}$  converges stably   in law in $\mathcal{C}_0$ to a conditional  Gaussian martingale $V=\{V^{k, j}\}$ with
	$$
	\begin{aligned}
		\langle V^{k_1, j_1}, V^{k_2, j_2}\rangle_t & = \begin{cases}\sum_{j=1}^{q}\frac{g^H_m}{\Gamma(2H+1)\sin\pi H}\int_{0}^{t}\sigma^{k_1}_j(X_{s})\sigma^{k_2}_{j}(X_{s})\d s& \text { if } j_1=j_2, \\
			0 & \text { if } j_1 \neq j_2,\end{cases} \\
		\langle V^{k, i}, W^j\rangle_t & =0, \quad{\forall} k \in\{1, \ldots, d\},{ \forall}(i, j) \in\{1, \ldots, q\}^2 .
	\end{aligned}
	$$
	
	Furthermore, since $\mathrm{d}\langle V, V\rangle_t \ll \mathrm{d} t$,  an application of  \cite[Proposition 1-4]{Ja1} yields that $V$ can be  represented by 
	$$
	V^{k, j}=\sum_{l=1}^q\sqrt{\frac{g^H_m}{\Gamma(2H+1)\sin\pi H}}\int_{0}^{t}\sigma^{k}_l(X_{s})\d B_s^{l,j},
	$$
	where $B$ is $q^2
	$-dimensional standard Brownian motion  independent of $W$. This concludes the proof.
\end{proof}
The following lemmas	\ref{nabla b}-\ref{unqueness of U}  are analogous to Lemma 2.4-2.8 of \cite{FU}, and so we omitted the proof, which are also similar.
\begin{lem}\label{nabla b}
	For all $i\in\{1,\cdots,d\},m\in\mathbb{N}\backslash \{0,1\} $ and any $\epsilon\in (0,H)$   we have 
	$$ \int_{0}^{t}K(t-s) n^{H}\nabla b^i(X^n_{\frac{[ns]}{n}})\Big(X^{n}_s-X^{n}_\frac{[ns]}{n}+X^{mn}_\frac{[mns]}{mn}-X^{mn}_s\Big)\d s
	\xrightarrow{ \text { in } \mathcal{C}_0^{H-\epsilon}} 0\quad \hbox{in probability} .$$
\end{lem}
\begin{lem}\label{remainder term}
	For all $m\in\mathbb{N}\backslash \{0,1\} $, $\|\mathcal{R}^{mn,n}\|_{C_0^{\gamma}}$ tends to zero in $L^p$ for any $\gamma\in (0,H),m\in\mathbb{N}\backslash \{0,1\} $ and any $p\geq 1$ as $n$ goes to infinity.
\end{lem}
\begin{lem}\label{limit of U}
	For all $m\in\mathbb{N}\backslash \{0,1\} $, if the sequence
	$$
	\left(U^{mn,n}, V^{mn,n},\left\{\nabla b^i(X^n)\right\}_i,\left\{\partial_k \sigma_j^i(X^n)\right\}_{i j k}\right)
	$$
	converges in law in $\mathcal{C}_0^{H-\epsilon} \times \mathcal{C}_0 \times \mathcal{D}_{d^2} \times \mathcal{D}_{d^2 q}$ to
	$$
	\left(U, V,\left\{\nabla b^i(X)\right\}_i,\left\{\partial_k \sigma_j^i(X)\right\}_{i j k}\right),
	$$
	then $U$ is the  solution of \eqref{limit Volterra eq}.
\end{lem}
\begin{lem}\label{tight of U n}
	For all $m\in\mathbb{N}\backslash \{0,1\} $, the sequence $U^{mn,n}$ is tight in $\mathcal{C}_0^{H-\epsilon}$ for any $\epsilon\in (0,H)$.
\end{lem}
\begin{lem}\label{unqueness of U}
	The strong uniqueness   holds for   solution to the equation  \eqref{limit Volterra eq}.
\end{lem}
We now prove  the main theorem.
\begin{proof}[Proof of Theorem \ref{t.main}]
	By Lemma \ref{tight of U n} and Lemma B.1 in \cite{FU}, or more directly by Lemma \ref{rate X-X n}, $X^n \rightarrow X$ in probability in the uniform topology. Therefore,
	$$
	\left(\left\{\nabla b^i(X^n)\right\}_i,\left\{\partial_k \sigma_j^i(X^n)\right\}_{i j k}\right) \rightarrow\left(\left\{\nabla b^i(X)\right\}_i,\left\{\partial_k \sigma_j^i(X)\right\}_{i j k}\right)
	$$
	in probability in the uniform topology as well. Together with Lemmas \ref{charact of V} and \ref{tight of U n}, we conclude that
	$$
	\left(U^{mn,n}, V^{mn,n},\left\{\nabla b^i(X^n)\right\}_i,\left\{\partial_k \sigma_j^i(X^n)\right\}_{i j k}, Y\right)
	$$
	is tight in $\mathcal{C}_0^{H-\epsilon} \times \mathcal{C}_0 \times \mathcal{D}_{d^2} \times \mathcal{D}_{d^2 q} \times \mathbb{R}$ for any random variable $Y$ on $(\Omega, \mathcal{F},\mathbb{P})$. For any subsequence of this tight sequence, there exists   further a convergent subsequence  by Prokhorov's theorem (see, e.g., Theorem 5.1 of Billingsley \cite{Bil} for a nonseparable case). Lemmas \ref{limit of U} and \ref{unqueness of U} imply the uniqueness of the limit. Therefore the original sequence itself has to converge.    Again by Lemma \ref{limit of U} the limit $U$ of $U^n$ is characterized by \eqref{limit Volterra eq}. This  convergence of $U^n$ is stable because $Y$ is arbitrary.
\end{proof}
	
	\section{Lindeberg-Feller central limit theorem}\label{CLT}
	For $L=\log n/\log m $,   set 
	\begin{align}\label{def of Ef}
		\tilde Q_n=\E[f(X^1_T)]+\sum_{l=1}^{L}\E\big[f(X_T^{m^l})-f(X_T^{m^{l-1}})\big].
	\end{align}
	We are going to use 
  \begin{align}\label{def of Q}
		Q_n=\frac{1}{N_0}\sum_{k=1}^{N_0}f(X^1_{T,k})+\sum_{l=1}^{L}\frac{1}{N_l}\sum_{k=1}^{N_l}\big(f(X^{l,m^l}_{T,k})-f(X^{l,m^{l-1}}_{T,k})\big).
	\end{align} 
	to approximate the above $ \tilde Q_n$. 
Here, it is important to point out that all these $L+1$ Monte Carlo estimators have to be based on different, independent samples. For each $\ell \in\{1, \ldots, L\}$ the samples $\left(X_{T, k}^{\ell, m^{\ell}}, X_{T, k}^{\ell, m^{\ell-1}}\right)_{1 \leq k \leq N_{\ell}}$ are independent copies of $\left(X_T^{\ell, m^{\ell}}, X_T^{\ell, m^{\ell-1}}\right.$ ) whose components denote the Euler schemes with time steps $m^{-\ell} T$ and $m^{-(\ell-1)} T$ and simulated with the same Brownian path. Concerning the first empirical mean, the samples $\left(X_{T, k}^1\right)_{1 \leq k \leq N_0}$ are independent copies of $X_T^1$. The following result gives us a first and rough description of the asymptotic behavior of the variance in the multilevel Monte Carlo Euler method associated with stochastic Volterra equations with singular kernels.
	\begin{prop} Assume  $H_{b,\sigma}$. For a Lipschitz continuous function $f:\mathbb{R}^d\to \mathbb{R}$, we have
		\begin{align}\label{Var of Q}
			\Var(Q_n)=O\Big(\sum_{l=0}^{L}N_l^{-1}m^{-2Hl}\Big).
		\end{align}
	\end{prop}
	\begin{proof}
		By the Lipschitz continuity  of $f$, Lemma \ref{bound of X n} and Lemma \ref{rate X-X n} we have
		\begin{align}
			\Var(Q_n)&=N_0^{-1}\Var(f(X^1_T))+\sum_{l=1}^{L}N_l^{-1}\Var\big(f(X^{l,m^l}_T)-f(X^{l,m^{l-1}}_T)\big)\nonumber\\&\quad\leq N_0^{-1}\Var(f(X^1_T))+2\sum_{l=1}^{L}N_l^{-1}\Big[\Var\big(f(X^{l,m^l}_T)-f(X_T)\big)+\Var\big(f(X_T)-f(X^{l,m^{l-1}}_T)\big)\Big]\nonumber\\&\quad\leq N_0^{-1}\Var(f(X^1_T))+2[f]_{lip}\sum_{l=1}^{L}N_l^{-1}\Big[\sup_{t\in[0,T]}\E\Big[|X_t^{m^l}-X_t|^2\Big]+\sup_{t\in[0,T]}\E\Big[|X_t^{m^{l-1}}-X_t|^2\Big]\Big]\nonumber\\&\quad\leq CN_0^{-1}+C\sum_{l=1}^{L}N_l^{-1} m^{-2Hl}\leq C\sum_{l=0}^{L}N_l^{-1}m^{-2Hl}.  .\nonumber                                                       
		\end{align}
		The proof is complete.
	\end{proof}
	We recall the following Lindeberg-Feller central limit theorem that will be
used in the sequel (see, e.g., Theorems 7.2 and 7.3 in \cite{Bil}).
\begin{thm}\label{Central limit theorem for triangular array}
	(Central limit theorem for triangular array). Let $\left(k_n\right)_{n \in \mathbb{N}}$ be a sequence such that $k_n \rightarrow \infty$ as $n \rightarrow \infty$. For each $n$, let $X_{n, 1}, \ldots, X_{n, k_n}$ be $k_n$ independent random variables with finite variance such that $\mathbb{E}\left(X_{n, k}\right)=0$ for all $k \in\left\{1, \ldots, k_n\right\}$. Suppose that the following conditions hold:
	
	(A1) $\lim _{n \rightarrow \infty} \sum_{k=1}^{k_n} \mathbb{E}\left|X_{n, k}\right|^2=\sigma^2 >0$.
	
	(A2) Lindeberg's condition: for all $\varepsilon>0, \lim _{n \rightarrow \infty} \sum_{k=1}^{k_n} \mathbb{E}\left(\left|X_{n, k}\right|^2   \mathbb{I}_{\left\{\left|X_{n, k}\right|>\varepsilon\right\}}\right) =0$. Then
	
	$$
	\sum_{k=1}^{k_n} X_{n, k} \Rightarrow \mathcal{N}\left(0, \sigma^2\right) \quad \text { as } n \rightarrow \infty.
	$$
	Moreover, if the $X_{n, k}$ have moments of order $p>2$, then the above Lindeberg's condition (A2)   is implied  by the following one:
	
	(A3) Lyapunov's condition: $\lim _{n \rightarrow \infty} \sum_{k=1}^{k_n} \mathbb{E}\left|X_{n, k}\right|^p=0$.
\end{thm}
 According to Section 2 of Jacod \cite{Ja1} and Lemma 2.1 of
Jacod and Protter \cite{Ja2}, we have the following result.
\begin{lem}\label{stably convergence}
	Let $V_n$ and $V$ be defined on $(\Omega,\mathcal{F})$ with values in another metric
	space $E'$.
	$$\text{If}\ V^n\xrightarrow{\mathbb{P}} V,\ X^n\Rightarrow^{\text {stably }}X,\quad \text{then}\ (V^n,X^n)\Rightarrow^{\text {stably }}(V,X).$$
	Conversely, if $(V^n,X^n)\Rightarrow(V,X)$ and $V$  generates the $\sigma$-field $\mathcal{F}$, we can realize this limit as $(V,X)$ with $X$ is defined on an extension of $(\Omega,\mathcal{F},\mathbb{P})$ and $X^n\Rightarrow^{\text {stably }}X$.
\end{lem}
In the same way as in the case of a crude Monte Carlo estimation, let us assume that the discretization error
$$
\varepsilon_n=\mathbb{E} f\left(X_T^n\right)-\mathbb{E} f\left(X_T\right)
$$
is of order $1 / n^\alpha$ for any $\alpha \in[1 / 2,1]$. Taking advantage from the limit theorem proven in the precedent  section, we are now able to establish a central limit theorem of Lindeberg-Feller type on the multilevel Monte Carlo Euler method. To do so, we introduce a real sequence $\left(a_{\ell}\right)_{\ell \in \mathbb{N}}$ of positive numbers  such that
$$
\lim _{L \rightarrow \infty} \sum_{\ell=1}^L a_{\ell}=\infty \quad \text { and } \quad \lim _{L \rightarrow \infty} \frac{1}{\left(\sum_{\ell=1}^L a_{\ell}\right)^{p / 2}} \sum_{\ell=1}^L a_{\ell}^{p / 2}=0
$$
for $p>2$
and we assume that the sample size $N_{\ell}$  is given by 
\begin{align}\label{def of N l}
	N_{\ell}=\frac{n^{2 \alpha}}{m^{2(\ell-1)H} a_{\ell}} \sum_{\ell=1}^L a_{\ell}, \quad \ell \in\{0, \ldots, L\} \text { and } L=\frac{\log n}{\log m}.
\end{align}
We choose this form for $N_{\ell}$ because it is a generic form allowing us a straightforward use of Toeplitz lemma that is a crucial tool used in the proof of our central limit theorem. Indeed, the above choice of $a_\ell$   implies that if $\left(x_{\ell}\right)_{\ell \geq 1}$ is a sequence converging to $x \in \mathbb{R}$ as $\ell$ tends to infinity then

$$
\lim _{L \rightarrow+\infty} \frac{\sum_{\ell=1}^L a_{\ell} x_{\ell}}{\sum_{\ell=1}^L a_{\ell}}=x
$$

In the sequel, we will denote by $\widetilde{\mathbb{E}}$  (respectively  $\widetilde{\operatorname{Var}}$) 
 the expectation (respectively  the variance)  defined on the probability space $(\widetilde{\Omega}, \widetilde{\mathcal{F}}, \widetilde{\mathbb{P}})$ introduced in Theorem \ref{t.main}. We can now state the central limit theorem under strengthened conditions on the diffusion coefficients.
\begin{thm} \label{t.3.4} 
Assume that the derivatives of  the coefficients  $b$ and $\sigma$  are bounded. Let $f$ be a real-valued function satisfying
	 
	$\left(\mathcal{H}_f\right) \quad|f(x)-f(y)| \leq C\left(1+|x|^p+|y|^p\right)|x-y| \quad$ for some $C, p>0$.
	
	Assume $\mathbb{P}\left(X_T \notin \mathcal{D}_f\right)=0$, where $\mathcal{D}_f:=\left\{x \in \mathbb{R}^d ; f\right.$ is differentiable at $\left.x\right\}$, and that for some $\alpha \in[1 / 2,1]$ we have
	
	$$
\left(\mathcal{H}_{\varepsilon_n}\right)\quad	\lim _{n \rightarrow \infty} n^\alpha \varepsilon_n=C_f(T, \alpha)
	$$ 
	
	 	 Fix  $m\in\mathbb{N}\backslash \{0,1\} $.   Then, for the choice of $N_{\ell}, \ell \in\{0,1, \ldots, L\}$ given by \eqref{def of N l}, we have
	
	$$
	n^\alpha\left(Q_n-\mathbb{E}\left(f\left(X_T\right)\right)\right) \Rightarrow \mathcal{N}\left(C_f(T, \alpha), \sigma^2\right)
	$$
	
	with $\sigma^2=\widetilde{\operatorname{Var}}\left(\nabla f\left(X_T\right) \cdot U_T\right)$ and $\mathcal{N}\left(C_f(T, \alpha), \sigma^2\right)$ denotes a normal distribution.
\end{thm}
	\begin{proof}
	To simplify the presentation we give the proof for $\alpha=1$, the case $\alpha \in[1 / 2,1)$ is a straightforward extension. Combining relations \eqref{def of Ef} and \eqref{def of Q} together, we have
	$$
	Q_n-\mathbb{E}\left(f\left(X_T\right)\right)=\widehat{Q}_n^1+\widehat{Q}_n^2+\varepsilon_n,
	$$ 
	where 
	$$
	\begin{aligned}
		& \widehat{Q}_n^1=\frac{1}{N_0} \sum_{k=1}^{N_0}\left(f(X_{T, k}^1)-\mathbb{E}\left(f(X_T^1)\right)\right) \\
		& \widehat{Q}_n^2=\sum_{\ell=1}^L \frac{1}{N_{\ell}} \sum_{k=1}^{N_{\ell}}\left(f(X_{T, k}^{\ell, m^{\ell}})-f(X_{T, k}^{\ell, m^{\ell-1}})-\mathbb{E}\left(f\left(X_T^{\ell, m^{\ell}}\right)-f\left(X_T^{\ell, m^{\ell-1}}\right)\right)\right)\,. 
	\end{aligned}
	$$

	Using assumption ( $\mathcal{H}_{\varepsilon_n}$ ) and taking the expectation, we see obviously   the term $C_f(T, \alpha)$ in the limit. Taking $N_0=\frac{n^2(m-1) T}{a_0} \sum_{\ell=1}^L a_{\ell}$, we can apply the classical central limit theorem to $\widehat{Q}_n^1$. Then we have $n \widehat{Q}_n^1 \xrightarrow{\mathbb{P}} 0$. Finally, we are going  to study the convergence of $n \widehat{Q}_n^2$ and we can  conclude our proof by establishing
	$$
	n \widehat{Q}_n^2 \Rightarrow \mathcal{N}\left(0, \widetilde{\operatorname{Var}}\left(\nabla f\left(X_T\right) \cdot U_T\right)\right)
	$$
	To show the above convergence, we plan to use Theorem \ref{Central limit theorem for triangular array} with the Lyapunov condition and we set
	$$
	\begin{aligned}
		X_{n, \ell} & :=\frac{n}{N_{\ell}} \sum_{k=1}^{N_{\ell}} Z_{T, k}^{m^{\ell}, m^{\ell-1}} \text { and } \\
		Z_{T, k}^{m^{\ell}, m^{\ell-1}} & :=f\left(X_{T, k}^{\ell, m^{\ell}}\right)-f\left(X_{T, k}^{\ell, m^{\ell-1}}\right)-\mathbb{E}\left(f\left(X_{T, K}^{\ell, m^{\ell}}\right)-f\left(X_{T, k}^{\ell, m^{\ell-1}}\right)\right).
	\end{aligned}
	$$
	In other words, we shall verify  the following statements. 
	
	- $\lim _{n \rightarrow \infty} \sum_{\ell=1}^L \mathbb{E}\left(X_{n, \ell}\right)^2=\widetilde{\operatorname{Var}}\left(\nabla f\left(X_T\right) \cdot U_T\right)$.
	
	- (Lyapunov condition) there exists $p>2$ such that $\lim _{n \rightarrow \infty} \sum_{\ell=1}^L \mathbb{E}\left|X_{n, \ell}\right|^p=0$.
	
	For the first claim, since $ \mathbb{E}\left(X_{n, \ell}\right)=0$, we have
	\begin{align}\label{X 2 nl}
		\sum_{\ell=1}^L \mathbb{E}\left(X_{n, \ell}\right)^2 & =\sum_{\ell=1}^L \operatorname{Var}\left(X_{n, \ell}\right) =\sum_{\ell=1}^L \frac{n^2}{N_{\ell}} \operatorname{Var}\left(Z_{T, 1}^{m^{\ell}, m^{\ell-1}}\right) \nonumber\\
		& =\frac{1}{\sum_{\ell=1}^L a_{\ell}} \sum_{\ell=1}^L a_{\ell} m^{2(\ell-1)H} \operatorname{Var}\left(Z_{T, 1}^{m^{\ell}, m^{\ell-1}}\right).
	\end{align}
	Otherwise, since $\mathbb{P}\left(X_T \notin \mathcal{D}_f\right)=0$, applying the Taylor expansion theorem twice, we have
	\begin{align}
		f(X_T^{\ell, m^{\ell}})-f(X_T^{\ell, m^{\ell-1}})= &\nabla f\left(X_T\right) \cdot U_T^{m^{\ell}, m^{\ell-1}}+\left(X_T^{\ell, m^{\ell}}-X_T\right) \varepsilon\left(X_T, X_T^{\ell, m^{\ell}}-X_T\right)\nonumber \\
		& -\left(X_T^{\ell, m^{\ell-1}}-X_T\right) \varepsilon\left(X_T, X_T^{\ell, m^{\ell-1}}-X_T\right)\nonumber.
	\end{align}
	The function $\varepsilon$ is given by the Taylor-Young expansion, so it satisfies $\varepsilon\left(X_T, X_T^{\ell, m^{\ell}}-X_T\right) \xrightarrow[\ell \rightarrow \infty]{\mathbb{P}} 0$ and $\varepsilon\left(X_T, X_T^{\ell, m^{\ell-1}}-X_T\right) \xrightarrow[\ell \rightarrow \infty]{\mathbb{P}} 0$. By Lemma \ref{rate X-X n}, we get the tightness of $m^{(\ell-1)H}\left(X_T^{\ell, m^{\ell}}-X_T\right)$ and $m^{(\ell-1)H}\left(X_T^{\ell, m^{\ell-1}}-X_T\right)$ and then we deduce
	$$
	\begin{aligned}
		m^{(\ell-1)H} \left(\left(X_T^{\ell, m^{\ell}}-X_T\right) \varepsilon\left(X_T, X_T^{\ell, m^{\ell}}-X_T\right)\right. 
		\left.-\left(X_T^{\ell, m^{\ell-1}}-X_T\right) \varepsilon\left(X_T, X_T^{\ell, m^{\ell-1}}-X_T\right)\right) \xrightarrow[\ell \rightarrow \infty]{\mathbb{P}} 0 .
	\end{aligned}
	$$
According to Theorem \ref{t.main} and Lemma \ref{stably convergence}, we conclude that
	$$
	m^{(\ell-1)H}\left(f\left(X_T^{\ell, m^{\ell}}\right)-f\left(X_T^{\ell, m^{\ell-1}}\right)\right) \Rightarrow^{\text {stably }} \nabla f\left(X_T\right) \cdot U_T
	$$	
	as $\ell \rightarrow \infty$.
	It follows from  $\left(\mathcal{H}_f\right)$, Lemma \ref{bound of X n} and Lemma \ref{rate X-X n} that
	 \begin{align}\label{limit of mul-level}
	 		\forall \varepsilon>0 \quad \sup _{\ell} \mathbb{E}\left|m^{(\ell-1)H}\left(f\left(X_T^{\ell, m^{\ell}}\right)-f\left(X_T^{\ell, m^{\ell-1}}\right)\right)\right|^{2+\varepsilon}<\infty.
	 \end{align}
	We deduce using \eqref{limit of mul-level} that
	$$
	\begin{aligned}
		\mathbb{E}\left(m^{(\ell-1)H}\left(f\left(X_T^{\ell, m^{\ell}}\right)-f\left(X_T^{\ell, m^{\ell-1}}\right)\right)\right)^k \rightarrow \widetilde{\mathbb{E}}\left(\nabla f\left(X_T\right) \cdot U_T\right)^k<\infty \\
		\quad \text { for } k \in\{1,2\} .
	\end{aligned}
	$$	
	Consequently,
	$$
	m^{2(\ell-1)H}\operatorname{Var}\left(Z_{T, 1}^{m^{\ell}, m^{\ell-1}}\right) \rightarrow \widetilde{\operatorname{Var}}\left(\nabla f\left(X_T\right) \cdot U_T\right)<\infty .
	$$
Combining this result together with  \eqref{X 2 nl}, we obtain the first condition using Toeplitz lemma. Concerning the second claim, by Burkhölder's inequality and elementary computations, we have for $p>2$
	$$
	\mathbb{E}\left|X_{n, \ell}\right|^p=\frac{n^p}{N_{\ell}^p} \mathbb{E}\left|\sum_{\ell=1}^{N_{\ell}} Z_{T, 1}^{m^{\ell}, m^{\ell-1}}\right|^p \leq C_p \frac{n^p}{N_{\ell}^{p / 2}} \mathbb{E}\left|Z_{T, 1}^{m^{\ell}, m^{\ell-1}}\right|^p,
	$$
	where $C_p$ is a numerical constant depending only on $p$. Otherwise, Lemma \ref{bound of X n} ensures the existence of a constant $K_p>0$ such that
	$$
	\mathbb{E}\left|Z_{T, 1}^{m^{\ell}, m^{\ell-1}}\right|^p \leq \frac{K_p}{m^{p \ell / 2}}\,. 
	$$ 	
	Therefore, 
		$$
	\begin{aligned}
		\sum_{\ell=1}^L \mathbb{E}\left|X_{n, \ell}\right|^p & \leq \widetilde{C}_p \sum_{\ell=1}^L \frac{n^p}{N_{\ell}^{p / 2} m^{p \ell / 2}} \\
		& \leq \frac{\widetilde{C}_p}{\left(\sum_{\ell=1}^L a_{\ell}\right)^{p / 2}} \sum_{\ell=1}^L a_{\ell}^{p / 2} \underset{n \rightarrow \infty}{\longrightarrow} 0 .
	\end{aligned}
	$$
	This completes the proof.
	\end{proof}
\section{Completion of the proofs of Lemma \ref{est of A} and Lemma \ref{est of V-2}}
\subsection{Completion of the proof of Lemma \ref{est of A}} In this subsection we provide the study of   the remaining terms of in the proof
of Lemma \ref{est of A}. 
\subsubsection{Analysis of $\textbf{(1,3)}_{j,l}^{ mn,n,k}+\textbf{(2,3)}_{j,l}^{ mn,n,k}$($\textbf{(3,1)}_{j,l}^{ mn,n,k}+\textbf{(3,2)}_{j,l}^{ mn,n,k}$)}\label{(4.2.1)}
We are now to deal with $\textbf{(1,3)}_{j,l}^{ mn,n,k}$ and $\textbf{(2,3)}_{j,l}^{ mn,n,k}$. By \eqref{dcp of stochastic integral}, we have 
\begin{align}
	\textbf{(1,3)}_{j,l}^{ mn,n,k}&=n^{2H}\int_{0}^{t}\Big(\int_{0}^{\frac{[n s]}{n}}\Delta K(n,s,u)\sigma_j^{k_1}(X^{n}_{\frac{[n u]}{n}})\d W^j_u\Big)\Big(\int_{0}^{\frac{[mn s]}{mn}}\Delta K(mn,s,u)\sigma^{k_2}_j(X^{mn}_{\frac{[mn u]}{mn}})\d W^l_u\Big)\d s\nonumber\\&=n^{2H}\int_{0}^{t}\Big(\int_{0}^{\frac{[n s]}{n}}\Delta K(n,s,u)\sigma_j^{k_1}(X^{n}_{\frac{[n u]}{n}})\d W^j_u\Big)\Big(\int_{0}^{\frac{[n s]}{n}}\Delta K(mn,s,u)\sigma^{k_2}_j(X^{mn}_{\frac{[mn u]}{mn}})\d W^l_u\Big)\d s\nonumber\\&\quad+n^{2H}\int_{0}^{t}\Big(\int_{0}^{\frac{[n s]}{n}}\Delta K(n,s,u)\sigma_j^{k_1}(X^{n}_{\frac{[n u]}{n}})\d W^j_u\Big)\Big(\int_{\frac{[ns]}{n}}^{\frac{[mn s]}{mn}}\Delta K(mn,s,u)\sigma^{k_2}_j(X^{mn}_{\frac{[mn u]}{mn}})\d W^l_u\Big)\d s\nonumber\\
	&=n^{2H}\int_{0}^{t}\int_{0}^{\frac{[ns]}{n}}\Big(\int_{0}^{u}\Delta K(mn,s,r) \sigma_l^{k_2}(X^{mn}_{\frac{[mn r]}{mn}})\d W^l_r\Big)\Delta K(n,s,u) \sigma_j^{k_1}(X^{mn}_{\frac{[mn u]}{mn}})\d W^j_u\d s\nonumber\\&\quad+n^{2H}\int_{0}^{t}\int_{0}^{\frac{[ns]}{n}}\Big(\int_{0}^{u}\Delta K(n,s,r)\sigma_j^{k_1}(X^{n}_{\frac{[n r]}{n}})\d W^j_r\Big)\Delta K(mn,s,u) \sigma_l^{k_2}(X^{mn}_{\frac{[mn u]}{mn}})\d W^l_u\d s\nonumber\\&\quad+n^{2H}\int_{0}^{t}\int_{0}^{\frac{[ns]}{n}}\Delta K(mn,s,u)\Delta K(n,s,u)\sigma_j^{k_1}(X^{n}_{\frac{[n u]}{n}})\sigma_l^{k_2}(X^{mn}_{\frac{[mn u]}{mn}})\d \langle W^{j}, W^{l} \rangle_u\d s\nonumber\\&\quad+n^{2H}\int_{0}^{t}\Big(\int_{0}^{\frac{[n s]}{n}}\Delta K(n,s,u)\sigma_j^{k_1}(X^{n}_{\frac{[n u]}{n}})\d W^j_u\Big)\Big(\int_{\frac{[ns]}{n}}^{\frac{[mn s]}{mn}}\Delta K(mn,s,u)\sigma^{k_2}_j(X^{mn}_{\frac{[mn u]}{mn}})\d W^l_u\Big)\d s\nonumber\\
	&:=\textbf{(1,3,1)}_{j,l}^{ mn,n,k}+\textbf{(1,3,2)}_{j,l}^{ mn,n,k}+\textbf{(1,3,3)}_{j,l}^{ mn,n,k}+\textbf{(1,3,4)}_{j,l}^{ mn,n,k},\nonumber
\end{align}
and
\begin{align}
	\textbf{(2,3)}_{j,l}^{ mn,n,k}&=n^{2H}\int_{0}^{t}\Big(\sigma^{k_1}_j(X^{n}_{\frac{[ns]}{n}})\int_{\frac{[n s]}{n}}^{s}K(s-u)\d W^j_u\Big)\Big(\int_{0}^{\frac{[mn s]}{mn}}\Delta K(mn,s,u)\sigma^{k_2}_j(X^{mn}_{\frac{[mn u]}{mn}})\d W^l_u\Big)\d s\nonumber\\&=n^{2H}\int_{0}^{t}\Big(\sigma^{k_1}_j(X^{n}_{\frac{[ns]}{n}})\int_{\frac{[n s]}{n}}^{\frac{[mns]}{mn}}K(s-u)\d W^j_u\Big)\Big(\int_{0}^{\frac{[n s]}{n}}\Delta K(mn,s,u)\sigma^{k_2}_j(X^{mn}_{\frac{[mn u]}{mn}})\d W^l_u\Big)\d s\nonumber\\&\quad+n^{2H}\int_{0}^{t}\Big(\sigma^{k_1}_j(X^{n}_{\frac{[ns]}{n}})\int_{\frac{[n s]}{n}}^{\frac{[mns]}{mn}}K(s-u)\d W^j_u\Big)\Big(\int_{\frac{[ns]}{n}}^{\frac{[mn s]}{mn}}\Delta K(mn,s,u)\sigma^{k_2}_j(X^{mn}_{\frac{[mn u]}{mn}})\d W^l_u\Big)\d s\nonumber\\&\quad+n^{2H}\int_{0}^{t}\Big(\sigma^{k_1}_j(X^{n}_{\frac{[ns]}{n}})\int_{\frac{[mn s]}{mn}}^{s}K(s-u)\d W^j_u\Big)\Big(\int_{0}^{\frac{[mn s]}{mn}}\Delta K(mn,s,u)\sigma^{k_2}_j(X^{mn}_{\frac{[mn u]}{mn}})\d W^l_u\Big)\d s\nonumber\\&=n^{2H}\int_{0}^{t}\Big(\sigma^{k_1}_j(X^{n}_{\frac{[ns]}{n}})\int_{\frac{[n s]}{n}}^{\frac{[mns]}{mn}}K(s-u)\d W^j_u\Big)\Big(\int_{0}^{\frac{[n s]}{n}}\Delta K(mn,s,u)\sigma^{k_2}_j(X^{mn}_{\frac{[mn u]}{mn}})\d W^l_u\Big)\d s\nonumber\\&\quad+n^{2H}\int_{0}^{t}\int_{\frac{[ns]}{n}}^{\frac{[mns]}{mn}}\Big(\int_{\frac{[ns]}{n}}^{u}\Delta K(mn,s,r) \sigma_l^{k_2}(X^{mn}_{\frac{[mn r]}{mn}})\d W^l_r\Big)K(s-u) \sigma_j^{k_1}(X^{n}_{\frac{[n u]}{n}})\d W^j_u\d s\nonumber\\&\quad+n^{2H}\int_{0}^{t}\int_{\frac{[ns]}{n}}^{\frac{[mns]}{mn}}\Big(\int_{\frac{[ns]}{n}}^{u}K(s-r)\sigma_j^{k_1}(X^{n}_{\frac{[n r]}{n}})\d W^j_r\Big)\Delta K(mn,s,u) \sigma_l^{k_2}(X^{mn}_{\frac{[mn u]}{mn}})\d W^l_u\d s\nonumber\\&\quad+n^{2H}\int_{0}^{t}\int_{\frac{[ns]}{n}}^{\frac{[mns]}{mn}}\Delta K(mn,s,u)K(s-u)\sigma_j^{k_1}(X^{n}_{\frac{[n u]}{n}})\sigma_l^{k_2}(X^{mn}_{\frac{[mn u]}{mn}})\d \langle W^{j}, W^{l} \rangle_u\d s\nonumber\\&\quad+n^{2H}\int_{0}^{t}\Big(\sigma^{k_1}_j(X^{n}_{\frac{[ns]}{n}})\int_{\frac{[mn s]}{mn}}^{s}K(s-u)\d W^j_u\Big)\Big(\int_{0}^{\frac{[mn s]}{mn}}\Delta K(mn,s,u)\sigma^{k_2}_j(X^{mn}_{\frac{[mn u]}{mn}})\d W^l_u\Big)\d s\nonumber\\
	&:=\textbf{(2,3,1)}_{j,l}^{ mn,n,k}+\textbf{(2,3,2)}_{j,l}^{ mn,n,k}+\textbf{(2,3,3)}_{j,l}^{ mn,n,k}+\textbf{(2,3,4)}_{j,l}^{ mn,n,k}+ \textbf{(2,3,5)}_{j,l}^{ mn,n,k}.\nonumber
\end{align}
To study the term $\textbf{(1,3,1)}_{j,l}^{ mn,n,k}$  we let
\begin{align}
	A_{1,s}^{(m,n,j,l)}=n^{2H}\int_{0}^{\frac{[ns]}{n}}\Big(\int_{0}^{u}\Delta K(mn,s,r) \sigma_l^{k_2}(X^{mn}_{\frac{[mn r]}{mn}})\d W^l_r\Big)\Delta K(n,s,u) \sigma_j^{k_1}(X^{n}_{\frac{[n u]}{n}})\d W^j_u.\nonumber
\end{align}
By Fubini's theorem we can rewrite $\E[|\textbf{(1,3,1)}_{j,l}^{ mn,n,k}|^2]$ as
\begin{align}
	\E[|\textbf{(1,3,1)}_{j,l}^{ mn,n,k}|^2]&=\E\Big[\Big|\int_{0}^{t}A_{1,s}^{(m,n,j,l)}\d s\Big|^2\Big]\nonumber\\&=\E\Big[\int_{0}^{t}\int_{0}^{t}A_{1,s}^{(m,n,j,l)}A_{1,v}^{(m,n,j,l)}\d v\d s\Big]\nonumber\\&=2\int_{0}^{t}\int_{0}^{s}\E\big[A_{1,s}^{(m,n,j,l)}A_{1,v}^{(m,n,j,l)}\big]\d v\d s.\nonumber
\end{align}
Applying \eqref{dcp of stochastic integral}, the tower property (see, e.g. \cite{BD}) and Fubini's theorem, we have
\begin{align}\label{double of A 1}
	&\E\big[A_{1,s}^{(m,n,j,l)}A_{1,v}^{(m,n,j,l)}\big]\nonumber\\&=n^{4H}\E\Big[\int_{0}^{\frac{[nv]}{n}}\Big(\int_{0}^{u}\Delta K(mn,s,r) \sigma_l^{k_2}(X^{mn}_{\frac{[mn r]}{mn}})\d W^l_r\Big)\Big(\int_{0}^{u}\Delta K(mn,v,r) \sigma_l^{k_2}(X^{n}_{\frac{[mn r]}{mn}})\d W^l_r\Big)\nonumber\\&\quad\cdot\Delta K(n,s,u)\Delta K(n,v,u) |\sigma_j^{k_1}(X^{n}_{\frac{[n u]}{n}})|^2\d u\Big]\nonumber\\&=n^{4H}\int_{0}^{\frac{[nv]}{n}}\Delta K(n,s,u)\Delta K(n,v,u)\cdot E^{(1)}_u\d u,
\end{align}
where
\begin{align}
	E^{(1)}_u:=\E\Big[\Big(\int_{0}^{u}\Delta K(mn,s,r) \sigma_l^{k_2}(X^{mn}_{\frac{[mn r]}{mn}})\d W^l_r\Big)\Big(\int_{0}^{u}\Delta K(mn,v,r) \sigma_l^{k_2}(X^{mn}_{\frac{[mn r]}{mn}})\d W^l_r\Big)|\sigma_j^{k_1}(X^{n}_{\frac{[n u]}{n}})|^2\Big].\nonumber
\end{align}
Since for any $m\geq 2, u\leq \frac{[mns]}{mn}$ and $s\leq T$, it holds
\begin{align}\label{esti of stochastic integral_1}
	&\Big\|\int_{0}^{u}\Delta K(mn,s,r)\sigma_l^{k_2}(X^{mn}_{\frac{[mn r]}{mn}}) \d W_r^l\Big\|_{L^m}\nonumber\\&\leq C \Big\|\int_{0}^{u}\Big|\Delta K(mn,s,r)\sigma_l^{k_2}(X^{mn}_{\frac{[mn r]}{mn}})\Big|^2\d r\Big\|^{\frac{1}{2}}_{L^{m/2}}\nonumber\\&\leq C\Big\|\int_{0}^{\frac{[mns]}{mn}}\Big(K(s-r)-K(\frac{[mns]}{mn}-r)\Big)^2|\sigma_l^{k_2}(X^{mn}_{\frac{[mn r]}{mn}})|^2\d r\Big\|^{\frac{1}{2}}_{L^{m/2}}\nonumber\\&\leq C\Big(\int_{0}^{\frac{[mns]}{mn}}\Big(K(s-r)-K(\frac{[mns]}{mn}-r)\Big)^2\||\sigma_l^{k_2}(X^{mn}_{\frac{[mn r]}{mn}})|^2\|_{L^{m/2}}\d r\Big)^{1/2}\nonumber\\&\leq Cn^{-H},
\end{align}
where Lemma \ref{est-3_m}, the boundness of $\sigma$, BDG's inequality, Minkowski's  inequality are used.

Then by the Cauchy-Schwarz inequality and \eqref{esti of stochastic integral_1}, we have
\begin{align}
	\E[E^{(1)}_{u}]&\leq \Big\|\int_{0}^{u}\Delta K(mn,s,r) \sigma_l^{k_2}(X^{mn}_{\frac{[mn r]}{mn}})\d W_r^j\Big\|_{L^4}\Big\| \int_{0}^{u}\Delta K(mn,v,r) \sigma_l^{k_2}(X^{mn}_{\frac{[mn r]}{mn}})\d W_r^j\Big\|_{L^4}\|\sigma_j^{k_1}(X^{n}_{\frac{[n u]}{n}})\|^2_{L^4}\nonumber\\&\leq Cn^{-2H}.\nonumber
\end{align}
Finally, Lemma \ref{esti of A n} and \eqref{double of A 1} give that
$$\Big|\E\big[A_{1,s}^{(m,n,j,l)}A_{1,v}^{(m,n,j,l)}\big]\Big|\leq Cn^{2H}\int_{0}^{\frac{[nv]}{n}}\Delta K(n,s,u)\Delta K(n,v,u)\d u=0.$$
Applying the dominated convergence theorem  with respect to $\d v\otimes \d s$, we have $$\textbf{(1,3,1)}_{j,l}^{ mn,n,k}\xrightarrow {  L^2} 0.$$
Similar to $\textbf{(1,3,1)}_{j,l}^{ mn,n,k}$, it holds that $\textbf{(1,3,2)}_{j,l}^{ mn,n,k}\to 0$ in $L^2$.

We now deal with $\textbf{(1,3,3)}_{j,l}^{ mn,n,k}+\textbf{(2,3,4)}_{j,l}^{ mn,n,k}$ together, this term vanishes if $j\neq l$.  When   $j=l$, noting  that $\frac{[mnu]}{mn}=\frac{[mns]}{mn}$ for $u\in (\frac{[ns]}{n},\frac{[mns]}{mn})$,   by  simple calculation, we have
\begin{align}
	&\textbf{(1,3,3)}_{j,j}^{ mn,n,k}+\textbf{(2,3,4)}_{j,j}^{ mn,n,k}\nonumber\\&=n^{2H}\int_{0}^{t}\int_{0}^{\frac{[ns]}{n}}\Delta K(mn,s,u)\Delta K(n,s,u)\sigma_j^{k_1}(X^{n}_{\frac{[n u]}{n}})\sigma_j^{k_2}(X^{mn}_{\frac{[mn u]}{mn}})\d u\d s\nonumber\\&\quad+n^{2H}\int_{0}^{t}\int_{\frac{[ns]}{n}}^{\frac{[mns]}{mn}}\Delta K(mn,s,u)K(s-u)\sigma_j^{k_1}(X^{n}_{\frac{[n u]}{n}})\sigma_j^{k_2}(X^{mn}_{\frac{[mn u]}{mn}})\d u\d s\nonumber\\
	&=-\frac{1}{2}n^{2H}\int_{0}^{t}\int_{0}^{\frac{[ns]}{n}}\Big(\Delta K(mn,n,s,u)\Big)^2\sigma_j^{k_1}(X^{n}_{\frac{[n u]}{n}})\sigma_j^{k_2}(X^{mn}_{\frac{[mn u]}{mn}})\d u\d s\nonumber\\
	&\quad+\frac{1}{2}n^{2H}\int_{0}^{t}\int_{0}^{\frac{[ns]}{n}}\Big(\Delta K(n,s,u)\Big)^2\sigma_j^{k_1}(X^{n}_{\frac{[n u]}{n}})\sigma_j^{k_2}(X^{mn}_{\frac{[mn u]}{mn}})\d u\d s\nonumber\\&\quad-\frac{1}{2}n^{2H}\int_{0}^{t}E_s^{m,n,j}\int_{\frac{[ns]}{n}}^{\frac{[mns]}{mn}}\Big( K(\frac{[mns]}{mn}-u)\Big)^2\d u\d s +\frac{1}{2}n^{2H}\int_{0}^{t}E_s^{m,n,j}\int_{\frac{[ns]}{n}}^{\frac{[mns]}{mn}}\Big( K(s-u)\Big)^2\d u\d s\nonumber\\&\quad +\frac{1}{2}n^{2H}\int_{0}^{t}\int_{0}^{\frac{[ns]}{n}}\Big(\Delta K(mn,s,u)\Big)^2\sigma_j^{k_1}(X^{n}_{\frac{[n u]}{n}})\sigma_j^{k_2}(X^{mn}_{\frac{[mn u]}{mn}})\d u\d s\nonumber\\	&\quad+\frac{1}{2}n^{2H}\int_{0}^{t}E_s^{m,n,j}\int_{\frac{[ns]}{n}}^{\frac{[mns]}{mn}}\Big(\Delta K(mn,s,u)\Big)^2\d u\d s\nonumber\\&=-\frac{1}{2}n^{2H}\int_{0}^{t}\int_{0}^{\frac{[ns]}{n}}\Big(\Delta K(mn,n,s,u)\Big)^2\sigma_j^{k_1}(X^{n}_{\frac{[n u]}{n}})\sigma_j^{k_2}(X^{mn}_{\frac{[mn u]}{mn}})\d u\d s\nonumber\\
	&\quad+\frac{1}{2}n^{2H}\int_{0}^{t}\int_{0}^{\frac{[ns]}{n}}\Big(\Delta K(n,s,u)\Big)^2\sigma_j^{k_1}(X^{n}_{\frac{[n u]}{n}})\sigma_j^{k_2}(X^{mn}_{\frac{[mn u]}{mn}})\d u\d s\nonumber\\&\quad +\frac{1}{2}n^{2H}\int_{0}^{t}\int_{0}^{\frac{[mns]}{mn}}\Big(\Delta K(mn,s,u)\Big)^2\sigma_j^{k_1}(X^{n}_{\frac{[n u]}{n}})\sigma_j^{k_2}(X^{mn}_{\frac{[mn u]}{mn}})\d u\d s\nonumber\\&\quad-\frac{1}{2}n^{2H}\int_{0}^{t}E_s^{m,n,j}\int_{\frac{[ns]}{n}}^{\frac{[mns]}{mn}}\Big( K(\frac{[mns]}{mn}-u)\Big)^2\d u\d s +\frac{1}{2}n^{2H}\int_{0}^{t}E_s^{m,n,j}\int_{\frac{[ns]}{n}}^{\frac{[mns]}{mn}}\Big( K(s-u)\Big)^2\d u\d s\nonumber\\
	&:=-\frac{1}{2}\textbf{(1,3,3,1)}_{j,j}^{ mn,n,k}+\frac{1}{2}\textbf{(1,3,3,2)}_{j,j}^{ mn,n,k}+\frac{1}{2}\textbf{(1,3,3,3)}_{j,j}^{ mn,n,k}\nonumber\\&\quad-\frac{1}{2}\textbf{(2,3,4,1)}_{j,j}^{ mn,n,k}+\frac{1}{2}\textbf{(2,3,4,2)}_{j,j}^{ mn,n,k},\nonumber
\end{align}
where $E_s^{m,n,j}:=\sigma_j^{k_1}(X^{n}_{\frac{[n s]}{n}})\sigma_j^{k_2}(X^{mn}_{\frac{[mn s]}{mn}})$ and $\Delta K(mn,n,s,u):=K(\frac{[mn s]}{mn}-u)-K(\frac{[n s]}{n}-u)$. 

On the one hand, the change of variables $\frac{[ns]}{n}-u=v$ and $r=v/\delta_{(mn,n,s)}$ in  $\textbf{(1,3,3,1)}_{j,j}^{ mn,n,k}$ yields
\begin{align}
	&n^{2H}\int_{0}^{t}\int_{0}^{\frac{[ns]}{n}}\Big(\Delta K(mn,n,s,u)\Big)^2\sigma_j^{k_1}(X^{n}_{\frac{[n u]}{n}})\sigma_j^{k_2}(X^{mn}_{\frac{[mn u]}{mn}})\d u\d s\nonumber\\&=n^{2H}\int_{0}^{t}\int_{0}^{\frac{[ns]}{n}}\Big(K(v+\delta_{(mn,n,s)})-K(v)\Big)^2\sigma_j^{k_1}(X^{n}_{\frac{[ns]+[-nv ]}{n}})\sigma_j^{k_2}(X^{mn}_{\frac{[m[ns]-mnv]}{mn}})\d v\d s\nonumber\\&=\frac{1}{G}\int_{0}^{t}(n\delta_{(mn,n,s)})^{2H}\int_{0}^{\frac{[ns]}{n\delta_{(mn,n,s)}}}|\mu(r,1)|^2\sigma_j^{k_1}(X^{n}_{\frac{[ns]+[-n\delta_{(mn,n,s)}r ]}{n}})\sigma_j^{k_2}(X^{mn}_{\frac{[m[ns]-mn\delta_{(mn,n,s)} r]}{mn}})\d r\d s,\nonumber
\end{align}	
where $\delta_{(mn,n,s)}=\frac{[mn s]}{mn}-\frac{[n s]}{n}$.	

Making the  change of variables $\frac{[ns]}{n}-u=v$ and $r=v/\delta_{(n,s)}$  in  $\textbf{(1,3,3,2)}_{j,j}^{ mn,n,k}$ yields 
\begin{align}
	&n^{2H}\int_{0}^{t}\int_{0}^{\frac{[ns]}{n}}\Big(\Delta K(n,s,u)\Big)^2\sigma_j^{k_1}(X^{n}_{\frac{[n u]}{n}})\sigma_j^{k_2}(X^{mn}_{\frac{[mn u]}{mn}})\d u\d s\nonumber\\&=n^{2H}\int_{0}^{t}\int_{0}^{\frac{[ns]}{n}}\Big(K(v+\delta_{(n,s)})-K(v)\Big)^2\sigma_j^{k_1}(X^{n}_{\frac{[ns]+[-nv ]}{n}})\sigma_j^{k_2}(X^{mn}_{\frac{[m[ns]-mnv]}{mn}})\d v\d s\nonumber\\&=\frac{1}{G}\int_{0}^{t}(n\delta_{(n,s)})^{2H}\int_{0}^{\frac{[ns]}{n\delta_{(n,s)}}}|\mu(r,1)|^2\sigma_j^{k_1}(X^{n}_{\frac{[ns]+[-n\delta_{(n,s)}r ]}{n}})\sigma_j^{k_2}(X^{mn}_{\frac{[m[ns]-mn\delta_{(n,s)} r]}{mn}})\d r\d s,\nonumber
\end{align}	
where $\delta_{(n,s)}=s-\frac{[n s]}{n}$.

For the term  to $\textbf{(1,3,3,3)}_{j,j}^{ mn,n,k}$ making substitution  $\frac{[mns]}{mn}-u=v$ and $r=v/\delta_{(mn,s)}$  yields
\begin{align}
	&n^{2H}\int_{0}^{t}\int_{0}^{\frac{[mns]}{mn}}\Big(\Delta K(mn,s,u)\Big)^2\sigma_j^{k_1}(X^{n}_{\frac{[n u]}{n}})\sigma_j^{k_2}(X^{mn}_{\frac{[mn u]}{mn}})\d u\d s\nonumber\\&=n^{2H}\int_{0}^{t}\int_{0}^{\frac{[                      mns]}{mn}}\Big(K(v+\delta_{(mn,s)})-K(v)\Big)^2\sigma_j^{k_1}(X^{n}_{\frac{[\frac{[mns]}{m}]+[-nv ]}{n}})\sigma_j^{k_2}(X^{mn}_{\frac{[mns]+[-mnv]}{mn}})\d v\d s\nonumber\\&=\frac{1}{G}\int_{0}^{t}(n\delta_{(mn,s)})^{2H}\int_{0}^{\frac{[mns]}{mn\delta_{(mn,s)}}}|\mu(r,1)|^2\sigma_j^{k_1}(X^{n}_{\frac{[\frac{[mns]}{m}]+[-n\delta_{(mn,s)}r ]}{n}})\sigma_j^{k_2}(X^{mn}_{\frac{[[mns]+[-mn\delta_{(mn,s)} r]}{mn}})\d r\d s,\nonumber
\end{align}	
where $\delta_{(mn,s)}=s-\frac{[mn s]}{mn}$.	

Next, we shall   Lemmas \ref{limit distribution-1}, \ref{limit distribution-0_m} and \ref{limit distribution-2} in the  later section 
  \ref{limit theorems} to the evaluaton  of $\textbf{(1,3,3,1)}_{j,j}^{ mn,n,k}$, $\textbf{(1,3,3,2)}_{j,j}^{ mn,n,k}$ and $\textbf{(1,3,3,3)}_{j,j}^{ mn,n,k}$, respectively. To this end, we show  (as $n\to \infty$)
\begin{align}\label{back need-1}
	 \frac{1}{G}\int_{0}^{x_i}|\mu(r,1)|^2&\sigma_j^{k_1}(X^{n}_{y_i}) \sigma_j^{k_2}(X^{mn}_{z_i})\d r\nonumber\\
	&\xrightarrow {  L^2(\d u\otimes \d \PP) } \frac{1}{G}\sigma_j^{k_1}(X_{s})\sigma_j^{k_2}(X_{s})\int_{0}^{\infty}|\mu(r,1)|^2\d r\,, 
	\end{align}
where $i=1,2,3$, 
\[
\begin{split} 
(x_1,y_1,z_1)=&(\frac{[ns]}{n\delta_{(mn,n,s)}},\frac{[ns]+[-n\delta_{(mn,n,s)}r ]}{n},\frac{[m[ns]-mn\delta_{(mn,n,s)} r]}{mn})\,, \\(x_2,y_2,z_2)=&(\frac{[ns]}{n\delta_{(n,s)}},\frac{[ns]+[-n\delta_{(n,s)}r ]}{n},\frac{[m[ns]-mn\delta_{(n,s)} r]}{mn})\,, \\
(x_3,y_3,z_3)=&(\frac{[mns]}{mn\delta_{(mn,s)}},\frac{[\frac{[mns]}{m}]+[-n\delta_{(mn,s)}r ]}{n},\frac{[[mns]+[-mn\delta_{(mn,s)} r]}{mn})\,.
\end{split}
\]
 In the sequel,  we shall only   prove the case of $(x_1,y_1,z_1)$. Note that  in this case  the right-hand side is a continuous function of $s$. It follows from Fubini's theorem and Minkowski's inequality that
\begin{align}
	\frac{1}{G^2}&\E\Bigg[\int_{0}^{t}\Big|\int_{0}^{\infty}|\mu(r,1)|^2\Big(\mathbb{I}_{(0,\frac{[ns]}{n\delta_{(mn,n,s)}})}(r)\sigma_j^{k_1}(X^{n}_{\frac{[ns]+[-n\delta_{(mn,n,s)}r ]}{n}})\sigma_j^{k_2}(X^{mn}_{\frac{[m[ns]-mn\delta_{(mn,n,s)} r]}{mn}})\nonumber\\&\quad-\sigma_j^{k_1}(X_{s})\sigma_j^{k_2}(X_{s})\Big)\d r\Big|^2\d s\Bigg]\nonumber\\&\leq C\int_{0}^{t} \Big(\int_{0}^{\infty}|\mu(r,1)|^2\Big\|\mathbb{I}_{(0,\frac{[ns]}{n\delta_{(mn,n,s)}})}(r)\sigma_j^{k_1}(X^{n}_{\frac{[ns]+[-n\delta_{(mn,n,s)}r ]}{n}})\sigma_j^{k_2}(X^{mn}_{\frac{[m[ns]-mn\delta_{(mn,n,s)} r]}{mn}})\nonumber\\&\quad-\sigma_j^{k_1}(X_{s})\sigma_j^{k_2}(X_{s})\Big\|_{L^2}\d r\Big)^2\d s.\nonumber
\end{align}
By Minkowski's inequality, it holds that
\begin{align}\label{sigma-1}
	&\Big\|\mathbb{I}_{(0,\frac{[ns]}{n\delta_{(mn,n,s)}})}(r)\sigma_j^{k_1}(X^{n}_{\frac{[ns]+[-n\delta_{(mn,n,s)}r ]}{n}})\sigma_j^{k_2}(X^{mn}_{\frac{[m[ns]-mn\delta_{(mn,n,s)} r]}{mn}})-\sigma_j^{k_1}(X_{s})\sigma_j^{k_2}(X_{s})\Big\|_{L^2}\nonumber\\&\leq \Big\|\sigma_j^{k_1}(X^{n}_{\frac{[ns]+[-n\delta_{(mn,n,s)}r ]}{n}})\sigma_j^{k_2}(X^{mn}_{\frac{[m[ns]-mn\delta_{(mn,n,s)} r]}{mn}})-\sigma_j^{k_1}(X_{s})\sigma_j^{k_2}(X_{s})\Big\|_{L^2}\mathbb{I}_{(0,\frac{[ns]}{n\delta_{(mn,n,s)}})}(r)\nonumber\nonumber\\&\quad+\|\sigma_j^{k_1}(X_{s})\sigma_j^{k_2}(X_{s})\|_{L^2}\mathbb{I}_{(\frac{[ns]}{n\delta_{(mn,n,s)}},\infty)}(r),
\end{align}
with the last term vanishing as $n\to\infty$. For the first term, by the assumption $H_{b,\sigma}$, the Cauchy-Schwarz inequality, Lemmas \ref{bound of X},    \ref{bound of X n},    \ref{continuous of X n} and  \ref{rate X-X n}, we have
\begin{align}\label{sigma-2}
	&\Big\|\sigma_j^{k_1}(X^{n}_{\frac{[ns]+[-n\delta_{(mn,n,s)}r ]}{n}})\sigma_j^{k_2}(X^{mn}_{\frac{[m[ns]-mn\delta_{(mn,n,s)} r]}{mn}})-\sigma_j^{k_1}(X_{s})\sigma_j^{k_2}(X_{s})\Big\|_{L^2}\nonumber\\&\leq \Big\|\big(\sigma_j^{k_1}(X^{n}_{\frac{[ns]+[-n\delta_{(mn,n,s)}r ]}{n}})-\sigma_j^{k_1}(X_{s})\big)\sigma^{k_2}_j(X^{mn}_{\frac{[m[ns]-mn\delta_{(mn,n,s)} r]}{mn}})\Big\|_{L^2}\nonumber\\&\quad+\Big\|\big(\sigma_j^{k_2}(X^{mn}_{\frac{[m[ns]-mn\delta_{(mn,n,s)}r ]}{mn}})-\sigma_j^{k_2}(X_{s})\big)\sigma^{k_1}_{j}(X_{s})\Big\|_{L^2}\nonumber\\&\leq \Big\|\sigma_j^{k_1}(X^{n}_{\frac{[ns]+[-n\delta_{(mn,n,s)}r ]}{n}})-\sigma_j^{k_1}(X_{s})\Big\|_{L^4}\|\sigma^{k_2}_j(X^{mn}_{\frac{[m[ns]-mn\delta_{(mn,n,s)} r]}{mn}})\|_{L^4}\nonumber\\&\quad+\Big\|\sigma_j^{k_2}(X^{mn}_{\frac{[m[ns]-mn\delta_{(mn,n,s)}r ]}{mn}})-\sigma_j^{k_2}(X_{s})\Big\|_{L^4}\|\sigma^{k_1}_{j}(X_{s})\|_{L^4}\nonumber\\&\leq C\Big\|X^{n}_{\frac{[ns]+[-n\delta_{(mn,n,s)}r ]}{n}}-X^{n}_s+X^{n}_s-X_s\Big\|_{L^4}+C\Big\|X^{mn}_{\frac{[m[ns]-mn\delta_{(mn,n,s)}r ]}{mn}}-X^{mn}_s+X^{mn}_s-X_s\Big\|_{L^4}\nonumber\\&\leq C\|X^{n}_{\frac{[ns]+[-n\delta_{(mn,n,s)}r ]}{n}}-X^{n}_s\|_{L^4}+ C\|X^{mn}_{\frac{[m[ns]-mn\delta_{(mn,n,s)}r ]}{mn}}-X^{mn}_s\|_{L^4}\nonumber\\&\quad+C\|X^{n}_s-X_s\|_{L^4}+C\|X^{mn}_s-X_s\|_{L^4}\nonumber\\&\leq C\Big(\Big|\frac{[ns]+[-n\delta_{(mn,n,s)}r ]}{n}-s]\Big|^H+\Big|\frac{[m[ns]-mn\delta_{(mn,n,s)}r ]}{mn}-s\Big|^H+n^{-H}+ {(mn)}^{-H}\Big)\nonumber\\&\leq C(1+r^H)n^{-H}+C(1+r^H){(mn)}^{-H}\to 0 \quad \text{as $n\to\infty$}.
\end{align}
Consequently, the right-hand side of \eqref{sigma-1} tends to zero by applying the dominated convergence theorem to  the integral   of $\d u$ and $\d r$ respectively.

Hence, Lemma \ref{limit distribution-1} gives that
\begin{align}
	\textbf{(1,3,3,1)}_{j,j}^{ mn,n,k}\xrightarrow {  L^2 } \frac{g^H_m}{G}\int_{0}^{\infty}\mu(r,1)^2\d r\int_{0}^{t}\sigma^{k_1}_j(X_{s})\sigma^{k_2}_j(X_{s})\d s,\nonumber
\end{align}
where $g^{H}_m=\sum_{j=0}^{m-1}\frac{j^{2H}}{m^{2H+1}}$.\\
Lemma \ref{limit distribution-0_m} gives that
\begin{align}
	\textbf{(1,3,3,2)}_{j,j}^{ mn,n,k}\xrightarrow {  L^2 } \frac{1}{(2H+1)G}\int_{0}^{\infty}\mu(r,1)^2\d r\int_{0}^{t}\sigma^{k_1}_j(X_{s})\sigma^{k_2}_j(X_{s})\d s.\nonumber
\end{align}
Lemma \ref{limit distribution-2} gives that
\begin{align}
	\textbf{(1,3,3,3)}_{j,j}^{ mn,n,k}\xrightarrow { L^2 } \frac{1}{(2H+1)m^{2H}G}\int_{0}^{\infty}\mu(r,1)^2\d r\int_{0}^{t}\sigma^{k_1}_j(X_{s})\sigma^{k_2}_j(X_{s})\d s.\nonumber
\end{align}
On the other hand, a direct computation yields that
\begin{align}
	\textbf{(2,3,4,1)}_{j,j}^{ mn,n,k}=\frac{n^{2H}}{2HG}\int_{0}^{t}E_s^{m,n,j}(n\delta_{(mn,n,s)})^{2H}\d s,\nonumber
\end{align}
\begin{align}
	\textbf{(2,3,4,2)}_{j,j}^{ mn,n,k}=\frac{n^{2H}}{2HG}\int_{0}^{t}E_s^{m,n,j}\big((n\delta_{(n,s)})^{2H}-(n\delta_{(mn,s)})^{2H}\big)\d s.\nonumber
\end{align}
By Lemmas \ref{rate X-X n}, \ref{limit distribution-0_m},\ref{limit distribution-1} and \ref{limit distribution-2}, we have that
\begin{align}
	\textbf{(2,3,4,1)}_{j,j}^{ mn,n,k}\xrightarrow {L^2 } \frac{g^H_m}{2HG}\int_{0}^{t}\sigma^{k_1}_j(X_{s})\sigma^{k_2}_j(X_{s})\d s,\nonumber
\end{align}	
and
\begin{align}
	\textbf{(2,3,4,2)}_{j,j}^{ mn,n,k}\xrightarrow {L^2 } \frac{1}{2HG}\frac{m^{2H}-1}{(2H+1)m^{2H}}\int_{0}^{t}\sigma^{k_1}_j(X_{s})\sigma^{k_2}_j(X_{s})\d s.\nonumber
\end{align}
 
Therefore 
\begin{align}
	\textbf{(1,3,3)}_{j,j}^{ mn,n,k}+\textbf{(2,3,4)}_{j,j}^{ mn,n,k}&\xrightarrow {  L^2 } \frac{1}{2G}\Big[\frac{m^{2H}+1}{(2H+1)m^{2H}}-g_m^H\Big]\int_{0}^{\infty}\mu(r,1)^2\d r\int_{0}^{t}\sigma^{k_1}_j(X_{s})\sigma^{k_2}_j(X_{s})\d s\nonumber\\&\quad+\frac{1}{2G}\Big[\frac{m^{2H}-1}{2H(2H+1)m^{2H}}-\frac{g_m^H}{2H}\Big]\int_{0}^{t}\sigma^{k_1}_j(X_{s})\sigma^{k_2}_j(X_{s})\d s.\nonumber
\end{align}

To  study  the term $\textbf{(1,3,4)}_{j,l}^{mn,n,k}$  we   let
\begin{align}\label{def of A_2}
	A_{2,s}^{(m,n,j,l)}=n^{2H}\Big(\int_{0}^{\frac{[n s]}{n}}\Delta K(n,s,u)\sigma_j^{k_1}(X^{n}_{\frac{[n u]}{n}})\d W^j_u\Big)\Big(\int_{\frac{[ns]}{n}}^{\frac{[mn s]}{mn}}\Delta K(mn,s,u)\sigma^{k_2}_j(X^{mn}_{\frac{[mn u]}{mn}})\d W^l_u\Big).
\end{align}
By Fubini's theorem we can rewrite $\E[|\textbf{(1,3,4)}_{j,l}^{ mn,n,k}|^2]$ as
\begin{align}
	\E[|\textbf{(1,3,4)}_{j,l}^{ mn,n,k}|^2]&=\E\Big[\int_{0}^{t}\int_{0}^{t}A_{2,s}^{(m,n,j,l)}A_{2,v}^{(m,n,j,l)}\d v \d s\Big]\nonumber\\&=2\E\Big[\int_{0}^{t}\int_{0}^{\frac{[ns]}{n}}A_{2,s}^{(m,n,j,l)}A_{2,v}^{(m,n,j,l)}\d v \d s\Big]\nonumber\\&\quad+2\E\Big[\int_{0}^{t}\int_{\frac{[ns]}{n}}^{s}A_{2,s}^{(m,n,j,l)}A_{2,v}^{(m,n,j,l)}\d v \d s\Big]\nonumber\\&:=2(\textbf{(1,3,4,1)}_{j,l}^{ mn,n,k}+\textbf{(1,3,4,2)}_{j,l}^{ mn,n,k}).\nonumber
\end{align}
For the above first term, we use  the tower property and Fubini's theorem 
to obtain 
\begin{align}
	\textbf{(1,3,4,1)}_{j,l}^{ mn,n,k}&=\E\Big[\int_{0}^{t}\int_{0}^{\frac{[ns]}{n}}A_{2,s}^{(m,n,j,l)}A_{2,v}^{(m,n,j,l)}\d v \d s\Big]\nonumber\\&=\int_{0}^{t}\int_{0}^{\frac{[ns]}{n}}\E\Big[A_{2,s}^{(m,n,j,l)}A_{2,v}^{(m,n,j,l)}\Big]\d v \d s\nonumber\\&=\int_{0}^{t}\int_{0}^{\frac{[ns]}{n}}\E\Big[\E\Big[\int_{\frac{[ns]}{n}}^{\frac{[mn s]}{mn}}\Delta K(mn,s,u)\sigma_l^{k_2}(X^{mn}_{\frac{[mn u]}{mn}})\d W^l_u|\mathcal{F}_{\frac{[ns]}{n}}\Big]\nonumber\\\quad&\cdot\int_{0}^{\frac{[n s]}{n}}\Delta K(mn,n,s,u)\sigma_j^{k_1}(X^{mn}_{\frac{[mn u]}{mn}})\d W^j_uA_{2,v}^{(m,n,j,l)}\Big]\d v \d s=0.\nonumber
\end{align}

Now we study  the term $\textbf{(1,3,4,2)}_{j,l}^{ mn,n,k}$.  Similar to \eqref{esti of stochastic integral_1}  and  by Lemma \ref{est-3_m}, we have that
\begin{align}\label{esti of stochastic integral-1}
	\Big\|\int_{\frac{[nv]}{n}}^{u}\Delta K(mn,s,r)\sigma_l^{k_2}(X^{mn}_{\frac{[mn r]}{mn}}) \d W_r^l\Big\|_{L^m}\leq Cn^{-H}.
\end{align}

Then by the Cauchy-Schwarz inequality, $(26)$ of \cite{FU} and \eqref{esti of stochastic integral-1} we have
\begin{align}
	&\E\Big[A_{2,s}^{(m,n,j,l)}A_{2,v}^{(m,n,j,l)}\Big]\nonumber\\&\leq \E\Big[|A_{2,s}^{(m,n,j,l)}|^2\Big]^{\frac{1}{2}}\E\Big[|A_{2,v}^{(m,n,j,l)}|^2\Big]^{\frac{1}{2}}\nonumber\\&\leq n^{4H}\Big\|\int_{0}^{\frac{[n s]}{n}}\Delta K(n,s,u)\sigma_j^{k_1}(X^{n}_{\frac{[n u]}{n}})\d W^j_u\Big\|_{L^4}\Big\|\int_{\frac{[ns]}{n}}^{\frac{[mn s]}{mn}}\Delta K(mn,s,u)\sigma_l^{k_2}(X^{mn}_{\frac{[mn u]}{mn}})\d W^l_u\Big\|_{L^4}\nonumber\\&\quad\cdot \Big\|\int_{0}^{\frac{[n v]}{n}}\Delta K(n,v,u)\sigma_j^{k_1}(X^{n}_{\frac{[n u]}{n}})\d W^j_u\Big\|_{L^4}\Big\|\int_{\frac{[nv]}{n}}^{\frac{[mn v]}{mn}}\Delta K(mn,v,u)\sigma_l^{k_2}(X^{mn}_{\frac{[mn u]}{mn}})\d W^l_u\Big\|_{L^4}\nonumber\\&\leq C.\nonumber
\end{align}
The bounded convergence theorem yields
\begin{align}
	0\leq \lim_{n\to\infty}\textbf{(1,3,4,2)}_{j,l}^{ mn,n,k}=\int_{0}^{t}\int_{0}^{t}\lim_{n\to\infty}\mathbb{I}_{(\frac{[ns]}{n},s)}\E\Big[A_{2,s}^{(m,n,j,l)}A_{2,v}^{(m,n,j,l)}\Big]\d v\d s=0\nonumber, 
\end{align}
To bound  the term $\textbf{(2,3,1)}_{j,l}^{ mn,n,k}$ we let
\begin{align}\label{def of A_3}
	A_{3,s}^{(m,n,j,l)}=n^{2H}\Big(\sigma^{k_1}_j(X^{n}_{\frac{[ns]}{n}})\int_{\frac{[n s]}{n}}^{\frac{[mns]}{mn}}K(s-u)\d W^j_u\Big)\Big(\int_{0}^{\frac{[n s]}{n}}\Delta K(mn,s,u)\sigma^{k_2}_j(X^{mn}_{\frac{[mn u]}{mn}})\d W^l_u\Big).
\end{align}
By Fubini's theorem we can rewrite $\E[|\textbf{(2,3,1)}_{j,l}^{ mn,n,k}|^2]$ as
\begin{align}
	\E[|\textbf{(2,3,1)}_{j,l}^{ mn,n,k}|^2]&=\E\Big[\int_{0}^{t}\int_{0}^{t}A_{3,s}^{(m,n,j,l)}A_{3,v}^{(m,n,j,l)}\d v \d s\Big]\nonumber\\&=2\E\Big[\int_{0}^{t}\int_{0}^{\frac{[ns]}{n}}A_{3,s}^{(m,n,j,l)}A_{3,v}^{(m,n,j,l)}\d v \d s\Big]\nonumber\\&\quad+2\E\Big[\int_{0}^{t}\int_{\frac{[ns]}{n}}^{s}A_{3,s}^{(m,n,j,l)}A_{3,v}^{(m,n,j,l)}\d v \d s\Big]\nonumber\\&:=2(\textbf{(2,3,1,1)}_{j,l}^{ mn,n,k}+\textbf{(2,3,1,2)}_{j,l}^{ mn,n,k}).\nonumber
\end{align}
We also study the above two terms separately.  Notice 
\begin{align}
	\E[A_{3,s}^{(m,n,j,l)}|\mathcal{F}_{\frac{[ns]}{n}}]=n^{H}\int_{0}^{\frac{[n s]}{n}}\Delta K(mn,s,u)\sigma^{k_2}_j(X^{mn}_{\frac{[mn u]}{mn}})\d W^l_u\E\Big[n^{H}\sigma^{k_1}_j(X^{n}_{\frac{[ns]}{n}})\int_{\frac{[n s]}{n}}^{\frac{[mns]}{mn}}K(s-u)\d W^j_u|\mathcal{F}_{\frac{[ns]}{n}}\Big]=0,\nonumber
\end{align}
and  notice that  for $v\in (0,\frac{[ns]}{n})$
$$\frac{[mnv]}{mn}< \frac{[mn \frac{[ns]}{n}]}{mn}=\frac{[ns]}{n}.$$
By tower property and Fubini's theorem, we have
\begin{align}
	\textbf{(2,3,1,1)}_{j,l}^{ mn,n,k}&=\int_{0}^{t}\int_{0}^{\frac{[mns]}{mn}}\E\Big[A_{3,s}^{(m,n,j,l)}A_{3,v}^{(m,n,j,l)}\Big]\d v \d s\nonumber\\&=\int_{0}^{t}\int_{0}^{\frac{[mns]}{mn}}\E\Big[\sigma^{k_1}_j(X^{n}_{\frac{[ns]}{n}})\int_{\frac{[n s]}{n}}^{\frac{[mns]}{mn}}K(s-u)\d W^j_u|\mathcal{F}_{\frac{[mns]}{mn}}\Big]\nonumber\\\quad&\cdot n^{H}\int_{0}^{\frac{[n s]}{n}}\Delta K(mn,s,u)\sigma^{k_2}_j(X^{mn}_{\frac{[mn u]}{mn}})\d W^l_uA_{3,v}^{(m,n,j,l)}\Big]\d v \d s=0.\nonumber
\end{align}
For the term $\textbf{(2,3,1,2)}_{j,l}^{ mn,n,k}$, by the Cauchy-Schwarz inequality, \eqref{esti of stochastic integral_1} and \eqref{esti of stochastic integral-2} we have
\begin{align}
	&\E\Big[A_{3,s}^{(m,n,j,l)}A_{3,v}^{(m,n,j,l)}\Big]\nonumber\\&\leq \E\Big[|A_{3,s}^{(m,n,j,l)}|^2\Big]^{\frac{1}{2}}\E\Big[|A_{3,v}^{(m,n,j,l)}|^2\Big]^{\frac{1}{2}}\nonumber\\&\leq n^{2H}\Big\|\int_{0}^{\frac{[n s]}{n}}\Delta K(mn,s,u)\sigma_j^{k_2}(X^{mn}_{\frac{[mn u]}{mn}})\d W^j_u\Big\|_{L^4}\Big\|\int_{\frac{[mns]}{mn}}^{s} K(s-u)\sigma_l^{k_1}(X^{n}_{\frac{[n u]}{n}})\d W^l_u\Big\|_{L^4}\nonumber\\&\quad\cdot \Big\|\int_{0}^{\frac{[n v]}{n}}\Delta K(mn,v,u)\sigma_j^{k_2}(X^{mn}_{\frac{[mn u]}{mn}})\d W^j_u\Big\|_{L^4}\Big\|\int_{\frac{[mnv]}{mn}}^{v} K(v-u)\sigma_l^{k_1}(X^{n}_{\frac{[n u]}{n}})\d W^l_u\Big\|_{L^4}\nonumber\\&\leq C.\nonumber
\end{align}
The bounded convergence theorem yields
\begin{align}
	0\leq \lim_{n\to\infty}\textbf{(2,3,1,2)}_{j,l}^{ mn,n,k}=\int_{0}^{t}\int_{0}^{t}\lim_{n\to\infty}\mathbb{I}_{(\frac{[ns]}{n},s)}\E\Big[A_{3,s}^{(m,n,j,l)}A_{3,v}^{(m,n,j,l)}\Big]\d v\d s=0\,. \nonumber  
\end{align}
From above we conclude that $\lim_{n\to\infty}\textbf{(2,3,1)}_{j,l}^{ mn,n,k}=0$ in the sense of $L^2$.

For the term $\textbf{(2,3,2)}_{j,l}^{ mn,n,k}$, denote
\begin{align}
	A_{4,s}^{(m,n,j,l)}=n^{2H}\int_{\frac{[ns]}{n}}^{\frac{[mns]}{mn}}\Big(\int_{\frac{[ns]}{n}}^{u}\Delta K(mn,s,r) \sigma_l^{k_2}(X^{mn}_{\frac{[mn r]}{mn}})\d W^l_r\Big)K(s-u) \sigma_j^{k_1}(X^{n}_{\frac{[n u]}{n}})\d W^j_u.\nonumber
\end{align}
By Fubini's theorem we have
\begin{align}
	\E[|\textbf{(2,3,2)}_{j,l}^{ mn,n,k}|^2]&=\E\Big[\int_{0}^{t}\int_{0}^{t}A_{4,s}^{(m,n,j,l)}A_{4,v}^{(m,n,j,l)}\d v\d s\Big]\nonumber\\&=2\int_{0}^{t}\int_{0}^{s}\E\big[A_{4,s}^{(m,n,j,l)}A_{4,v}^{(m,n,j,l)}\big]\d v\d s.\nonumber
\end{align}
We now need to compare $\frac{[ns]}{n}$ and $\frac{[mnv]}{mn}$ for $v<s$. First, we consider the case $v\leq \frac{[ns]}{n}$.  Applying Fubini's theorem again and tower property, we have that
\begin{align}
	\E\big[A_{4,s}^{(m,n,j,l)}A_{4,v}^{(m,n,j,l)}\big]=n^{4H}\E\Big[\E\Big[A_{4,s}^{(m,n,j,l)}|\mathcal{F}_{\frac{[ns]}{n}}\Big]A_{4,v}^{(m,n,j,l)}\Big]=0.\nonumber
\end{align}
Now we consider the case   $\frac{[nv]}{n}\leq \frac{[mnv]}{mn}\leq v\leq \frac{[ns]}{n}\leq \frac{[mns]}{mn}\leq s$. 
By \eqref{dcp of stochastic integral} and Fubini's theorem, we have
\begin{align}\label{double of A 4}
	&\E\big[A_{4,s}^{(m,n,j,l)}A_{4,v}^{(m,n,j,l)}\big]\nonumber\\&=n^{4H}\E\Big[\int_{\frac{[ns]}{n}}^{\frac{[mnv]}{mn}}\Big(\int_{\frac{[ns]}{n}}^{u}\Delta K(mn,s,r)\sigma_l^{k_2}(X^{mn}_{\frac{[mn r]}{mn}})\d W^l_r\Big)\Big(\int_{\frac{[nv]}{n}}^{u}\Delta K(mn,v,r)\sigma_l^{k_2}(X^{mn}_{\frac{[mn r]}{mn}})\d W^l_r\Big)\nonumber\\&\quad\cdot K(s-u) K(v-u) |\sigma_j^{k_1}(X^{n}_{\frac{[n u]}{n}})|^2\d u\Big]\nonumber\\&=n^{4H}\int_{\frac{[ns]}{n}}^{\frac{[mnv]}{mn}} K(s-u) K(v-u)\cdot E^{(2)}_u\d u,
\end{align}
where
\begin{align}
	E^{(2)}_u:=\E\Big[\Big(\int_{\frac{[ns]}{n}}^{u}\Delta K(mn,s,r)\sigma_l^{k_2}(X^{mn}_{\frac{[mn r]}{mn}})\d W^l_r\Big)\Big(\int_{\frac{[nv]}{n}}^{u}\Delta K(mn,v,r)\sigma_l^{k_2}(X^{mn}_{\frac{[mn r]}{mn}})\d W^l_r\Big)|\sigma_j^{k_1}(X^{n}_{\frac{[n u]}{n}})|^2\Big].\nonumber
\end{align}
Since for any $m\geq 2, \frac{[ns]}{n}<u\leq \frac{[mns]}{mn}$ and $s\leq T$, it holds
\begin{align}\label{esti of stochastic integral_3}
	&\Big\|\int_{\frac{[ns]}{n}}^{u}\Delta K(mn,s,r)\sigma_l^{k_2}(X^{mn}_{\frac{[mn r]}{mn}}) \d W_r^l\Big\|_{L^m}\nonumber\\&\leq C \Big\|\int_{\frac{[ns]}{n}}^{u}\Big|\Delta K(mn,s,r)\sigma_l^{k_2}(X^{mn}_{\frac{[mn r]}{mn}})\Big|^2\d r\Big\|^{\frac{1}{2}}_{L^{m/2}}\nonumber\\&\leq C\Big\|\int_{\frac{[ns]}{n}}^{\frac{[mns]}{mn}}\Big(K(s-r)-K(\frac{[mns]}{mn}-r)\Big)^2|\sigma_l^{k_2}(X^{mn}_{\frac{[mn r]}{mn}})|^2\d r\Big\|^{\frac{1}{2}}_{L^{m/2}}\nonumber\\&\leq C\Big(\int_{0}^{\frac{[mns]}{mn}}\Big(K(s-r)-K(\frac{[mns]}{mn}-r)\Big)^2\|\sigma_l^{k_2}(X^{mn}_{\frac{[mn r]}{mn}})|^2\|_{L^{m/2}}\d r\Big)^{1/2}\nonumber\\&\leq Cn^{-H},
\end{align}
where Lemma \ref{est-3_m}, the boundness of $\sigma$, BDG's inequality, Minkowski's  inequality are used.

Then by the Cauchy-Schwarz inequality, Lemma \ref{bound of X n}, the boundness of $\sigma$ and \eqref{esti of stochastic integral_3}, we have
\begin{align}
	\E[E^{(2)}_{u}]&\leq \Big\|\int_{\frac{[ns]}{n}}^{u}\Delta K(mn,s,r)\sigma_l^{k_2}(X^{mn}_{\frac{[mn r]}{mn}})\d W^l_r\Big\|_{L^4}\Big\| \int_{\frac{[nv]}{n}}^{u}\Delta K(mn,v,r)\sigma_l^{k_2}(X^{mn}_{\frac{[mn r]}{mn}})\d W^l_r\Big\|_{L^4}\|\sigma_j^{k_2}(X^{n}_{\frac{[n u]}{n}})\|^2_{L^4}\nonumber\\&\leq Cn^{-2H}.\nonumber
\end{align}
Lemma \ref{esti of C} and \eqref{double of A 4} give that
$$\Big|\E\big[A_{4,s}^{(m,n,j,l)}A_{4,v}^{(m,n,j,l)}\big]\Big|\leq Cn^{2H}\int_{\frac{[ns]}{n}}^{\frac{[mnv]}{mn}}K(s-u)K(v-u)\d u=0.$$
Applying the dominated convergence theorem  with respect to $\d v\otimes \d s$, we have $$\textbf{(2,3,2)}_{j,l}^{ mn,n,k}\xrightarrow {  L^2} 0.$$
For the term $\textbf{(2,3,3)}_{j,l}^{ mn,n,k}$, we set
\begin{align}
	A_{5,s}^{(m,n,j,l)}=n^{2H}\int_{\frac{[ns]}{n}}^{\frac{[mns]}{mn}}\Big(\int_{\frac{[ns]}{n}}^{u}K(s-r)\sigma_j^{k_1}(X^{n}_{\frac{[n r]}{n}})\d W^j_r\Big)\Delta K(mn,s,u) \sigma_l^{k_2}(X^{mn}_{\frac{[mn u]}{mn}})\d W^l_u.\nonumber
\end{align}
By Fubini's theorem we  see
\begin{align}
	\E[|\textbf{(2,3,3)}_{j,l}^{ mn,n,k}|^2]&=\E\Big[\int_{0}^{t}\int_{0}^{t}A_{5,s}^{(m,n,j,l)}A_{5,v}^{(m,n,j,l)}\d v\d s\Big]\nonumber\\&=2\int_{0}^{t}\int_{0}^{s}\E\big[A_{5,s}^{(m,n,j,l)}A_{5,v}^{(m,n,j,l)}\big]\d v\d s.\nonumber
\end{align}
We study the above term in two cases. First, assume   $v\leq \frac{[ns]}{n}$. Similar to $\textbf{(2,3,2)}_{j,l}^{ mn,n,k}$, we have
\begin{align}
	\E\big[A_{5,s}^{(m,n,j,l)}A_{5,v}^{(m,n,j,l)}\big]=n^{4H}\E\Big[\E\Big[A_{5,s}^{(m,n,j,l)}|\mathcal{F}_{\frac{[ns]}{n}}\Big]A_{5,v}^{(m,n,j,l)}\Big]=0.\nonumber
\end{align}
Now we consider the case  $\frac{[ns]}{n}<v\leq s$. By \eqref{dcp of stochastic integral} and Fubini's theorem, we have
\begin{align}\label{double of A 5}
	&\E\big[A_{5,s}^{(m,n,j,l)}A_{5,v}^{(m,n,j,l)}\big]\nonumber\\&=n^{4H}\E\Big[\int_{\frac{[ns]}{n}}^{\frac{[mnv]}{mn}}\Big(\int_{\frac{[ns]}{n}}^{u}K(s-r)\sigma_j^{k_1}(X^{n}_{\frac{[n r]}{n}})\d W^j_r\Big)\Big(\int_{\frac{[nv]}{n}}^{u}K(v-r)\sigma_j^{k_1}(X^{n}_{\frac{[n r]}{n}})\d W^j_r\Big)\nonumber\\&\quad\cdot\Delta K(mn,s,u)\Delta K(mn,v,u) |\sigma_j^{k_2}(X^{mn}_{\frac{[mn u]}{mn}})|^2\d u\Big]\nonumber\\&=n^{4H}\int_{\frac{[ns]}{n}}^{\frac{[mnv]}{mn}}\Delta K(mn,s,u)\Delta K(mn,v,u)\cdot E^{(2)}_u\d u,
\end{align}
where
\begin{align}
	E^{(3)}_u:=\E\Big[\Big(\int_{\frac{[ns]}{n}}^{u}K(s-r)\sigma_l^{k_1}(X^{n}_{\frac{[n r]}{n}})\d W^j_r\Big)\Big(\int_{\frac{[nv]}{n}}^{u}K(v-r)\sigma_l^{k_1}(X^{n}_{\frac{[n r]}{n}})\d W^j_r\Big)|\sigma_j^{k_2}(X^{mn}_{\frac{[mn u]}{mn}})|^2\Big].\nonumber
\end{align}
Similar to \eqref{esti of stochastic integral_1}, for $u<s\in[0,T]$ by Lemma \ref{est-3}, we have that
\begin{align}\label{esti of stochastic integral_4}
	\Big\|\int_{\frac{[ns]}{n}}^{u} K(s-r)\sigma_l^{k_1}(X^{n}_{\frac{[n r]}{n}}) \d W_r^l\Big\|_{L^m}\leq Cn^{-H}.
\end{align}
Then by the Cauchy-Schwarz inequality, Lemma \ref{bound of X n}, the boundness of $\sigma$ and \eqref{esti of stochastic integral_4}, we have
\begin{align}
	\E[E^{(3)}_{u}]&\leq \Big\|\int_{\frac{[ns]}{n}}^{u}K(s-r)\sigma_j^{k_1}(X^{n}_{\frac{[n r]}{n}})\d W^j_r\Big\|_{L^4}\Big\| \int_{\frac{[nv]}{n}}^{u}K(v-r)\sigma_j^{k_1}(X^{n}_{\frac{[n r]}{n}})\d W^j_r\Big\|_{L^4}\|\sigma_j^{k_2}(X^{mn}_{\frac{[mn u]}{mn}})\|^2_{L^4}\nonumber\\&\leq Cn^{-2H}.\nonumber
\end{align}
Lemma \ref{esti of B} and \eqref{double of A 5} give that
$$\Big|\E\big[A_{5,s}^{(m,n,j,l)}A_{5,v}^{(m,n,j,l)}\big]\Big|\leq Cn^{2H}\int_{\frac{[ns]}{n}}^{\frac{[mnv]}{mn}}\Delta K(mn,s,u)\Delta K(mn,v,u)\d u=0.$$
Applying the dominated convergence theorem  with respect to $\d v\otimes \d s$, we have $$\textbf{(2,3,3)}_{j,l}^{mn,n,k}\xrightarrow {  L^2} 0.$$	
Finally,  we consider  the term $\textbf{(2,3,5)}_{j,l}^{ mn,n,k}$.  Let
\begin{align}\label{def of A_6}
	A_{6,s}^{(m,n,j,l)}=n^{2H}\Big(\sigma^{k_1}_j(X^{n}_{\frac{[ns]}{n}})\int_{\frac{[mn s]}{mn}}^{s}K(s-u)\d W^j_u\Big)\Big(\int_{0}^{\frac{[mn s]}{mn}}\Delta K(mn,s,u)\sigma^{k_2}_j(X^{mn}_{\frac{[mn u]}{mn}})\d W^l_u\Big).
\end{align}
Then by Fubini's theorem we can rewrite $\E[|\textbf{(2,3,5)}_{j,l}^{ mn,n,k}|^2]$ as
\begin{align}
	\E[|\textbf{(2,3,5)}_{j,l}^{ mn,n,k}|^2]&=\E\Big[\int_{0}^{t}\int_{0}^{t}A_{6,s}^{(m,n,j,l)}A_{6,v}^{(m,n,j,l)}\d v \d s\Big]\nonumber\\&=2\E\Big[\int_{0}^{t}\int_{0}^{\frac{[mns]}{mn}}A_{6,s}^{(m,n,j,l)}A_{6,v}^{(m,n,j,l)}\d v \d s\Big]\nonumber\\&\quad+2\E\Big[\int_{0}^{t}\int_{\frac{[mns]}{mn}}^{s}A_{6,s}^{(m,n,j,l)}A_{6,v}^{(m,n,j,l)}\d v \d s\Big]\nonumber\\&:=2(\textbf{(2,3,5,1)}_{j,l}^{ mn,n,k}+\textbf{(2,3,5,2)}_{j,l}^{ mn,n,k}).\nonumber
\end{align}
For the term $\textbf{(2,3,5,1)}_{j,l}^{ mn,n,k}$, by tower property and Fubini's theorem, we have
\begin{align}
	\textbf{(2,3,5,1)}_{j,l}^{ mn,n,k}&=\int_{0}^{t}\int_{0}^{\frac{[mns]}{mn}}\E\Big[A_{6,s}^{(m,n,j,l)}A_{6,v}^{(m,n,j,l)}\Big]\d v \d s\nonumber\\&=\int_{0}^{t}\int_{0}^{\frac{[mns]}{mn}}\E\Big[n^{H}\int_{\frac{[mn s]}{mn}}^{s}K(s-u)\d W^j_u|\mathcal{F}_{\frac{[mns]}{mn}}\Big]\nonumber\\\quad&\cdot \sigma^{k_1}_j(X^{n}_{\frac{[ns]}{n}})n^{H}\int_{0}^{\frac{[mn s]}{mn}}\Delta K(mn,s,u)\sigma^{k_2}_j(X^{mn}_{\frac{[mn u]}{mn}})\d W^l_uA_{3,v}^{(m,n,j,l)}\Big]\d v \d s=0.\nonumber
\end{align}
For the term $\textbf{(2,3,5,2)}_{j,l}^{ mn,n,k}$, by the Cauchy-Schwarz inequality, \eqref{esti of stochastic integral_1} and \eqref{esti of stochastic integral-2} we have
\begin{align}
	&\E\Big[A_{6,s}^{(m,n,j,l)}A_{6,v}^{(m,n,j,l)}\Big]\nonumber\\&\leq \E\Big[|A_{6,s}^{(m,n,j,l)}|^2\Big]^{\frac{1}{2}}\E\Big[|A_{6,v}^{(m,n,j,l)}|^2\Big]^{\frac{1}{2}}\nonumber\\&\leq n^{2H}\Big\|\sigma^{k_1}_j(X^{n}_{\frac{[ns]}{n}})\int_{\frac{[mn s]}{mn}}^{s}K(s-u)\d W^j_u\Big\|_{L^4}\Big\|\int_{0}^{\frac{[mn s]}{mn}}\Delta K(mn,s,u)\sigma^{k_2}_j(X^{mn}_{\frac{[mn u]}{mn}})\d W^l_u\Big\|_{L^4}\nonumber\\&\quad\cdot \Big\|\sigma^{k_1}_j(X^{n}_{\frac{[nv]}{n}})\int_{\frac{[mn v]}{mn}}^{v}K(v-u)\d W^j_u\Big\|_{L^4}\Big\|\int_{0}^{\frac{[mn v]}{mn}}\Delta K(mn,v,u)\sigma^{k_2}_j(X^{mn}_{\frac{[mn u]}{mn}})\d W^l_u\Big\|_{L^4}\nonumber\\&\leq C.\nonumber
\end{align}
The bounded convergence theorem yields
\begin{align}
	0\leq \lim_{n\to\infty}\textbf{(2,3,5,2)}_{j,l}^{ mn,n,k}=\int_{0}^{t}\int_{0}^{t}\lim_{n\to\infty}\mathbb{I}_{(\frac{[ns]}{n},s)}\E\Big[A_{6,s}^{(m,n,j,l)}A_{6,v}^{(m,n,j,l)}\Big]\d v\d s=0\,, \nonumber  
\end{align}
  concluding  that $\lim_{n\to\infty}\textbf{(2,3,5)}_{j,l}^{ mn,n,k}=0$ in the sense of $L^2$.

From all the above bounds  we conclude that 
\begin{align}
	\textbf{(1,3)}_{j,l}^{ mn,n,k}+\textbf{(2,3)}_{j,l}^{ mn,n,k}&\xrightarrow {  L^2 } \frac{1}{2G}\Big[\frac{m^{2H}+1}{(2H+1)m^{2H}}-g_m^H\Big]\int_{0}^{\infty}\mu(r,1)^2\d r\int_{0}^{t}\sigma^{k_1}_j(X_{s})\sigma^{k_2}_j(X_{s})\d s\nonumber\\&\quad+\frac{1}{2G}\Big[\frac{m^{2H}-1}{2H(2H+1)m^{2H}}-\frac{g_m^H}{2H}\Big]\int_{0}^{t}\sigma^{k_1}_j(X_{s})\sigma^{k_2}_j(X_{s})\d s.\nonumber
\end{align}
Similarly, $\textbf{(3,1)}_{j,l}^{ mn,n,k}+\textbf{(3,2)}_{j,l}^{ mn,n,k}$ 
satisfies.  
\begin{align}
	\textbf{(3,1)}_{j,l}^{ mn,n,k}+\textbf{(3,2)}_{j,l}^{ mn,n,k}&\xrightarrow {  L^2 } \frac{1}{2G}\Big[\frac{m^{2H}+1}{(2H+1)m^{2H}}-g_m^H\Big]\int_{0}^{\infty}\mu(r,1)^2\d r\int_{0}^{t}\sigma^{k_1}_j(X_{s})\sigma^{k_2}_j(X_{s})\d s\nonumber\\&\quad+\frac{1}{2G}\Big[\frac{m^{2H}-1}{2H(2H+1)m^{2H}}-\frac{g_m^H}{2H}\Big]\int_{0}^{t}\sigma^{k_1}_j(X_{s})\sigma^{k_2}_j(X_{s})\d s.\nonumber
\end{align}
\subsubsection{Error analysis of $\textbf{(1,4)}_{j,l}^{mn,n,k}$ and $\textbf{(4,1)}_{j,l}^{mn,n,k}$}	For the term $\textbf{(1,4)}_{j,l}^{mn,n,k}$, denote
\begin{align}\label{def of A_7}
	A_{7,s}^{(m,n,j,l)}=n^{2H}\Big(\int_{0}^{\frac{[n s]}{n}}\Delta K(n,s,u)\sigma_j^{k_1}(X^{n}_{\frac{[n u]}{n}})\d W^j_u\Big)\Big(\sigma^{k_2}_j(X^{mn}_{\frac{[mn s]}{mn}})\int_{\frac{[mns]}{mn}}^{s}K(s-u)\d W^l_u\Big).
\end{align}
Then by Fubini's theorem we can rewrite $\E[|\textbf{(1,4)}_{j,l}^{ mn,n,k}|^2]$ as
\begin{align}
	\E[|\textbf{(1,4)}_{j,l}^{ mn,n,k}|^2]&=\E\Big[\int_{0}^{t}\int_{0}^{t}A_{7,s}^{(m,n,j,l)}A_{7,v}^{(m,n,j,l)}\d v \d s\Big]\nonumber\\&=2\E\Big[\int_{0}^{t}\int_{0}^{\frac{[mns]}{mn}}A_{7,s}^{(m,n,j,l)}A_{7,v}^{(m,n,j,l)}\d v \d s\Big]\nonumber\\&\quad+2\E\Big[\int_{0}^{t}\int_{\frac{[mns]}{mn}}^{s}A_{7,s}^{(m,n,j,l)}A_{7,v}^{(m,n,j,l)}\d v \d s\Big]\nonumber\\&:=2(\textbf{(1,4,1)}_{j,l}^{ mn,n,k}+\textbf{(1,4,2)}_{j,l}^{ mn,n,k}).\nonumber
\end{align}
For the term $\textbf{(1,4,1)}_{j,l}^{ mn,n,k}$, by tower property and Fubini's theorem, we have
\begin{align}
	\textbf{(1,4,1)}_{j,l}^{ mn,n,k}&=\E\Big[\int_{0}^{t}\int_{0}^{\frac{[mns]}{mn}}A_{7,s}^{(m,n,j,l)}A_{7,v}^{(m,n,j,l)}\d v \d s\Big]\nonumber\\&=\int_{0}^{t}\int_{0}^{\frac{[mns]}{mn}}\E\Big[A_{7,s}^{(m,n,j,l)}A_{7,v}^{(m,n,j,l)}\Big]\d v \d s\nonumber\\&=\int_{0}^{t}\int_{0}^{\frac{[mns]}{mn}}\E\Big[\E\Big[n^{H}\int_{\frac{[mns]}{mn}}^{s}K(s-u)\sigma_l^{k_2}(X^{mn}_{\frac{[mn u]}{mn}})\d W^l_u|\mathcal{F}_{\frac{[mns]}{mn}}\Big]\nonumber\\\quad&\cdot \sigma_j^{k_2}(X^{mn}_{\frac{[mn s]}{mn}}) n^{H}\int_{0}^{\frac{[n s]}{n}}\Delta K(n,s,u)\sigma_j^{k_1}(X^{n}_{\frac{[n u]}{n}})\d W^j_uA_{2,v}^{(m,n,j,l)}\Big]\d v \d s=0.\nonumber
\end{align}
To study the term $\textbf{(1,4,2)}_{j,l}^{ mn,n,k}$, similar to \eqref{esti of stochastic integral-1},  we   need some more effort.   For any $m\geq 2, v\leq s, s\leq T$, it holds
\begin{align}\label{esti of stochastic integral-2}
	&\Big\|\int_{\frac{[mnv]}{mn}}^{s}K(s-r)\sigma_l^{k_2}(X^{mn}_{\frac{[mn r]}{mn}}) \d W_r^l\Big\|_{L^m}\nonumber\\&\leq C \Big\|\int_{\frac{[mnv]}{mn}}^{s}\Big| K(s-r)\sigma_l^{k_2}(X^{mn}_{\frac{[mn r]}{mn}})\Big|^2\d r\Big\|^{\frac{1}{2}}_{L^{m/2}}\nonumber\\&\leq C\Big\|\int_{\frac{[mnv]}{mn}}^{s}\Big(K(s-r)\Big)^2|\sigma_l^{k_2}(X^{mn}_{\frac{[mn r]}{mn}})|^2\d r\Big\|^{\frac{1}{2}}_{L^{m/2}}\nonumber\\&\leq C\Big(\int_{\frac{[mnv]}{mn}}^{s}\Big(K(s-r)\Big)^2\||\sigma_l^{k_2}(X^{mn}_{\frac{[mn r]}{mn}})|^2\|_{L^{m/2}}\d r\Big)^{1/2}\nonumber\\&\leq Cn^{-H},
\end{align}
where Lemma \ref{est-3_m} , the boundness of $\sigma$, BDG's inequality, Minkowski's  inequality are used.

Then by the Cauchy-Schwarz inequality, $(26)$ of \cite{FU} and \eqref{esti of stochastic integral-2} we have
\begin{align}
	&\E\Big[A_{7,s}^{(m,n,j,l)}A_{7,v}^{(m,n,j,l)}\Big]\nonumber\\&\leq \E\Big[|A_{7,s}^{(m,n,j,l)}|^2\Big]^{\frac{1}{2}}\E\Big[|A_{7,v}^{(m,n,j,l)}|^2\Big]^{\frac{1}{2}}\nonumber\\&\leq n^{2H}\Big\|\int_{0}^{\frac{[n s]}{n}}\Delta K(n,s,u)\sigma_j^{k_1}(X^{n}_{\frac{[n u]}{n}})\d W^j_u\Big\|_{L^4}\Big\|\int_{\frac{[mns]}{mn}}^{s} K(s-u)\sigma_l^{k_2}(X^{mn}_{\frac{[mn u]}{mn}})\d W^l_u\Big\|_{L^4}\nonumber\\&\quad\cdot \Big\|\int_{0}^{\frac{[n v]}{n}}\Delta K(n,v,u)\sigma_j^{k_1}(X^{n}_{\frac{[n u]}{n}})\d W^j_u\Big\|_{L^4}\Big\|\int_{\frac{[mnv]}{mn}}^{v} K(v-u)\sigma_l^{k_2}(X^{mn}_{\frac{[mn u]}{mn}})\d W^l_u\Big\|_{L^4}\nonumber\\&\leq C.\nonumber
\end{align}
So the bounded convergence theorem yields
\begin{align}
	0\leq \lim_{n\to\infty}\textbf{(1,4,2)}_{j,l}^{ mn,n,k}=\int_{0}^{t}\int_{0}^{t}\lim_{n\to\infty}\mathbb{I}_{(\frac{[ns]}{n},s)}\E\Big[A_{7,s}^{(m,n,j,l)}A_{7,v}^{(m,n,j,l)}\Big]\d v\d s=0\nonumber, 
\end{align}
from which we conclude that $\lim_{n\to\infty}\textbf{(1,4)}_{j,l}^{ mn,n,k}=0$ in the sense of $L^2$, and $\textbf{(4,1)}_{j,l}^{ mn,n,k}\xrightarrow {  L^2} 0$ can be obtained similarly.
\subsubsection{Error analysis of $\textbf{(2,4)}_{j,l}^{mn,n,k}$ and $\textbf{(4,2)}_{j,l}^{mn,n,k}$}Recall that
\begin{align}
	\textbf{(2,4)}_{j,l}^{mn,n,k}&=\int_{0}^{t}n^{2H}\Big(\sigma^{k_1}_j(X^{n}_{\frac{[ns]}{n}})\int_{\frac{[n s]}{n}}^{s}K(s-u)\d W^j_u\Big)\Big(\sigma^{k_2}_j(X^{mn}_{\frac{[mn s]}{mn}})\int_{\frac{[mn s]}{mn}}^{s}K(s-u)\d W^l_u\Big)\d s\nonumber\\&=\int_{0}^{t}n^{2H}\Big(\sigma^{k_1}_j(X^{n}_{\frac{[ns]}{n}})\int_{\frac{[n s]}{n}}^{\frac{[mns]}{mn}}K(s-u)\d W^j_u\Big)\Big(\sigma^{k_2}_j(X^{mn}_{\frac{[mn s]}{mn}})\int_{\frac{[mn s]}{mn}}^{s}K(s-u)\d W^l_u\Big)\d s\nonumber\\&\quad+\int_{0}^{t}n^{2H}\Big(\sigma^{k_1}_j(X^{n}_{\frac{[ns]}{n}})\int_{\frac{[mn s]}{mn}}^{s}K(s-u)\d W^j_u\Big)\Big(\sigma^{k_2}_j(X^{mn}_{\frac{[mn s]}{mn}})\int_{\frac{[mn s]}{mn}}^{s}K(s-u)\d W^l_u\Big)\d s\nonumber\\&:=\textbf{(2,4,1)}_{j,l}^{mn,n,k}+\textbf{(2,4,2)}_{j,l}^{mn,n,k}.\nonumber
\end{align}
For the term $\textbf{(2,4,1)}_{j,l}^{mn,n,k}$, we can rewrite $\E[|\textbf{(2,4,1)}_{j,l}^{mn,n,k}|^2]$ as
\begin{align}
	\E[|\textbf{(2,4,1)}_{j,l}^{ mn,n,k}|^2]&=\E\Big[\int_{0}^{t}\int_{0}^{t}A_{8,s}^{(m,n,j,l)}A_{8,v}^{(m,n,j,l)}\d v \d s\Big]\nonumber\\&=2\E\Big[\int_{0}^{t}\int_{0}^{\frac{[mns]}{mn}}A_{8,s}^{(m,n,j,l)}A_{8,v}^{(m,n,j,l)}\d v \d s\Big]\nonumber\\&\quad+2\E\Big[\int_{0}^{t}\int_{\frac{[mns]}{mn}}^{s}A_{8,s}^{(m,n,j,l)}A_{8,v}^{(m,n,j,l)}\d v \d s\Big]\nonumber\\&:=2(\textbf{(2,4,1,1)}_{j,l}^{ mn,n,k}+\textbf{(2,4,1,2)}_{j,l}^{ mn,n,k})\,, \nonumber
\end{align}
where
\begin{align}\label{def of A 8}
	A_{8,s}^{(m,n,j,l)}=n^{2H}\Big(\sigma^{k_1}_j(X^{n}_{\frac{[ns]}{n}})\int_{\frac{[n s]}{n}}^{\frac{[mns]}{mn}}K(s-u)\d W^j_u\Big)\Big(\sigma^{k_2}_j(X^{mn}_{\frac{[mn s]}{mn}})\int_{\frac{[mn s]}{mn}}^{s}K(s-u)\d W^l_u\Big).
\end{align}
For the term $\textbf{(2,4,1,1)}_{j,l}^{ mn,n,k}$, by tower property and Fubini's theorem, we have
\begin{align}
	\textbf{(2,4,1,1)}_{j,l}^{ mn,n,k}&=\int_{0}^{t}\int_{0}^{\frac{[mns]}{mn}}\E\Big[A_{8,s}^{(m,n,j,l)}A_{8,v}^{(m,n,j,l)}\Big]\d v \d s\nonumber\\&=\int_{0}^{t}\int_{0}^{\frac{[mns]}{mn}}\E\Big[n^{H}\int_{\frac{[mn s]}{mn}}^{s}K(s-u)\d W^j_u|\mathcal{F}_{\frac{[mns]}{mn}}\Big]\nonumber\\\quad&\cdot \sigma^{k_1}_j(X^{n}_{\frac{[ns]}{n}})n^{H}\sigma^{k_2}_j(X^{mn}_{\frac{[mn s]}{mn}})\int_{\frac{[mn s]}{mn}}^{s}K(s-u)\d W^l_uA_{8,v}^{(m,n,j,l)}\Big]\d v \d s=0.\nonumber
\end{align}
For the term $\textbf{(2,4,1,2)}_{j,l}^{ mn,n,k}$, by the Cauchy-Schwarz inequality and \eqref{esti of stochastic integral_4}, we have
\begin{align}
	&\E\Big[A_{8,s}^{(m,n,j,l)}A_{8,v}^{(m,n,j,l)}\Big]\nonumber\\&\leq \E\Big[|A_{8,s}^{(m,n,j,l)}|^2\Big]^{\frac{1}{2}}\E\Big[|A_{8,v}^{(m,n,j,l)}|^2\Big]^{\frac{1}{2}}\nonumber\\&\leq n^{2H}\Big\|\sigma^{k_1}_j(X^{n}_{\frac{[ns]}{n}})\int_{\frac{[n s]}{n}}^{\frac{[mns]}{mn}}K(s-u)\d W^j_u\Big\|_{L^4}\Big\|\sigma^{k_2}_j(X^{mn}_{\frac{[mn s]}{mn}})\int_{\frac{[mn s]}{mn}}^{s}K(s-u)\d W^l_u\Big\|_{L^4}\nonumber\\&\quad\cdot \Big\|\sigma^{k_1}_j(X^{n}_{\frac{[nv]}{n}})\int_{\frac{[n v]}{n}}^{\frac{[mnv]}{mn}}K(v-u)\d W^j_u\Big\|_{L^4}\Big\|\sigma^{k_2}_j(X^{mn}_{\frac{[mn v]}{mn}})\int_{\frac{[mn v]}{mn}}^{v}K(v-u)\d W^l_u\Big\|_{L^4}\nonumber\\&\leq C.\nonumber
\end{align}
So the bounded convergence theorem yields
\begin{align}
	0\leq \lim_{n\to\infty}\textbf{(2,4,1,2)}_{j,l}^{ mn,n,k}=\int_{0}^{t}\int_{0}^{t}\lim_{n\to\infty}\mathbb{I}_{(\frac{[mns]}{mn},s)}\E\Big[A_{6,s}^{(m,n,j,l)}A_{6,v}^{(m,n,j,l)}\Big]\d v\d s=0\nonumber, 
\end{align}
from above we conclude that $\lim_{n\to\infty}\textbf{(2,4,1)}_{j,l}^{ mn,n,k}=0$ in the sense of $L^2$.

For the term $\textbf{(2,4,2)}_{j,l}^{ mn,n,k}$, by \eqref{dcp of stochastic integral}, we have that
\begin{align}
	\textbf{(2,4,2)}_{j,l}^{ mn,n,k}&=\int_{0}^{t}n^{2H}\Big(\sigma^{k_1}_j(X^{n}_{\frac{[ns]}{n}})\int_{\frac{[mn s]}{mn}}^{s}K(s-u)\d W^j_u\Big)\Big(\sigma^{k_2}_j(X^{mn}_{\frac{[mn s]}{mn}})\int_{\frac{[mn s]}{mn}}^{s}K(s-u)\d W^l_u\Big)\d s\nonumber\\&=\int_{0}^{t}n^{2H}E_s^{m,n,j}\int_{\frac{[mn s]}{mn}}^{s}K(s-u)\int_{\frac{[mn s]}{mn}}^{u}K(u-r)\d W^j_r \d W^l_u\d s\nonumber\\&\quad+\int_{0}^{t}n^{2H}E_s^{m,n,j}\int_{\frac{[mn s]}{mn}}^{s}K(s-u)\int_{\frac{[mn s]}{mn}}^{u}K(u-r)\d W^l_r \d W^j_u\d s\nonumber\\&\quad+\int_{0}^{t}n^{2H}E_s^{m,n,j}\int_{\frac{[mn s]}{mn}}^{s}K(s-u)^2\d \langle W^{j}, W^{l} \rangle_u\d s\nonumber\\&:=\textbf{(2,4,2,1)}_{j,l}^{ mn,n,k}+\textbf{(2,4,2,2)}_{j,l}^{ mn,n,k}+\textbf{(2,4,2,3)}_{j,l}^{ mn,n,k}.\nonumber
\end{align}
For the term $\textbf{(2,4,2,1)}_{j,l}^{ mn,n,k}$, set
\begin{align}\label{def of A 9}
	A_{9,s}^{(m,n,j,l)}=n^{2H}E_s^{m,n,j}\int_{\frac{[mn s]}{mn}}^{s}K(s-u)\int_{\frac{[mn s]}{mn}}^{u}K(u-r)\d W^j_r \d W^l_u,
\end{align}
and by Fubini's theorem we have 
\begin{align}
	\E[|\textbf{(2,4,2,1)}_{j,l}^{ mn,n,k}|^2]&=\E\Big[\int_{0}^{t}\int_{0}^{t}A_{9,s}^{(m,n,j,l)}A_{9,v}^{(m,n,j,l)}\d v\d s\Big]\nonumber\\&=2\int_{0}^{t}\int_{0}^{s}\E\Big[A_{9,s}^{(m,n,j,l)}A_{9,v}^{(m,n,j,l)}\Big]\d v\d s\nonumber.
\end{align}
We deal with the above expectation in two cases.  If $v\leq \frac{[mns]}{mn}$, by tower property we have
\begin{align}
	\E\Big[A_{9,s}^{(m,n,j,l)}A_{9,v}^{(m,n,j,l)}\Big]&=\E\Big[\E\Big[n^{H}\int_{\frac{[mn s]}{mn}}^{s}K(s-u)\int_{\frac{[mn s]}{mn}}^{u}K(u-r)\d W^j_r \d W^l_u|\mathcal{F}_{\frac{[mns]}{mn}}\Big]\nonumber\\&\quad\cdot n^{H}E_s^{m,n,j}E_v^{m,n,j}A_{9,v}^{(m,n,j,l)}\Big]=0\nonumber.
\end{align}
If $\frac{[mns]}{mn}<v$, by Fubini's theorem and \eqref{dcp of stochastic integral}, we have
\begin{align}\label{double of A 9}
	&\E\Big[A_{9,s}^{(m,n,j,l)}A_{9,v}^{(m,n,j,l)}\Big]\nonumber\\&=\E\Big[n^{4H}E_s^{m,n,j}E_v^{m,n,j}\int_{\frac{[mns]}{mn}}^{v}K(s-u)K(v-u)\int_{\frac{[mn s]}{mn}}^{u}K(u-r)\d W^j_r\int_{\frac{[mn v]}{mn}}^{u}K(u-r)\d W^j_r\d u\Big]\nonumber\\&=n^{4H}\E\Big[\int_{\frac{[mns]}{mn}}^{v}K(s-u)K(v-u)E^{(3)}_u\d u\Big],
\end{align}
where	
\begin{align}
	E^{(3)}_u:=\E\Big[E_s^{m,n,j}E_v^{m,n,j}\int_{\frac{[mn s]}{mn}}^{u}K(u-r)\d W^j_r\int_{\frac{[mn v]}{mn}}^{u}K(u-r)\d W^j_r\Big].\nonumber
\end{align}
By the Cauchy-Schwarz inequality, the boundness of $\sigma$, Lemma \ref{bound of X n} and \eqref{esti of stochastic integral-2}, we have that
\begin{align}
	\E[E^{(3)}_{u}]&\leq \Big\|\int_{\frac{[mns]}{mn}}^{u} K(u-r)\sigma_l^{k_2}(X^{mn}_{\frac{[mn r]}{mn}})\d W^l_r\Big\|_{L^4}\Big\| \int_{\frac{[mnv]}{mn}}^{u} K(u-r)\sigma_l^{k_2}(X^{mn}_{\frac{[mn r]}{mn}})\d W^l_r\Big\|_{L^4}\|\sigma_j^{k_1}(X^{n}_{\frac{[n s]}{n}})\|_{L^8}\|\sigma_j^{k_1}(X^{n}_{\frac{[n v]}{n}})\|_{L^8}\nonumber\\&\leq Cn^{-2H}.\nonumber
\end{align}
This implies by \eqref{double of A 9} that
\begin{align}
	\Big|\E\Big[A_{9,s}^{(m,n,j,l)}A_{9,v}^{(m,n,j,l)}\Big]\Big|\leq Cn^{2H}\int_{\frac{[mns]}{mn}}^{v}K(s-u)K(v-u)\d u.\nonumber
\end{align}
Therefore, Lemma \ref{esti of C} leads to $ \E\Big[A_{9,s}^{(m,n,j,l)}A_{9,v}^{(m,n,j,l)}\Big]\to 0$. Applying the bounded convergence theorem with respect to $\d v\otimes\d s$, we have $\lim_{n\to\infty}\textbf{(2,4,2,1)}_{j,l}^{ mn,n,k}=0$ in the sense of $L^2$, and $\textbf{(2,4,2,2)}_{j,l}^{ mn,n,k}\xrightarrow { L^2} 0$ can be obtained similarly.

For the last term $\textbf{(2,4,2,3)}_{j,l}^{ mn,n,k}$, this term vanishes if $j\neq l$. Otherwise, by Lemma \eqref{est-3_m}, we have
\begin{align}
	\textbf{(2,4,2,3)}_{j,j}^{ mn,n,k}&=\int_{0}^{t}n^{2H}E_s^{m,n,j}\int_{\frac{[mn s]}{mn}}^{s}K(s-u)^2\d u\d s\nonumber\\&=\frac{1}{2HG}\int_{0}^{t}(n\delta_{(n,s)})^{2H}E_s^{m,n,j}\d s\,. 
\end{align}
By Lemma \ref{rate X-X n} and Lemma \ref{limit distribution-0_m}, we have
\begin{align}
	\textbf{(2,4,2,3)}_{j,j}^{ mn,n,k}\xrightarrow {L^2 } \frac{1}{2HG(2H+1)}\int_{0}^{t}\sigma^{k_1}_j(X_{s})\sigma^{k_2}_j(X_{s})\d s.\nonumber
\end{align}
Therefore, if $j\neq l$, the term $\textbf{(2,4)}_{j,l}^{ mn,n,k}$ vanishes, and if $j=l$, we have 
\begin{align}
	\textbf{(2,4)}_{j,j}^{ mn,n,k}\xrightarrow {L^2 } \frac{1}{2HG(2H+1)}\int_{0}^{t}\sigma^{k_1}_j(X_{s})\sigma^{k_2}_j(X_{s})\d s.\nonumber
\end{align}
Similarly, we also have 
\begin{align}
	\textbf{(4,2)}_{j,j}^{ mn,n,k}\xrightarrow {L^2 } \frac{1}{2HG(2H+1)}\int_{0}^{t}\sigma^{k_1}_j(X_{s})\sigma^{k_2}_j(X_{s})\d s.\nonumber
\end{align}
\subsubsection{Error analysis of $\textbf{(3,3)}_{j,l}^{mn,n,k}$}
By \eqref{dcp of stochastic integral}, we have that
\begin{align}
	\textbf{(3,3)}_{j,l}^{mn,n,k}&=n^{2H}\int_{0}^{t}\Big(\int_{0}^{\frac{[mn s]}{mn}}\Delta K(mn,s,u)\sigma^{k_1}_j(X^{mn}_{\frac{[mn u]}{mn}})\d W^j_u\Big)\Big(\int_{0}^{\frac{[mn s]}{mn}}\Delta K(mn,s,u)\sigma^{k_2}_j(X^{mn}_{\frac{[mn u]}{mn}})\d W^l_u\Big)\d s\nonumber\\&=n^{2H}\int_{0}^{t}\int_{0}^{\frac{[mn s]}{mn}}\Delta K(mn,s,u)\sigma^{k_1}_j(X^{mn}_{\frac{[mn u]}{mn}})\Big(\int_{0}^{\frac{[mn u]}{mn}}\Delta K(mn,u,r)\sigma^{k_2}_j(X^{mn}_{\frac{[mn r]}{mn}})\d W^l_r\Big)\d W^j_u\d s\nonumber\\&\quad+n^{2H}\int_{0}^{t}\int_{0}^{\frac{[mn s]}{mn}}\Delta K(mn,s,u)\sigma^{k_2}_j(X^{mn}_{\frac{[mn u]}{mn}})\Big(\int_{0}^{\frac{[mn u]}{mn}}\Delta K(mn,u,r)\sigma^{k_1}_j(X^{mn}_{\frac{[mn r]}{mn}})\d W^j_r\Big)\d W^l_u\d s\nonumber\\&\quad+n^{2H}\int_{0}^{t}\int_{0}^{\frac{[mn s]}{mn}}\Big(\Delta K(mn,s,u)\Big)^2\sigma^{k_1}_j(X^{mn}_{\frac{[mn u]}{mn}})\sigma^{k_2}_j(X^{mn}_{\frac{[mn u]}{mn}})\d \langle W^{j}, W^{l} \rangle_u\d s\nonumber\\&:=\textbf{(3,3,1)}_{j,l}^{mn,n,k}+\textbf{(3,3,2)}_{j,l}^{mn,n,k}+\textbf{(3,3,3)}_{j,l}^{mn,n,k}.\nonumber
\end{align}
For the term $\textbf{(3,3,1)}_{j,l}^{mn,n,k}$, denote
\begin{align}\label{def of A 10}
	A_{10,s}^{(m,n,j,l)}=n^{2H}\int_{0}^{\frac{[mn s]}{mn}}\Delta K(mn,s,u)\sigma^{k_1}_j(X^{mn}_{\frac{[mn u]}{mn}})\Big(\int_{0}^{\frac{[mn u]}{mn}}\Delta K(mn,u,r)\sigma^{k_2}_j(X^{mn}_{\frac{[mn r]}{mn}})\d W^l_r\Big)\d W^j_u\,. 
\end{align}
  Fubini's theorem gives that
\begin{align}
	\E[|\textbf{(3,3,1)}_{j,l}^{ mn,n,k}|^2]&=\E\Big[\int_{0}^{t}\int_{0}^{t}A_{10,s}^{(m,n,j,l)}A_{10,v}^{(m,n,j,l)}\d v\d s\Big]\nonumber\\&=2\int_{0}^{t}\int_{0}^{s}\E\Big[A_{10,s}^{(m,n,j,l)}A_{10,v}^{(m,n,j,l)}\Big]\d v\d s\nonumber.
\end{align}
By Fubini's theorem and \eqref{dcp of stochastic integral}, we have
\begin{align}\label{double of A 10}
	&\E\Big[A_{10,s}^{(m,n,j,l)}A_{10,v}^{(m,n,j,l)}\Big]\nonumber\\&=\E\Big[n^{4H}\int_{0}^{\frac{[mnv]}{mn}}\Delta K(mn,s,u)\Delta K(mn,v,u)|\sigma^{k_1}_j(X^{mn}_{\frac{[mn u]}{mn}})|^2\Big(\int_{0}^{\frac{[mn u]}{mn}}\Delta K(mn,u,r)\sigma^{k_2}_j(X^{mn}_{\frac{[mn r]}{mn}})\d W^l_r\Big)^2\d u\Big]\nonumber\\&=n^{4H}\E\Big[\int_{0}^{\frac{[mnv]}{mn}}\Delta K(mn,s,u)\Delta K(mn,v,u)E^{(4)}_u\d u\Big],
\end{align}
where	
\begin{align}
	E^{(4)}_u:=\E\Big[|\sigma^{k_1}_j(X^{mn}_{\frac{[mn u]}{mn}})|^2\Big(\int_{0}^{\frac{[mn u]}{mn}}\Delta K(mn,u,r)\sigma^{k_2}_j(X^{mn}_{\frac{[mn r]}{mn}})\d W^l_r\Big)^2\Big].\nonumber
\end{align}
By the Cauchy-Schwarz inequality, the boundness of $\sigma$, Lemma \ref{bound of X n} and \eqref{esti of stochastic integral_1}, we have that
\begin{align}
	\E[E^{(4)}_{u}]&\leq \Big\|\int_{0}^{\frac{[mn u]}{mn}}\Delta K(mn,u,r)\sigma^{k_2}_j(X^{mn}_{\frac{[mn r]}{mn}})\d W^l_r\Big\|^2_{L^4}\|\sigma_j^{k_1}(X^{n}_{\frac{[mn u]}{mn}})\|^2_{L^4}\nonumber\\&\leq Cn^{-2H}.\nonumber
\end{align}
This implies by \eqref{double of A 10} that
\begin{align}
	\Big|\E\Big[A_{10,s}^{(m,n,j,l)}A_{10,v}^{(m,n,j,l)}\Big]\Big|\leq Cn^{2H}\int_{0}^{\frac{[mnv]}{mn}}\Delta K(mn,s,u)\Delta K(mn,v,u)\d u.\nonumber
\end{align}
Therefore, Proposition \ref{esti of A mn} leads to $ \E\Big[A_{10,s}^{(m,n,j,l)}A_{10,v}^{(m,n,j,l)}\Big]\to 0$. Applying the bounded convergence theorem with respect to $\d v\otimes\d s$, we have $\lim_{n\to\infty}\textbf{(3,3,1)}_{j,l}^{ mn,n,k}=0$ in the sense of $L^2$, and $\textbf{(3,3,2)}_{j,l}^{ mn,n,k}\xrightarrow { L^2} 0$ can be obtained similarly.

For the term $\textbf{(3,3,3)}_{j,l}^{ mn,n,k}$, if $j\neq l$, this term vanishes, and for $j=l$, similar to $\textbf{(1,3,3,3)}_{j,l}^{ mn,n,k}$, we have that
\begin{align}
	\textbf{(3,3,3)}_{j,j}^{ mn,n,k}\xrightarrow { L^2 } \frac{1}{(2H+1)m^{2H}G}\int_{0}^{\infty}\mu(r,1)^2\d r\int_{0}^{t}\sigma^{k_1}_j(X_{s})\sigma^{k_2}_j(X_{s})\d s.\nonumber
\end{align}
Therefore, if $j\neq l$, the term $\textbf{(3,3)}_{j,l}^{ mn,n,k}$ vanishes, and if $j=l$, we conclude that
\begin{align}
	\textbf{(3,3)}_{j,j}^{ mn,n,k}\xrightarrow {L^2 } \frac{1}{(2H+1)m^{2H}G}\int_{0}^{\infty}\mu(r,1)^2\d r\int_{0}^{t}\sigma^{k_1}_j(X_{s})\sigma^{k_2}_j(X_{s})\d s.\nonumber
\end{align}
\subsubsection{Error analysis of $\textbf{(3,4)}_{j,l}^{mn,n,k}$ and $\textbf{(4,3)}_{j,l}^{mn,n,k}$}
Similar to $\textbf{(2,3,5)}_{j,l}^{mn,n,k}$, we conclude that $\lim_{n\to\infty}\textbf{(3,4)}_{j,l}^{ mn,n,k}=0$ in the sense of $L^2$, and $\textbf{(4,3)}_{j,l}^{ mn,n,k}\xrightarrow { L^2} 0$ can be obtained similarly.
\subsubsection{Error analysis of $\textbf{(4,4)}_{j,l}^{mn,n,k}$}\label{(4.2.6)}
Similar to the term $\textbf{(2,4)}_{j,l}^{mn,n,k}$, we can get that when $j\neq l$, this term vanishes, if $j=l$, we get
\begin{align}
	\textbf{(4,4)}_{j,j}^{ mn,n,k}\xrightarrow {L^2 } \frac{1}{2HG(2H+1)}\int_{0}^{t}\sigma^{k_1}_j(X_{s})\sigma^{k_2}_j(X_{s})\d s.\nonumber
\end{align}
\subsection{Proof of Lemma \ref{est of V-2}}
\begin{proof}
	Recall the expression of \eqref{def of V}.  We shall  compute the $L^1$-limit of 
	$$\langle V^{mn,n,k,j}, W^{j} \rangle_t=n^{H}\int_{0}^{t}\big(X^{n,k}_s-X^{n,k}_\frac{[ns]}{n}+X^{mn,k}_\frac{[mns]}{mn}-X^{mn,k}_s\big)\d s $$
	for $j\in\{1,\cdots,q\}$. Write $\Delta X^{mn,n}_s=X^{n,k}_s-X^{n,k}_\frac{[ns]}{n}+X^{mn,k}_\frac{[mns]}{mn}-X^{mn,k}_s$. Then it follows from Fubini’s theorem that
	\begin{align}
		\E\Big[|\langle V^{mn,n,k,j}, W^{j} \rangle_t|^2\Big]=2\int_{0}^{t}\int_{0}^{s}n^{2H}\E\Big[\Delta X^{mn,n}_s\Delta X^{mn,n}_v\Big]\d v\d s.\nonumber
	\end{align}
	
	From Lemma \ref{rate X-X n}, together with the Cauchy-Schwarz inequality, Minkowski's inequality, we have that
	\begin{align}
		\E\Big[\Delta X^{mn,n}_s\Delta X^{mn,n}_v\Big]&\leq \|\Delta X^{mn,n}_s\|_{L^2}\|\Delta X^{mn,n}_v\|_{L^2}\nonumber\\&\leq \Big[\|X^{n,k}_s-X^{n,k}_\frac{[ns]}{n}\|_{L^2}+\|X^{mn,k}_\frac{[mns]}{mn}-X^{mn,k}_s\|_{L^2}\Big]\nonumber\\&\quad\cdot \Big[\|X^{n,k}_v-X^{n,k}_\frac{[nv]}{n}\|_{L^2}+\|X^{mn,k}_\frac{[mnv]}{mn}-X^{mn,k}_v\|_{L^2}\Big]\nonumber\\&\leq C\Big[(s-\frac{[ns]}{n})^H+(s-\frac{[mns]}{mn})^H\Big]\cdot\Big[(v-\frac{[nv]}{n})^H+(v-\frac{[mnv]}{mn})^H\Big]\nonumber\\&\leq Cn^{-2H}.\nonumber
	\end{align}
	Therefore, in light of the dominated convergence theorem, it suffices to show $$ n^{2H}\E\Big[\Delta X^{mn,n}_s\Delta X^{mn,n}_v\Big]\to 0$$ 
	for each $s,v$ with $v<s$. We only need to consider the case $v<\frac{[ns]}{n}$.
	In this case by \eqref{def of Double X n m} and the tower property we have
	\begin{align}
		\quad\quad&\E\Big[\Delta X^{mn,n}_s\Delta X^{mn,n}_v\Big]\nonumber\\&=\E\Big[\E\Big[\Delta X^{mn,n}_s|\mathcal{F}_{\frac{[mns]}{mn}}\Big]\Delta X^{mn,n}_v\Big]\nonumber\\&=\E\Big[\Big(\mathcal{A}^{mn,n,k}_{1,s}+\sum_{j=1}^q\int_{0}^{\frac{[n s]}{n}} \Delta K(n,s,u)\sigma^k_j(X^{n}_{\frac{[nu]}{n}})\d W^j_u+\sum_{j=1}^q\sigma^k_j(X^{n}_{\frac{[ns]}{n}})\int_{\frac{[n s]}{n}}^{\frac{[mns]}{mn}}K(s-u)\d W^j_u\nonumber\\&\quad-\sum_{j=1}^q\int_{0}^{\frac{[mn s]}{mn}}\Delta K(mn,s,u)\sigma^k_j(X^{mn}_{\frac{[mn u]}{mn}})\d W^j_u\Big)\Delta X^{mn,n}_v\Big]\nonumber\\&=\E\Big[\mathcal{A}^{mn,n,k}_{1,s}\Delta X^{mn,n}_v\Big]+\sum_{j=1}^{q}\E\Big[\Delta X^{mn,n}_v\mathcal{W}^{1,n,j}_s\Big]+\sum_{j=1}^{q}\E\Big[\Delta X^{mn,n}_v\mathcal{W}^{3,mn,n,j}_s\Big]\nonumber\\&\quad-\sum_{j=1}^{q}\E\Big[\Delta X^{mn,n}_v\mathcal{W}^{1,mn,n,j}_s\Big]\nonumber\\&=\E\Big[\mathcal{A}^{mn,n,k}_{1,s}\Delta X^{mn,n}_v\Big]+\sum_{j=1}^{q}\E\Big[\mathcal{A}^{mn,n,k}_{1,v}\mathcal{W}^{1,n,j}_s\Big]+\sum_{j,l=1}^{q}\E\Big[\mathcal{W}^{1,n,l}_v\mathcal{W}^{1,n,j}_s\Big]\nonumber\\&\quad+\sum_{j,l=1}^{q}\E\Big[\mathcal{W}^{2,n,l}_v\mathcal{W}^{1,n,j}_s\Big]-\sum_{j,l=1}^{q}\E\Big[\mathcal{W}^{1,mn,n,l}_v\mathcal{W}^{1,n,j}_s\Big]-\sum_{j,l=1}^{q}\E\Big[\mathcal{W}^{2,mn,n,l}_v\mathcal{W}^{1,n,j}_s\Big]\nonumber\\&\quad+\sum_{j=1}^{q}\E\Big[\mathcal{A}^{mn,n,k}_{1,v}\mathcal{W}^{3,mn,n,j}_s\Big]+\sum_{j,l=1}^{q}\E\Big[\mathcal{W}^{1,n,l}_v\mathcal{W}^{3,mn,n,j}_s\Big]\nonumber\\&\quad+\sum_{j,l=1}^{q}\E\Big[\mathcal{W}^{2,n,l}_v\mathcal{W}^{3,mn,n,j}_s\Big]-\sum_{j,l=1}^{q}\E\Big[\mathcal{W}^{1,mn,n,l}_v\mathcal{W}^{3,mn,n,j}_s\Big]-\sum_{j,l=1}^{q}\E\Big[\mathcal{W}^{2,mn,n,l}_v\mathcal{W}^{3,mn,n,j}_s\Big]\nonumber\\&\quad+\sum_{j=1}^{q}\E\Big[\mathcal{A}^{mn,n,k}_{1,v}\mathcal{W}^{1,mn,n,j}_s\Big]+\sum_{j,l=1}^{q}\E\Big[\mathcal{W}^{1,n,l}_v\mathcal{W}^{1,mn,n,j}_s\Big]\nonumber\\&\quad+\sum_{j,l=1}^{q}\E\Big[\mathcal{W}^{2,n,l}_v\mathcal{W}^{1,mn,n,j}_s\Big]-\sum_{j,l=1}^{q}\E\Big[\mathcal{W}^{1,mn,n,l}_v\mathcal{W}^{1,mn,n,j}_s\Big]-\sum_{j,l=1}^{q}\E\Big[\mathcal{W}^{2,mn,n,l}_v\mathcal{W}^{1,mn,n,j}_s\Big]\nonumber\\&:=\mathcal{T}_1+\sum_{j=1}^{q}\mathcal{T}^j_{21}+\sum_{j,l=1}^{q}\mathcal{T}^{jl}_{22}+\sum_{j,l=1}^{q}\mathcal{T}^{jl}_{23}-\sum_{j,l=1}^{q}\mathcal{T}^{jl}_{24}-\sum_{j,l=1}^{q}\mathcal{T}^{jl}_{25}\nonumber\\&\quad+\sum_{j=1}^{q}\mathcal{T}^j_{31}+\sum_{j,l=1}^{q}\mathcal{T}^{jl}_{32}+\sum_{j,l=1}^{q}\mathcal{T}^{jl}_{33}-\sum_{j,l=1}^{q}\mathcal{T}^{jl}_{34}-\sum_{j,l=1}^{q}\mathcal{T}^{jl}_{35}\nonumber\\&\quad+\sum_{j=1}^{q}\mathcal{T}^j_{41}+\sum_{j,l=1}^{q}\mathcal{T}^{jl}_{42}+\sum_{j,l=1}^{q}\mathcal{T}^{jl}_{43}-\sum_{j,l=1}^{q}\mathcal{T}^{jl}_{44}-\sum_{j,l=1}^{q}\mathcal{T}^{jl}_{45}\,, \nonumber
	\end{align}
	where $\Delta X^{mn,n}_s=\mathcal{A}^{mn,n,k}_{1,s}+\sum_{j=1}^{q}\big[\mathcal{W}^{1,n,j}_s+\mathcal{W}^{2,n,j}_s-\mathcal{W}^{1,mn,n,j}_s-\mathcal{W}^{2,mn,n,j}_s\big]$ and
	\begin{align}
		\mathcal{W}^{1,n,j}_s&:=\int_{0}^{\frac{[n s]}{n}} \Delta K(n,s,u)\sigma^k_j(X^{n}_{\frac{[nu]}{n}})\d W^j_u,\nonumber\\
		\mathcal{W}^{2,n,j}_s&:=\sigma^k_j(X^{n}_{\frac{[ns]}{n}})\int_{\frac{[n s]}{n}}^{s}K(s-u)\d W^j_u,\nonumber\\
		\mathcal{W}^{1,mn,n,j}_s&:=\int_{0}^{\frac{[mn s]}{mn}}\Delta K(mn,s,u)\sigma^k_j(X^{mn}_{\frac{[mn u]}{mn}})\d W^j_u,\nonumber\\
		\mathcal{W}^{2,mn,n,j}_s&:=\sigma^k_j(X^{mn}_{\frac{[mn s]}{mn}})\int_{\frac{[mn s]}{mn}}^{s}K(s-u)\d W^j_u,\nonumber\\
		\mathcal{W}^{3,mn,n,j}_s&:=\sigma^k_j(X^{n}_{\frac{[ns]}{n}})\int_{\frac{[n s]}{n}}^{\frac{[mns]}{mn}}K(s-u)\d W^j_u.\nonumber
	\end{align}
	For the term $\mathcal{T}_1$, by Fubini's theorem, the Cauchy-Schwarz inequality, Lemmas \ref{bound of X n}, \ref{rate X-X n}, the properties of singular kernel $K$, we have
	\begin{align}
		\mathcal{T}_1&=\E\Big[\mathcal{A}^{mn,n,k}_{1,s}\Delta X^{mn,n}_v\Big]\nonumber\\&\leq C\int_{0}^{\frac{[n s]}{n}}\Delta K(n,s,u)\E\Big[\big(b^k(X^{n}_{\frac{[n u]}{n}})-b^k(X^{mn}_{\frac{[mn u]}{mn}})\Delta X^{mn,n}_v\big)\Big]\d u\nonumber\\&\quad+C\int_{0}^{\frac{[n s]}{n}}\Delta K(mn,n,s,u)\E\Big[b^k(X^{mn}_{\frac{[mn u]}{mn}})\Delta X^{mn,n}_v\Big]\d u\nonumber\\&\quad+C\int_{\frac{[ns]}{n}}^{\frac{[mn s]}{mn}}\Delta K(mn,s,u)\E\Big[b^k(X^{mn}_{\frac{[mn u]}{mn}})\Delta X^{mn,n}_v\Big]\d u\nonumber\\&\quad+C\E\Big[b^k(X^{n}_{\frac{[n s]}{n}})\Delta X^{mn,n}_v\Big]\int_{\frac{[n s]}{n}}^{\frac{[mn s]}{mn}} K(s-u)\d u\\&\leq C\Big[\sup_{0\leq s\leq T}\|b^k(X^{n}_{\frac{[n s]}{n}})\|_{L^2}+\sup_{0\leq s\leq T}\|b^k(X^{mn}_{\frac{[mn s]}{mn}})\|_{L^2}\Big]\sup_{0\leq v\leq T}\|\Delta X^{mn,n}_v\|_{L^2}\Big[\int_{0}^{\frac{[n s]}{n}}\Delta K(n,s,u)\d u\nonumber\\&\quad+\int_{0}^{\frac{[n s]}{n}}\Delta K(mn,n,s,u)\d u+\int_{\frac{[ns]}{n}}^{\frac{[mn s]}{mn}}\Delta K(mn,s,u)\d u+\int_{\frac{[n s]}{n}}^{\frac{[mn s]}{mn}} K(s-u)\d u\Big]\nonumber\\&\leq Cn^{-(2H+1/2)}.\nonumber
	\end{align}
	Hence, $n^{2H}\mathcal{T}_1\leq Cn^{-1/2}\to 0.$ Similarly, we   can get $n^{2H}\mathcal{T}^j_{21}\leq Cn^{-1/2}\to 0$, $n^{2H}\mathcal{T}^j_{31}\leq Cn^{-1/2}\to 0$ and $n^{2H}\mathcal{T}^{j}_{41}\leq Cn^{-1/2}\to 0$ for $j\in\{1,\cdots,q\}$.
	
	For the term $\mathcal{T}^{jl}_{22}$, by \eqref{dcp of stochastic integral}, Fubini's theorem, Lemmas \ref{bound of X n}, \ref{esti of A n} and the Cauchy-Schwarz inequality, we have
	\begin{align}
		n^{2H}	\mathcal{T}^{jl}_{22}&=n^{2H}\E\Big[\int_{0}^{\frac{[n v]}{n}} \Delta K(n,v,u)\sigma^k_l(X^{n}_{\frac{[nu]}{n}})\d W^l_u\cdot\int_{0}^{\frac{[n s]}{n}} \Delta K(n,s,u)\sigma^k_j(X^{n}_{\frac{[nu]}{n}})\d W^j_u\Big]\nonumber\\&=n^{2H}\E\Big[\int_{0}^{\frac{[nv]}{n}} \Delta K(n,v,u) \Delta K(n,s,u)\sigma^k_l(X^{n}_{\frac{[nu]}{n}})\sigma^k_j(X^{n}_{\frac{[nu]}{n}})\d \langle W^{l}, W^{j} \rangle_u\Big]\nonumber\\&\leq n^{2H}\int_{0}^{\frac{[nv]}{n}} \Delta K(n,v,u) \Delta K(n,s,u)\E\Big[\sigma^k_l(X^{n}_{\frac{[nu]}{n}})\sigma^k_j(X^{n}_{\frac{[nu]}{n}})\Big]\d u\nonumber\\&\leq n^{2H}\int_{0}^{\frac{[nv]}{n}} \Delta K(n,v,u) \Delta K(n,s,u)\|\sigma^k_l(X^{n}_{\frac{[nu]}{n}})\|_{L^2}\|\sigma^k_j(X^{n}_{\frac{[nu]}{n}})\|_{L^2}\d u\nonumber\\&\leq Cn^{2H}\int_{0}^{\frac{[nv]}{n}} \Delta K(n,v,u) \Delta K(n,s,u)\d u\to 0\quad \text{as $n\to\infty$} \nonumber
	\end{align}
	for all $j,l\in\{1,\cdots,q\}$. 
	
	For the term $\mathcal{T}^{jl}_{22}$,   following  the same idea as above,   we have that
	\begin{align}
		n^{2H}\mathcal{T}^{jl}_{23}&\leq Cn^{2H}\int_{\frac{[nv]}{n}}^{v} K(v-u) \Delta K(n,s,u)\d u\nonumber\\&= 
		Cn^{2H} \int_0^{v-\frac{[n v]}{n}}\Big|(r+s-v)^{H-1/2}-(r+\frac{[n s]}{n}-v)^{H-1 / 2}\Big| r^{H-1 / 2} \mathrm{d} r\nonumber\\&
		=C \int_0^{n v-[n v]}\left|(z+n s-n v)^{H-1 / 2}-(z+[n s]-n v)^{H-1 / 2}\right| z^{H-1 / 2} \mathrm{d} z\to 0\quad \text{as $n\to\infty$},\nonumber
	\end{align}
	since \begin{align}
		\int_0^1\left|(z+n s-n v)^{H-1 / 2}-(z+[n s]-n v)^{H-1 / 2}\right| z^{H-1 / 2} \mathrm{d} z \rightarrow 0.\nonumber
	\end{align}
	For the term $\mathcal{T}^{jl}_{24}$, we have
	\begin{align}
		n^{2H}\mathcal{T}^{jl}_{24}&\leq Cn^{2H}\int_{0}^{\frac{[mnv]}{mn}} \Delta K(mn,v,u) \Delta K(n,s,u)\d u\to 0\quad \text{as $n\to\infty$},\nonumber
	\end{align}
	where Lemma \ref{esti of D} is used.\\
	For the term $\mathcal{T}^{jl}_{25}$, we have
	\begin{align}
		n^{2H}\mathcal{T}^{jl}_{24}&\leq Cn^{2H}\int_{\frac{[mnv]}{mn}}^{v} K(v-u) \Delta K(n,s,u)\d u\nonumber\\&\leq Cn^{2H}\int_{\frac{[nv]}{n}}^{v} K(v-u) \Delta K(n,s,u)\d u \to 0\quad \text{as $n\to\infty$},\nonumber
	\end{align}
	where the proof of $\mathcal{T}^{jl}_{23}$ is used.\\
	For the terms $\mathcal{T}^{jl}_{32}$-$\mathcal{T}^{jl}_{35}$, by the tower property, we have
	\begin{align}
		n^{2H}\mathcal{T}^{jl}_{32}=n^{2H}\E\Big[\int_{0}^{\frac{[n v]}{n}} \Delta K(n,v,u)\sigma^l_j(X^{n}_{\frac{[nu]}{n}})\d W^l_u\E\Big[\sigma^k_j(X^{n}_{\frac{[ns]}{n}})\int_{\frac{[n s]}{n}}^{\frac{[mns]}{mn}}K(s-u)\d W^j_u|\mathcal{F}_{\frac{[ns]}{n}}\Big]\Big]=0,\nonumber
	\end{align}
	\begin{align}
		n^{2H}\mathcal{T}^{jl}_{33}=n^{2H}\E\Big[\sigma^k_l(X^{n}_{\frac{[nv]}{n}})\int_{\frac{[n v]}{n}}^{v}K(v-u)\d W^l_u\E\Big[\sigma^k_j(X^{n}_{\frac{[ns]}{n}})\int_{\frac{[n s]}{n}}^{\frac{[mns]}{mn}}K(s-u)\d W^j_u|\mathcal{F}_{\frac{[ns]}{n}}\Big]\Big]=0,\nonumber
	\end{align}
	\begin{align}
		n^{2H}\mathcal{T}^{jl}_{34}=n^{2H}\E\Big[\int_{0}^{\frac{[mn v]}{mn}}\Delta K(mn,v,u)\sigma^k_l(X^{mn}_{\frac{[mn u]}{mn}})\d W^l_u\E\Big[\sigma^k_j(X^{n}_{\frac{[ns]}{n}})\int_{\frac{[n s]}{n}}^{\frac{[mns]}{mn}}K(s-u)\d W^j_u|\mathcal{F}_{\frac{[ns]}{n}}\Big]\Big]=0,\nonumber
	\end{align}
	and
	\begin{align}
		n^{2H}\mathcal{T}^{jl}_{35}=n^{2H}\E\Big[\int_{0}^{\frac{[mn v]}{mn}}\Delta K(mn,v,u)\sigma^k_l(X^{mn}_{\frac{[mn u]}{mn}})\d W^l_u\E\Big[\sigma^k_j(X^{n}_{\frac{[ns]}{n}})\int_{\frac{[n s]}{n}}^{\frac{[mns]}{mn}}K(s-u)\d W^j_u|\mathcal{F}_{\frac{[ns]}{n}}\Big]\Big]=0.\nonumber
	\end{align}
	For the term $\mathcal{T}^{jl}_{42}$, we have
	\begin{align}
		n^{2H}\mathcal{T}^{jl}_{42}&\leq Cn^{2H}\int_{0}^{\frac{[nv]}{n}} \Delta K(n,v,u) \Delta K(mn,s,u)\d u\to 0\quad \text{as $n\to\infty$},\nonumber
	\end{align}
	where Lemma \ref{esti of E} is used.\\
	For the term $\mathcal{T}^{jl}_{43}$, we have
	\begin{align}
		n^{2H}\mathcal{T}^{jl}_{43}\leq Cn^{2H}\int_{\frac{[nv]}{n}}^{v} K(v-u) \Delta K(n,s,u)\d u \to 0\quad \text{as $n\to\infty$},\nonumber
	\end{align}
	where the proof of $\mathcal{T}^{jl}_{23}$ is used.\\
	For the term $\mathcal{T}^{jl}_{44}$, we have
	\begin{align}
		n^{2H}\mathcal{T}^{jl}_{44}\leq Cn^{2H}\int_{0}^{\frac{[mnv]}{mn}}\Delta K(mn,s,u)\Delta K(mn,v,u)\d u \to 0\quad \text{as $n\to\infty$},\nonumber
	\end{align}
	where Lemma \ref{esti of A mn} is used.\\
	For the term $\mathcal{T}^{jl}_{45}$, we have
	\begin{align}
		n^{2H}\mathcal{T}^{jl}_{45}&\leq Cn^{2H}\int_{\frac{[mnv]}{mn}}^{v} K(v-u) \Delta K(mn,s,u)\d u\nonumber\\&= 
		Cn^{2H} \int_0^{v-\frac{[mn v]}{mn}}\Big|(r+s-v)^{H-1/2}-(r+\frac{[mn s]}{mn}-v)^{H-1 / 2}\Big| r^{H-1 / 2} \mathrm{d} r\nonumber\\&
		\leq C \int_0^{mn v-[mn v]}\left|(z+mn s-mn v)^{H-1 / 2}-(z+[mn s]-mn v)^{H-1 / 2}\right| z^{H-1 / 2} \mathrm{d} z\to 0\quad \text{as $n\to\infty$},\nonumber
	\end{align}
	since \begin{align}
		\int_0^1\left|(z+mn s-mn v)^{H-1 / 2}-(z+[mn s]-mn v)^{H-1 / 2}\right| z^{H-1 / 2} \mathrm{d} z \rightarrow 0.\nonumber
	\end{align}
	The proof is complete.
\end{proof}
	\appendix
	\section{Estimates for stochastic Volterra equations and its Euler scheme}
	\begin{lem}\cite{LHG}\label{basic lemma} Let $K(u)=u^{H-1/2}/\Gamma(H+1/2)$.  Then 
		\[
		\left\{	\begin{split}
			&\int_{0}^{h}K(t)\d t=O(h^{H+1/2})\,, \quad \int_{0}^{T}\big(K(t+h)-K(t)\big)\d t=O(h^{H+1/2})\,,\\ 
			& \int_{0}^{h}K(t)^2\d t =O(h^{2H})\,, \quad \int_{0}^{T}\big(K(t+h)-K(t)\big)^2\d t =O(h^{2H})\,, 
		\end{split}\right. 
		\]
		where the notation $A(h)=O(B(h))$ for two quantities  $A$ and $B$ means that there is a constant $C$ such that
		$A(h)\le CB(h) $  for all small $h$. 
		
		Moreover, for any adapted $\mathbb{R}^{d}$-valued process $Y$ and $\mathbb{R}^{d\times m}$-valued process $Z$,  the following inequalities hold 
		
		(i) For $p\geq 2$ and $t\in [0,T]$,
		$$\E\Big[\Big|\int_{0}^{t}K(t-s)Y_s\d s\Big|^p\Big]\leq C\int_{0}^{t}K(t-s)\cdot \E[|Y_s|^p]\d s\,.$$
		
		(ii) For $p\geq 2$ and $t\in [0,T]$,
		$$\E\Big[\Big|\int_{0}^{t}K(t-s)Z_s\d W_s\Big|^p\Big]\leq C\int_{0}^{t}K(t-s)^2\cdot E[\|Z_s\|^p]\d s\,.$$
		
		(iii) For $p\geq 1, t\in [0,T]$ and $h\geq 0$ with $t+h\leq T$,
		$$\E\Big[\Big|\int_{0}^{t}(K(t+h-s)-K(t-s))Y_s\d s\Big|^p\Big]+\E\Big[\Big|\int_{t}^{t+h}K(t+h-s)Y_s\d s\Big|^p\Big]\leq C h^{(H+1/2)p}\sup_{t\in[0,T]}\E[|Y_t|^p]. $$
		
		(iv) For $p\geq 2, t\in [0,T]$ and $h\geq 0$ with $t+h\leq T$,
		$$\E\Big[\Big|\int_{0}^{t}(K(t+h-s)-K(t-s))Z_s\d W_s\Big|^p\Big]+E\Big[\Big|\int_{t}^{t+h}K(t+h-s)Z_s\d W_s\Big|^p\Big]\leq C h^{H p}\sup_{t\in[0,T]}\E[\|Z_t\|^p],$$
		where $C$ depends only on $K,p$ and $T$.
	\end{lem}
	
	\begin{lem}\cite{ALP}\label{bound of X}
		Under the assumption $H_{b,\sigma}$,   then
		$$\sup_{t\in[0,T]}\E[|X_t|^p]\leq C \quad \forall ] p\geq 1 \,,  $$
		where $C$ is a constant that only depends on $|X_0|,K,L_1,p$ and $T$.
	\end{lem}
	\begin{lem}\cite{ALP}\label{continuous of X}
		Let $p>H^{-1}$. Then
		$$\E[|X_t-X_s|^p]\leq C|t-s|^{Hp}, \quad t,s\in [0,T]$$
		and $X$ admits a version which is $\mathrm{H\ddot{o}lder}$ continuous on $[0,T]$ of any order $\a<H-p^{-1}$. Denoting this version again by $X$, one has
		$$\E\Big[\Big(\sup_{0\leq s\leq t\leq T}\frac{|X_t-X_s|}{|t-s|^{\a}}\Big)^p\Big]\leq C_{\a}$$
		for all $\a\in[0,H-p^{-1})$, where $C_{\a}$ is a constant. As a consequence, we can regard $X-X_0$ as a $\mathcal{C}^{\alpha}_{0}$ valued random variable for any $\alpha<H$.
	\end{lem}
	\begin{lem}\cite{FU}\label{bound of X n}
		Let $p\geq 1$, then
		$$\sup_{t\in [0,T]}\E[|X^n_t|^p]\leq C,$$
		where $C$ is a constant that  depends  only  on $|X_0|,K,L_1,p,T$ and the derivatives of $b$ and $\sigma$.
	\end{lem}
	\begin{lem}\cite{FU}\label{continuous of X n}
		Let $p>H^{-1}$. Then
		$$\E[|X^n_t-X^n_s|^p]\leq C|t-s|^{H p}, \quad t,s\in [0,T]$$
		and $X^n$ admits a version which is $\mathrm{H\ddot{o}lder}$ continuous on $[0,T]$ of any order $\a<H-p^{-1}$. Denoting this version again by $X^n$, one has
		$$\E\Big[\Big(\sup_{0\leq s\leq t\leq T}\frac{|X^n_t-X^n_s|}{|t-s|^{\a}}\Big)^p\Big]\leq C_{\a}$$
		for all $\a\in[0,H-p^{-1})$, where $C_{\a}$ is a constant. As a consequence, we can regard $X^n-X_0$ as a $\mathcal{C}^{\alpha}_{0}$ valued random variable for any $\alpha<H$.
	\end{lem}
	\begin{lem}\cite{FU}\label{rate X-X n}
		Under the assumption $H_{b,\sigma}$,   for any $p\geq 1$ the process $X_t-X^n_t$ uniformly converges to zero in $L^p$ with the rate $n^{-Hp}$ as $n$ goes to infinity, that is
		$$\sup_{t\in [0,T]}\E[|X_t-X^n_t|^p]\leq C n^{-H p},$$
		where $C$ is a positive constant which does not depend on $n$.
	\end{lem}
	Let us point out that the index $n$ in Lemmas \ref{bound of X n}, \ref{continuous of X n}, \ref{rate X-X n} can be replaced  by the index $mn$.
	\begin{lem}\label{rate of mn,n}
		Let $p>H^{-1}$. Then for $n,m\in\mathbb{N}/\{0\}$
		$$\E[|X^{mn}_\frac{[mns]}{mn}-X^{n}_\frac{[ns]}{n}|^p]\leq Cn^{-H p}, \quad s\in [0,T],$$
		where $C$ is a positive constant which does not depend on $n,m$.
	\end{lem}
	\begin{proof}
		By Lemma \ref{continuous of X n} and Lemma \ref{rate X-X n} we have
		\begin{align}
			\E[|X^{mn}_\frac{[mns]}{mn}-X^{n}_\frac{[ns]}{n}|^p]&\leq C_p\E[|X^{mn}_\frac{[mns]}{mn}-X^{mn}_s|^p]+C_p\E[|X^{mn}_s-X_s|^p]\nonumber\\&\quad+C_p\E[|X_s-X^{n}_s|^p]+C_p\E[|X^{n}_s-X^{n}_\frac{[ns]}{n}|^p]\nonumber\\&\leq Cn^{-Hp}.\nonumber
		\end{align}
	\end{proof}
	\begin{lem}\cite{FU}
		For all $p\geq 1$ and $\varepsilon\in (0,H)$, under the assumption $H_{b,\sigma}$, there exists a constant $C>0$ which does not depend on $n$ such that
		$$\E\Big[\sup_{t\in [0,T]}|X_t-X_t^n|^p\Big]\leq Cn^{-p(H-\varepsilon)}.$$
	\end{lem}
	\section{Fractional singular kernel caluclus}
	Now let us introduce some notations for $s\in [0,T]$:
	\begin{align}\label{basic types}
		&\delta_{(n,s)}=s-\frac{[ns]}{n},\quad\delta_{(mn,u)}=s-\frac{[mns]}{mn},\quad \delta_{(mn,n,s)}=\frac{[mns]}{mn}-\frac{[ns]}{n}\nonumber\\ &G=\Gamma(H+1/2)^2, 
		\quad \mu(r,y)=(r+y)^{H-1/2}-r^{H-1/2}.
	\end{align}
	The following $L^1$ and $L^2$ bounds associated with fractional singular kernel will  also be used in the  sequel.
	
	\begin{lem}\label{est-1}
		For any $s\in [0,T], m\in\mathbb{N}\backslash \{0,1\} $, we have
					$$(i)\quad\int_{0}^{\frac{[ns]}{n}}K(s-u)-K(\frac{[n s]}{n}-u)\d u=O\left(n^{-(H+1/2)}\right),$$
					$$(ii)\quad\int_{0}^{\frac{[ns]}{n}}K(\frac{[mn s]}{mn}-u)-K(\frac{[n s]}{n}-u)\d u=O\left(n^{-(H+1/2)}\right),$$
		$$(iii)\quad\int_{\frac{[ns]}{n}}^{\frac{[mns]}{mn}}K(s-u)-K(\frac{[mn s]}{mn}-u)\d u=O\left(n^{-(H+1/2)}\right),$$
		$$(iv)\quad\int_{\frac{[ns]}{n}}^{\frac{[mns]}{mn}}K(s-u)\d u=O\left(n^{-(H+1/2)}\right),$$
		$$(v)\quad\int_{\frac{[mns]}{mn}}^{s}K(s-u)\d u=O\left(n^{-(H+1/2)}\right).$$	
	\end{lem}
	\begin{proof}
		For the cliam $(i)$, by the change of variable $z=\frac{[n s]}{n}-u$ and  the property of fractional kernels, we have
		\begin{align}
			0&\geq \int_{0}^{\frac{[ns]}{n}}K(s-u)-K(\frac{[n s]}{n}-u)\d u\nonumber\\&=\int_{0}^{\frac{[ns]}{n}}K(z+s-\frac{[ns]}{n})-K(z)\d z\nonumber\\&\geq \int_{0}^{T}K(z+s-\frac{[ns]}{n})-K(z)\d z=O\left(\Big(s-\frac{[ns]}{n}\Big)^{H+1/2}\right)=O\left(n^{-(H+1/2)}\right).\nonumber
		\end{align}
		For the cliam $(ii)$, by the change of variable $z=\frac{[n s]}{n}-u$ and  the property of fractional kernels, we have
			\begin{align}
				0&\geq \int_{0}^{\frac{[ns]}{n}}K(\frac{[mn s]}{mn}-u)-K(\frac{[n s]}{n}-u)\d u\nonumber\\&=\int_{0}^{\frac{[ns]}{n}}K(z+\frac{[mn s]}{mn}-\frac{[ns]}{n})-K(z)\d z\nonumber\\&\geq \int_{0}^{T}K(z+\frac{[mn s]}{mn}-\frac{[ns]}{n})-K(z)\d z=O\left(\Big(\frac{[mn s]}{mn}-\frac{[ns]}{n}\Big)^{H+1/2}\right)=O\left(n^{-(H+1/2)}\right).\nonumber
			\end{align}
		For the cliam $(iii)$, by the change of variable  $z=\frac{[mn s]}{mn}-u$ and the property of fractional kernels, we have
		\begin{align}
			0&\geq \int_{\frac{[ns]}{n}}^{\frac{[mns]}{mn}}K(s-u)-K(\frac{[mn s]}{mn}-u)\d u\nonumber\\&=\int_{\frac{[mns]}{mn}-\frac{[ns]}{n}}^{\frac{[mns]}{mn}}K(z+s-
			\frac{[mn s]}{mn})-K(z)\d z\nonumber\\&\geq \int_{0}^{T}K(z+s-\frac{[mn s]}{mn})-K(z)\d z=O\left(\Big(s-\frac{[mns]}{mn}\Big)^{H+1/2}\right)=O\left(n^{-(H+1/2)}\right).\nonumber
		\end{align}
	For the cliam $(iv)$ and cliam $(v)$ , a direct computation gives that 
	\begin{align}
		\int_{\frac{[ns]}{n}}^{\frac{[mns]}{mn}}K(s-u)\d u&=O\left(\Big(\frac{[mns]}{mn}-\frac{[ns]}{n}\Big)^{H+1/2}\right)=O\left(n^{-(H+1/2)}\right)\nonumber\\	\int_{\frac{[mns]}{mn}}^{s}K(s-u)\d u&=O\left(\Big(s-\frac{[mns]}{mn}\Big)^{H+1/2}\right)=O\left(n^{-(H+1/2)}\right).\nonumber
	\end{align}
		The proof is complete.
	\end{proof}
	\begin{lem}\label{est-2}
		For $s\in [0,T]$, we have
		$$ n^{2H}\int_{\frac{[ns]}{n}}^{\frac{[mns]}{mn}}\Big|K(s-u)-K(\frac{[mns]}{mn}-u)\Big|^2\d u\leq C,$$
		where $C$ does not depend on $n$.
	\end{lem}
	\begin{proof}
		Let $z=\frac{[mns]}{mn}-u$ and $r=z/\delta_{(mn,s)}$ we have
		\begin{align}
			&\int_{\frac{[ns]}{n}}^{\frac{[mns]}{mn}}\Big|K(s-u)-K(\frac{[mns]}{mn}-u)\Big|^2\d u\nonumber\\&\leq\int_{0}^{\frac{[mns]}{mn}}\Big|K(s-u)-K(\frac{[mns]}{mn}-u)\Big|^2\d u\nonumber\\&=\frac{1}{G}\int_{0}^{\frac{[mns]}{mn}}\Big|\mu(z,s-\frac{[mns]}{mn})\Big|^2\d z\nonumber\\&= \frac{\delta_{(mn,s)}^{2H}}{G}\int_{0}^{\frac{[mns]}{mn\delta_{(mn,s)}}}|\mu(r,1)|^2\d r\nonumber\\&\leq \frac{n^{-2H}}{G}\int_{0}^{\infty}|\mu(r,1)|^2\d r.\nonumber
		\end{align}
		The proof is complete.
	\end{proof}
	\begin{lem}\label{est-2 m}
		For $s\in [0,T]$, we have
		$$ n^{2H}\int_{\frac{[ns]}{n}}^{\frac{[mns]}{mn}}\Big|K(s-u)\Big|^2\d u\leq C,$$
		where $C$ does not depend on $n$.
	\end{lem}
	\begin{proof}
		It is clear that
		\begin{align}
			\int_{\frac{[ns]}{n}}^{\frac{[mns]}{mn}}\Big|K(s-u)\Big|^2\d u=\frac{1}{2HG}\delta^{2H}_{(mn,n,s)}\leq Cn^{-2H}\nonumber
		\end{align}
		The proof is complete.
	\end{proof}
	\begin{lem}[\cite{FU}] \label{est-3}
		For $s\in [0,T]$, we have
		$$(i)\quad n^{2H}\int_{\frac{[ns]}{n}}^{s}\Big|K(s-u)\Big|^2\d u=\frac{1}{2HG}\delta_{(n,s)}^{2H}\leq C,$$
		and
		$$(ii)\quad n^{2H}\int_{0}^{\frac{[ns]}{n}}\Big|K(s-u)-K(\frac{[ns]}{n}-u)\Big|^2\d u\leq C,$$
		where $C$ does not depend on $n$.
	\end{lem}
	\begin{lem}\label{est-3_m}
		For $s\in [0,T]$ and $m\in\mathbb{N}/\{0\}$, we have
		$$(i)\quad n^{2H}\int_{\frac{[mns]}{mn}}^{s}\Big|K(s-u)\Big|^2\d u=\frac{1}{2HG}\delta_{(n,s)}^{2H}\leq C,$$
		$$(ii)\quad n^{2H}\int_{0}^{\frac{[mns]}{mn}}\Big|K(s-u)-K(\frac{[mns]}{mn}-u)\Big|^2\d u\leq C,$$
		and
		$$(iii)\quad n^{2H}\int_{\frac{[ns]}{n}}^{\frac{[mns]}{mn}}\Big|K(s-u)-K(\frac{[mns]}{mn}-u)\Big|^2\d u\leq C,$$
		where $C$ does not depend on $n$.
	\end{lem}
	\begin{proof}
		For the cliam $(i)$, a direct computation gives that
		\begin{align}
			n^{2H}\int_{\frac{[mns]}{mn}}^{s}\Big|K(s-u)\Big|^2\d u=\frac{1}{2HG}\delta_{(mn,s)}^{2H}.\nonumber
		\end{align}
		For the cliam $(ii)$, let $z=\frac{[mns]}{mn}-u$ and $v=z/\delta_{(mn,s)}$ we have
		\begin{align}
			&\int_{0}^{\frac{[mns]}{mn}}\Big|K(s-u)-K(\frac{[mns]}{mn}-u)\Big|^2\d u\nonumber\\&=\frac{1}{G}\int_{0}^{\frac{[mns]}{mn}}\Big|\mu(z,s-\frac{[mns]}{mn})\Big|^2\d z\nonumber\\&= \frac{\delta_{(mn,s)}^{2H}}{G}\int_{0}^{\frac{[mns]}{mn\delta_{(mn,s)}}}|\mu(v,1)|^2\d v\nonumber\\&\leq C\frac{n^{-2H}}{G}\int_{0}^{\infty}|\mu(v,1)|^2\d v.\nonumber
		\end{align}
		The proof is complete.
	\end{proof}
	\begin{lem}\label{est-4}
		For $s\in [0,T]$ and $m\in\mathbb{N}\backslash \{0,1\} $, we have
		$$n^{2H}\int_{0}^{\frac{[ns]}{n}}\Big(K(\frac{[mns]}{mn}-r)-K(\frac{[ns]}{n}-r)\Big)^2\d r\leq C,$$
		where $C$ does not depend on $n$.
	\end{lem}
	\begin{proof}
		Let $z=\frac{[ns]}{n}-r$ and $v=z/\delta_{(mn,n,s)}$ we have
		\begin{align}
			&\int_{0}^{\frac{[ns]}{n}}\Big|K(\frac{[mns]}{mn}-r)-K(\frac{[ns]}{n}-r)\Big|^2\d r\nonumber\\&=\frac{1}{G}\int_{0}^{\frac{[ns]}{n}}\Big|\mu(z,\frac{[mns]}{mn}-\frac{[ns]}{n})\Big|^2\d z\nonumber\\&= \frac{\delta_{(mn,n,s)}^{2H}}{G}\int_{0}^{\frac{[ns]}{n\delta_{(mn,n,s)}}}|\mu(v,1)|^2\d v\nonumber\\&\leq C\frac{n^{-2H}}{G}\int_{0}^{\infty}|\mu(v,1)|^2\d v.\nonumber
		\end{align}
		The proof is complete.
	\end{proof}
	\begin{lem}\cite{FU}\label{esti of A n}
		For $v\leq s$, let
		$$A^{(m)}(v,s)=n^{2H}\int_{0}^{\frac{[nv]}{n}}\Big(K(s-u)-K(\frac{[ns]}{n}-u)\Big)\Big(K(v-u)-K(\frac{[nv]}{n}-u)\Big)\d u.$$  
		Then $\sup_{0\leq v\leq s\leq T}\sup_{n}|A^{(n)}(v,s)|<\infty$ and for $v<s$, $\lim_{n\to\infty} A^{(n)}(v,s)=0.$
	\end{lem}
	If we replace $n$ by $mn$, we can obtain the following proposition.
	\begin{prop}\cite{FU}\label{esti of A mn}
		For $v\leq s$ and $m\in\mathbb{N}/\{0\}$, let
		$$A^{(mn)}(v,s)=n^{2H}\int_{0}^{\frac{[nv]}{n}}\Big(K(s-u)-K(\frac{[mns]}{mn}-u)\Big)\Big(K(v-u)-K(\frac{[mnv]}{mn}-u)\Big)\d u.$$  
		Then $\sup_{0\leq v\leq s\leq T}\sup_{n}|A^{(mn)}(v,s)|<\infty$ and for $v<s$, $\lim_{n\to\infty} A^{(mn)}(v,s)=0.$
	\end{prop}
	\begin{lem}\label{esti of A}
		For $v\leq s\in [0,T]$ and $m\in\mathbb{N}/\{0\}$, let
		$$A^{(m,n)}(v,s)=n^{2H}\int_{0}^{\frac{[nv]}{n}}\Big((\frac{[mns]}{mn}-u)^{H-1/2}-(\frac{[ns]}{n}-u)^{H-1/2}\Big)\Big((\frac{[mnv]}{mn}-u)^{H-1/2}-(\frac{[nv]}{n}-u)^{H-1/2}\Big)\d u.$$  
		Then $\sup_{0\leq v\leq s\leq T}\sup_{n}|A^{(m,n)}(v,s)|<\infty$ and for $v<s$, $\lim_{n\to\infty} A^{(m,n)}(v,s)=0.$
	\end{lem}
	\begin{proof}
		By the change of variable $z=[nv]-nu$ we have
		\begin{align}
			0&\leq A^{(m,n)}(v,s)\nonumber\\&=\int_{0}^{[nv]}\Big[(z+\frac{[mns]}{m}-[nv])^{H-1/2}-(z+[ns]-[nv])^{H-1/2}\Big]\nonumber\\&\cdot \Big[(z+\frac{[mnv]}{m}-[nv])^{H-1/2}-z^{H-1/2}\Big]\d z\nonumber\\&\leq \int_{0}^{\infty}\Big[(z+\frac{[mns]}{m}-[nv])^{H-1/2}-(z+[ns]-[nv])^{H-1/2}\Big]\nonumber\\&\cdot \Big[z^{H-1/2}-(z+1)^{H-1/2}\Big]\d z.\nonumber
		\end{align}
		Since
		\begin{align}
			0&\leq (z+\frac{[mns]}{m}-[nv])^{H-1/2}-(z+[ns]-[nv])^{H-1/2}\nonumber\\&\leq (z+\frac{[mns]}{m}-[nv])^{H-1/2}\leq z^{H-1/2},\nonumber
		\end{align}
		for all $z\in (0,\infty)$   we obtain the same bound as Lemma $4.1$ of \cite{FU}. 
		
		Moreover, by the dominated convergence theorem, we have
		\begin{align}
			\int_{0}^{\infty}\Big[(z+\frac{[mns]}{m}-[nv])^{H-1/2}-(z+[ns]-[nv])^{H-1/2}\Big]\cdot \Big[z^{H-1/2}-(z+1)^{H-1/2}\Big]\d z\nonumber\to 0,
		\end{align}
		as $n\to\infty$ for $v<s$, because $(z+\frac{[mns]}{m}-[nv])^{H-1/2}-(z+[ns]-[nv])^{H-1/2}\to 0-0=0$.
	\end{proof}
	\begin{lem}\label{esti of B}
		For $\frac{[ns]}{n}<v\leq s\in [0,T]$ and $m\in\mathbb{N}/\{0\}$, let
		$$B^{(m,n)}(v,s)=n^{2H}\int_{\frac{[ns]}{n}}^{\frac{[mnv]}{mn}}\Big((s-u)^{H-1/2}-(\frac{[mns]}{mn}-u)^{H-1/2}\Big)\Big((v-u)^{H-1/2}-(\frac{[mnv]}{mn}-u)^{H-1/2}\Big)\d u.$$  
		Then $\sup_{\frac{[ns]}{n}< v\leq s\leq T}\sup_{n}|B^{(m,n)}(v,s)|<\infty$ and for $v<s$, $\lim_{n\to\infty} B^{(m,n)}(v,s)=0.$
	\end{lem}
	\begin{proof}
		By the change of variable $z=[mnv]-mnu$ we have
		\begin{align}
			0&\leq B^{(m,n)}(v,s)\nonumber\\&=m^{-2H}\int_{0}^{[mnv]-m[ns]}\Big[(z+mns-[mnv])^{H-1/2}-(z+[mns]-[mnv])^{H-1/2}\Big]\nonumber\\&\cdot \Big[(z+mnv-[mnv])^{H-1/2}-z^{H-1/2}\Big]\d z\nonumber\\&\leq C \int_{0}^{\infty}\Big[(z+[mns]-[mnv])^{H-1/2}-(z+mns-[mnv])^{H-1/2}\Big]\nonumber\\&\cdot \Big[z^{H-1/2}-(z+1)^{H-1/2}\Big]\d z\nonumber\\&\leq C\int_{0}^{\infty}z^{H-1/2}\Big[z^{H-1/2}-(z+1)^{H-1/2}\Big]\d z\nonumber,
		\end{align}
		for all $z\in (0,\infty)$. So $B^{(m,n)}$ is bounded by Lemma $4.1$ of \cite{FU}.
		
		Moreover, by the dominated convergence theorem, we have
		\begin{align}
			\int_{0}^{\infty}\Big[(z+[mns]-[mnv])^{H-1/2}-(z+mns-[mnv])^{H-1/2}\Big]\cdot \Big[z^{H-1/2}-(z+1)^{H-1/2}\Big]\d z\nonumber\to 0,
		\end{align}
		as $n\to\infty$ for $v<s$, because $(z+[mns]-[mnv])^{H-1/2}-(z+mns-[mnv])^{H-1/2}\to 0-0=0$.
	\end{proof} 
	\begin{lem}\label{esti of C}\cite{LHG}
		For $\frac{[ns]}{n}\leq v\leq s$, let
		$$C_n(v,s)=n^{2H}\int_{\frac{[ns]}{n}}^{v}K(s-u)K(v-u)\d u.$$  
		Then $\sup_{0\leq \frac{[ns]}{n}\leq v\leq s\leq T}\sup_{n}|C_n(v,s)|<\infty$ and for $v<s$, $\lim_{n\to\infty} C_n(v,s)=0.$
	\end{lem} 
In the next a few  lemmas,  we will discuss the limit of integrals of combinations  of different singular kernels,    which are critical in our proof.
\begin{lem}\label{esti of D}
	For $v\leq \frac{[ns]}{n},s\in [0,T]$ and $m\in\mathbb{N}/\{0\}$, let
	$$D^{(m,n)}(v,s)=n^{2H}\int_{0}^{\frac{[mnv]}{mn}}\Big((s-u)^{H-1/2}-(\frac{[ns]}{n}-u)^{H-1/2}\Big)\Big((v-u)^{H-1/2}-(\frac{[mnv]}{mn}-u)^{H-1/2}\Big)\d u.$$  
	Then $\sup_{0\leq v\leq s\leq T}\sup_{n}|A^{(m,n)}(v,s)|<\infty$ and for $v<s$, $\lim_{n\to\infty} A^{(m,n)}(v,s)=0.$
\end{lem}
\begin{proof}
	By the change of variable $z=[mnv]-mnu$, we have
	\begin{align}
		0&\leq D^{(m,n)}(v,s)\nonumber\\&=m^{-2H}\int_{0}^{[mnv]-m[ns]}\Big[(z+mns-[mnv])^{H-1/2}-(z+m[ns]-[mnv])^{H-1/2}\Big]\nonumber\\&\cdot \Big[(z+mnv-[mnv])^{H-1/2}-z^{H-1/2}\Big]\d z\nonumber\\&\leq C \int_{0}^{\infty}\Big[(z+m[ns]-[mnv])^{H-1/2}-(z+mns-[mnv])^{H-1/2}\Big]\nonumber\\&\cdot \Big[z^{H-1/2}-(z+1)^{H-1/2}\Big]\d z\nonumber\\&\leq C\int_{0}^{\infty}z^{H-1/2}\Big[z^{H-1/2}-(z+1)^{H-1/2}\Big]\d z\nonumber,
	\end{align}
	for all $z\in (0,\infty)$. So $D^{(m,n)}$ is bounded by Lemma $4.1$ of \cite{FU}.
	
	Further, by the dominated convergence theorem, we have
	\begin{align}
		\int_{0}^{\infty}\Big[(z+m[ns]-[mnv])^{H-1/2}-(z+mns-[mnv])^{H-1/2}\Big]\cdot \Big[z^{H-1/2}-(z+1)^{H-1/2}\Big]\d z\nonumber\to 0,
	\end{align}
	as $n\to\infty$ for $v<s$, because $(z+m[ns]-[mnv])^{H-1/2}-(z+mns-[mnv])^{H-1/2}\to 0-0=0$.
\end{proof}
\begin{lem}\label{esti of E}
	For $v\leq \frac{[ns]}{n},s\in [0,T]$ and $m\in\mathbb{N}/\{0\}$, let
	$$E^{(m,n)}(v,s)=n^{2H}\int_{0}^{\frac{[nv]}{n}}\Big((s-u)^{H-1/2}-(\frac{[mns]}{mn}-u)^{H-1/2}\Big)\Big((v-u)^{H-1/2}-(\frac{[nv]}{n}-u)^{H-1/2}\Big)\d u.$$  
	Then $\sup_{0\leq v\leq s\leq T}\sup_{n}|E^{(m,n)}(v,s)|<\infty$ and for $v<s$, $\lim_{n\to\infty} E^{(m,n)}(v,s)=0.$
\end{lem}
\begin{proof}
	By the change of variable $z=[nv]-nu$, we have
	\begin{align}
		0&\leq E^{(m,n)}(v,s)\nonumber\\&=\int_{0}^{[mnv]-m[ns]}\Big[(z+ns-[nv])^{H-1/2}-(z+\frac{[mns]}{m}-[nv])^{H-1/2}\Big]\nonumber\\&\cdot \Big[(z+nv-[nv])^{H-1/2}-z^{H-1/2}\Big]\d z\nonumber\\&\leq C \int_{0}^{\infty}\Big[(z+\frac{[mns]}{m}-[nv])^{H-1/2}-(z+ns-[nv])^{H-1/2}\Big]\nonumber\\&\cdot \Big[z^{H-1/2}-(z+1)^{H-1/2}\Big]\d z\nonumber\\&\leq C\int_{0}^{\infty}z^{H-1/2}\Big[z^{H-1/2}-(z+1)^{H-1/2}\Big]\d z\nonumber,
	\end{align}
	for all $z\in (0,\infty)$. So $E^{(m,n)}$ is bounded by Lemma $4.1$ of \cite{FU}.
	
	Moreover, by the dominated convergence theorem, we have
	\begin{align}
		\int_{0}^{\infty}\Big[(z+\frac{[mns]}{m}-[nv])^{H-1/2}-(z+ns-[nv])^{H-1/2}\Big]\cdot \Big[z^{H-1/2}-(z+1)^{H-1/2}\Big]\d z\nonumber\to 0,
	\end{align}
	as $n\to\infty$ for $v<s$, because $(z+m[ns]-[mnv])^{H-1/2}-(z+mns-[mnv])^{H-1/2}\to 0-0=0$.
\end{proof}
\section{Limit theorems for fractional integral}\label{limit theorems}
\begin{lem}\label{limit distribution-0}\cite{FU}
	Assume that $g(\cdot)\in L^2(0,1)$. Let $H^{(n)}$ and $H$ be stochastic process on $[0,T]$  such that
	$$\E\Big[\int_{0}^{t}\Big|H_s^{(n)}-H_s\Big|^2\d s\Big]\to 0$$
	with $H$ being almost surely continuous. Then, for all $t\in[0,T]$,
	$$\int_{0}^{t}H^{(n)}_s g(ns-[ns]) \d s\xrightarrow[\text { in } L^2]{n \rightarrow \infty} \int_{0}^{1}g(r)\d r \int_0^t H_s \mathrm{d} s.$$
\end{lem}
The above lemma can be directly extended to
\begin{lem}\label{limit distribution-0_m}
	Assume that $g(\cdot)\in L^2(0,1), m\in\mathbb{N}$. Let $H^{(m,n)}$ and $H$ be stochastic process on $[0,T]$  such that
	$$\E\Big[\int_{0}^{t}\Big|H_s^{(m,n)}-H^{(m)}_s\Big|^2\d s\Big]\to 0$$
	with $H$ being almost surely continuous. Then, for all $t\in[0,T]$,
	$$\int_{0}^{t}H^{(m,n)}_s g(ns-[ns]) \d s\xrightarrow[\text { in } L^2]{n \rightarrow \infty} \int_{0}^{1}g(r)\d r \int_0^t H^{(m)}_s \mathrm{d} s.$$
\end{lem}
\subsection{Limit theorem related to $g(\frac{[mn s]}{m}-[n s])$}
\begin{lem}
	Let $g(\cdot)\in L^1(0,1)$ be either nonnegative or nonpositive, for any $m\in\mathbb{N}\backslash \{0,1\} $, let $\{Y^{(m,n)}\}_{n\in\mathbb{N}}$ be
	a sequence of random variables on $[0,t]$
	whose density functions are each
	$$f_{Y^{(m,n)}}(s)=C_{m,n,t}g(\frac{[mn s]}{m}-[n s]), \quad 0<s<t,$$
	where $g(\frac{[mn \cdot]}{m}-[n\cdot])\in L^1(0,1)$ and
	$$C_{m,n,t}=\Big(g_m[\frac{[mnt]-1}{m}]/n+\int_{\frac{[mnt]}{mn}}^{t}g(\frac{[mn s]}{n}-[n s])\d s\Big)^{-1}$$
	is the normalizing constant with 
	$$g_m=\Big[g(0)+g(\frac{1}{m})+\cdots+g(\frac{m-1}{m})\Big]\frac{1}{m}<\infty.$$
	Then $Y^{(m,n)}$ cconverges in law to the uniform distribution on $[0,t]$ as $n$ goes to infinity.
\end{lem}
\begin{proof}
	Firstly, we confirm that for any fixed $m\in\mathbb{N}\backslash \{0,1\} $, $f_{Y^{(m,n)}}(s)$ is certainly a probability density function. 
	\begin{align}\label{lm-1}
		\int_{0}^{t}g(\frac{[mn s]}{m}-[n s])\d s&=\sum_{j=0}^{[mnt]-1}\int_{\frac{j}{mn}}^{\frac{j+1}{mn}}g(\frac{j}{m}-[\frac{j}{m}])\d s+\int_{\frac{[mnt]}{mn}}^{t}g(\frac{[mn s]}{n}-[n s])\d s\nonumber\\&=\sum_{j=0}^{[mnt]-1}g(\frac{j}{m}-[\frac{j}{m}])\frac{1}{mn}+\int_{\frac{[mnt]}{mn}}^{t}g(\frac{[mn s]}{n}-[n s])\d s\nonumber\\&=\sum_{k=0}^{[\frac{[mnt]-1}{m}]}\sum_{km\leq j<(k+1)m}g(\frac{j}{m}-k)\frac{1}{mn}+\int_{\frac{[mnt]}{mn}}^{t}g(\frac{[mn s]}{n}-[n s])\d s.
	\end{align}
	Note that for $km\leq j<(k+1)m$ we can expand $\sum_{km\leq j<(k+1)m}g(\frac{j}{m}-k)$ as
	\begin{align}\label{lm-2}
		\sum_{km\leq j<(k+1)m}g(\frac{j}{m}-k)=g(0)+g(\frac{1}{m})+\cdots+g(\frac{m-1}{m}),
	\end{align}
which is independent of  the choice of $k$. 
	
	Define
	\begin{align}\label{ def of gm}
		g_m=\Big[g(0)+g(\frac{1}{m})+\cdots+g(\frac{m-1}{m})\Big]\frac{1}{m}.	
	\end{align}
	
	Inserting \eqref{lm-2} into \eqref{lm-1}, we obtain that
	\begin{align}\label{lm-3}
		\int_{0}^{t}g(\frac{[mn s]}{m}-[n s])\d s&=\sum_{k=0}^{[\frac{[mnt]-1}{m}]}g_m\frac{1}{n}+\int_{\frac{[mnt]}{mn}}^{t}g(\frac{[mn s]}{n}-[n s])\d s\nonumber\\&=g_m[\frac{[mnt]-1}{m}]/n+\int_{\frac{[mnt]}{mn}}^{t}g(\frac{[mn s]}{n}-[n s])\d s\nonumber\\&:=\big(C_{m,n,t}\big)^{-1}.
	\end{align}
	
	We now show the convergence of the characteristic function of $Y^{(m,n)}$. For all $x\in\mathbb{R}$ and $i=\sqrt{-1}$, by
	\eqref{lm-1} and \eqref{lm-2} we have
	\begin{align}\label{part of F}
		\int_{0}^{t}e^{ixs}f_{Y^{(m,n)}}(s)\d s&=\int_{0}^{t}e^{ixs}C_{m,n,t}g(\frac{[mn s]}{n}-[n s])\d s\nonumber\\&=C_{m,n,t}\sum_{j=0}^{[mnt]-1}\int_{\frac{j}{mn}}^{\frac{j+1}{mn}}e^{ixs}g(\frac{j}{m}-[\frac{j}{m}])\d s+C_{m,n,t}\int_{\frac{[mnt]}{mn}}^{t}e^{ixs}g(\frac{[mn s]}{n}-[n s])\d s\nonumber\\&=C_{m,n,t}\sum_{j=0}^{[mnt]-1}g(\frac{j}{m}-[\frac{j}{m}])\int_{\frac{j}{mn}}^{\frac{j+1}{mn}}e^{ixs}\d s+C_{m,n,t}\int_{\frac{[mnt]}{mn}}^{t}e^{ixs}g(\frac{[mn s]}{n}-[n s])\d s\nonumber\\&=C_{m,n,t}\sum_{k=0}^{[\frac{[mnt]-1}{m}]}\sum_{km\leq j<(k+1)m}g(\frac{j}{m}-k)\int_{\frac{j}{mn}}^{\frac{j+1}{mn}}e^{ixs}\d s\nonumber\\&\quad+C_{m,n,t}\int_{\frac{[mnt]}{mn}}^{t}e^{ixs}g(\frac{[mn s]}{n}-[n s])\d s=F_1+F_2.
	\end{align}
	In order to characterize  the convergence of the above expression we first study the convergence of $C_{m,n,t}$ as $n$ tends to infinity. By \eqref{lm-2} we have
	\begin{align}
		C_{m,n,t}&=\Big(g_m[\frac{[mnt]-1}{m}]/n+\int_{\frac{[mnt]}{mn}}^{t}g(\frac{[mn s]}{n}-[n s])\d s\Big)^{-1}\stackrel{n\to\infty}{\rightarrow }\Big(g_m t\Big)^{-1}.\nonumber
	\end{align}
	So the convergence of the term $F_2$ is clear:
		\begin{align}\label{cov of F 2}
		|F_2|\leq C_{m,n,t}\int_{0}^{T}\mathbb{I}_{(\frac{[mnt]}{mn},t)}(s)\Big|g(\frac{[mn s]}{n}-[n s])\Big|\d s\to 0.
	\end{align}
	For the term $F_1$, by the change of variable $r=mns-j$, we have
	\begin{align}
		F_1&=C_{m,n,t}\sum_{k=0}^{[\frac{[mnt]-1}{m}]}\sum_{km\leq j<(k+1)m}g(\frac{j}{m}-k)\int_{\frac{j}{mn}}^{\frac{j+1}{mn}}e^{ixs}\d s\nonumber\\&=\frac{C_{m,n,t}}{mn}\sum_{k=0}^{[\frac{[mnt]-1}{m}]}\sum_{km\leq j<(k+1)m}g(\frac{j}{m}-k)\int_{0}^{1}e^{ix(r+j)/mn}\d r\nonumber\\&=C_{m,n,t}\int_{0}^{1}e^{ixr/mn}\d r\frac{1}{mn}\sum_{k=0}^{[\frac{[mnt]-1}{m}]}\sum_{km\leq j<(k+1)m}g(\frac{j}{m}-k)e^{ixj/mn}.\nonumber
	\end{align}
	Note that
	\begin{align}
		&\quad\frac{1}{mn}\sum_{k=0}^{[\frac{[mnt]-1}{m}]}\sum_{km\leq j<(k+1)m}g(\frac{j}{m}-k)e^{ix j/mn}\nonumber\\&=g(0)\frac{1}{mn}\sum_{k=0}^{[\frac{[mnt]-1}{m}]}e^{ix\frac{k}{n}}+g(\frac{1}{m})\frac{1}{mn}\sum_{k=0}^{[\frac{[mnt]-1}{m}]}e^{ix(\frac{k}{n}+\frac{1}{mn})}\nonumber\\&\quad+\cdots+g(\frac{m-1}{m})\frac{1}{mn}\sum_{k=0}^{[\frac{[mnt]-1}{m}]}e^{ix(\frac{k}{n}+\frac{m-1}{mn})}\,. \nonumber
	\end{align}
Togather with $k=[\frac{j}{m}]$,  we can rewrite the term $F_1$ as
	\begin{align}
		F_1&=C_{m,n,t}\int_{0}^{1}e^{ixr/mn}\d r\Big[g(0)\frac{1}{mn}\sum_{k=0}^{[\frac{[mnt]-1}{m}]}e^{ix\frac{k}{n}}+g(\frac{1}{m})\frac{1}{mn}\sum_{k=0}^{[\frac{[mnt]-1}{m}]}e^{ix(\frac{k}{n}+\frac{1}{mn})}\nonumber\\&\quad+\cdots+g(\frac{m-1}{m})\frac{1}{mn}\sum_{k=0}^{[\frac{[mnt]-1}{m}]}e^{ix(\frac{k}{n}+\frac{m-1}{mn})}\Big]\nonumber\\&=C_{m,n,t}\int_{0}^{1}e^{ixr/mn}\d r\Big[g(0)\frac{1}{mn}\sum_{j=0}^{[mnt]-1}e^{ix\frac{1}{n}[\frac{j}{m}]}+g(\frac{1}{m})\frac{1}{mn}\sum_{j=0}^{[mnt]-1}e^{ix(\frac{1}{n}[\frac{j}{m}]+\frac{1}{mn})}\nonumber\\&\quad+\cdots+g(\frac{m-1}{m})\frac{1}{mn}\sum_{j=0}^{[mnt]-1}e^{ix(\frac{1}{n}[\frac{j}{m}]+\frac{m-1}{mn})}\Big]\nonumber\\&=C_{m,n,t}\int_{0}^{1}e^{ixr/mn}\d r\sum_{l=0}^{\frac{m-1}{m}}g(l)\frac{1}{mn}\sum_{j=0}^{[mnt]-1}e^{ix(\frac{1}{n}[\frac{j}{m}]+\frac{l}{n})}.\nonumber
	\end{align}
	By the triangle inequality, for every $l$ it holds that
	\begin{align}
		&\Big|\frac{1}{mn}\sum_{j=0}^{[mnt]-1}e^{ix(\frac{1}{n}[\frac{j}{m}]+\frac{l}{n})}-\int_{0}^{t}e^{ixs}\d s\Big|\nonumber\\&\leq \Big|\frac{1}{mn}\sum_{j=0}^{[mnt]-1}e^{ix(\frac{1}{n}[\frac{j}{m}]+\frac{l}{n})}-\int_{0}^{\frac{[mnt]}{mn}}e^{ixs}\d s\Big|+\int_{\frac{[mnt]}{mn}}^{t}e^{ixs}\d s.\nonumber
	\end{align}
	The last term on the right-hand side vanishes as $n$ goes to infinity. The remainder term on the right-hand side is evaluated as
	\begin{align}
		&\Big|\frac{1}{mn}\sum_{j=0}^{[mnt]-1}e^{ix(\frac{1}{n}[\frac{j}{m}]+\frac{l}{n})}-\int_{0}^{\frac{[mnt]}{mn}}e^{ixs}\d s\Big|\nonumber\\&=\Big|\frac{1}{mn}\sum_{j=0}^{[mnt]-1}\Big[e^{ix(\frac{1}{n}[\frac{j}{m}]+\frac{l}{n})}-mn\int_{\frac{j}{mn}}^{\frac{j+1}{mn}}e^{ixs}\d s\Big]\Big|\nonumber\\&=\Big|\frac{1}{mn}\sum_{j=0}^{[mnt]-1}\Big[e^{ix(\frac{1}{n}[\frac{j}{m}]+\frac{l}{n})}-\int_{0}^{1}e^{ix(r+j)/mn}\d r\Big]\Big|\nonumber\\&\leq \frac{1}{mn}\Big|\sum_{j=0}^{[mnt]-1}\Big[e^{ix(\frac{1}{n}[\frac{j}{m}]+\frac{l}{n})}-e^{ix(\frac{j}{mn}+\frac{l}{n})}\Big]\Big|\nonumber\\&\quad+\Big|\frac{1}{mn}\sum_{j=0}^{[mnt]-1}\Big[e^{ix(\frac{j}{mn}+\frac{l}{n})}-e^{ix\frac{j}{mn}}\int_{0}^{1}e^{ixr/mn}\d r\Big]\Big|\nonumber\\&\leq C\frac{1}{mn}\sum_{j=0}^{[mnt]-1}\Big|\frac{1}{n}[\frac{j}{m}]-\frac{1}{n}\frac{j}{m}\Big|
		+\frac{1}{mn}\sum_{j=0}^{[mnt]-1}\int_{0}^{1}\Big|e^{ix\frac{l}{n}}-e^{ixr/mn}\Big|\d r\Big]\nonumber\\&\leq C\frac{[mnt]-1}{mn^2}+C\frac{[mnt]-1}{mn}\big|\frac{l}{n}-\frac{r}{mn}\big|\to 0,\nonumber
	\end{align}
	where $r=mns-j\in(0,1),l\in\{0,\frac{1}{m},\cdots,\frac{m-1}{m}\}$ and for every fixed $x\in\mathbb{R}$ $\Big|e^{ixu}-e^{ixv}\Big|\leq C|u-v|$ is used.
	As a consequence, we have
	\begin{align}
		\lim_{n\to\infty}	\frac{1}{mn}\sum_{j=0}^{[mnt]-1}e^{ix(\frac{1}{n}[\frac{j}{m}]+\frac{l}{n})}=\int_{0}^{t}e^{ixs}\d s.\nonumber
	\end{align}
	Then, by the dominated convergence theorem, Fubini's theorem and \eqref{ def of gm} we have
	\begin{align}
		\lim_{n\to\infty}F_1&=\lim_{n\to\infty}\Big(C_{m,n,t}\int_{0}^{1}e^{ixr/mn}\d r\sum_{l=0}^{\frac{m-1}{m}}g(l)\frac{1}{mn}\sum_{j=0}^{[mnt]-1}e^{ix(\frac{1}{n}[\frac{j}{m}]+\frac{l}{n})}\Big)\nonumber\\&=\lim_{n\to\infty}C_{m,n,t}\sum_{l=0}^{\frac{m-1}{m}}\Big(\int_{0}^{1}e^{ixr/mn}g(l)\Big(\frac{1}{mn}\sum_{j=0}^{[mnt]-1}e^{ix(\frac{1}{n}[\frac{j}{m}]+\frac{l}{n})}\Big)\d r\Big)\nonumber\\&=\Big(g_m t\Big)^{-1}\sum_{l=0}^{\frac{m-1}{m}}g(l)\int_{0}^{t}e^{ix s}\d s=\int_{0}^{t}\frac{1}{t}e^{ix s}\d s.\nonumber
	\end{align}
	Indeed, a dominating function is derived as
	\begin{align}
		\Big|e^{ixr/mn}g(l)\Big(\frac{1}{mn}\sum_{j=0}^{[mnt]-1}e^{ix(\frac{1}{n}[\frac{j}{m}]+\frac{l}{n})}\Big)\Big|\leq \frac{[mnt]-1}{mn}g(l)\leq T|g(l)|.\nonumber
	\end{align}
This yields 
	$$\int_{0}^{t}e^{ixs}f_{Y^{(m,n)}}(s)\d s\to\int_{0}^{t}\frac{1}{t}e^{ix s}\d s\,. $$
	  Since the function $s\mapsto \frac{1}{t}$ is the density function of the uniform distribution on $[0,t]$, this means the convergence of the characteristic functions,   concluding the proof.
\end{proof}
By this lemma, for any continuous function $k$, by the property of convergence in law of $Y^{(m,n)}$ we have
\begin{align}\label{example of k}
	\int_{0}^{t}k(s)g(\frac{[mn s]}{m}-[n s])\d s&=\big(C_{m,n,t}\big)^{-1}\int_{0}^{t}k(s)f_{Y^{(m,n)}}(s)\d s\nonumber\\&\stackrel{n\to\infty}{\rightarrow }g_m t\int_{0}^{t}k(s)\frac{1}{t}\d s=g_m\int_{0}^{t}k(s)\d s.
\end{align}
We will apply this result to stochastic processes.
\begin{lem}\label{limit distribution-1}
	Assume that $g(\cdot),g(\frac{[mn \cdot]}{m}-[n \cdot])\in L^2(0,1)$ and $|g_m|^2<\infty$. For any fixed $m\in\mathbb{N}\backslash \{0,1\} $, let $H^{(m,n)}$ and $H^{(m)}$ be stochastic processes  on $[0,T]$  satisfying 
	$$\E\Big[\int_{0}^{t}\Big|H_s^{(m,n)}-H^{(m)}_s\Big|^2\d s\Big]\to 0$$
	with $H$ being almost surely continuous. Then, for all $t\in[0,T]$,
	$$\int_{0}^{t}H^{(m,n)}_s g(\frac{[mn s]}{m}-[n s]) \d s\xrightarrow[\text { in } L^2]{n \rightarrow \infty} g_m \int_0^t H^{(m)}_s \mathrm{d} s.$$
\end{lem}
\begin{proof}
	It follows from Minkowski's inequality that
	\begin{align}
		&\Big\|\int_{0}^{t}H_s^{(m,n)}g(\frac{[mn s]}{m}-[n s])-g_m\int_{0}^{t}H_s\d s\Big\|_{L_2}\nonumber\\&=\Big\|\int_{0}^{t}\big(H_s^{(m,n)}-H_s\big)g(\frac{[mn s]}{m}-[n s])\d s\Big\|_{L_2}+\Big\|\int_{0}^{t}H_s\big(g(\frac{[mn s]}{m}-[n s])-g_m\big)\d s\Big\|_{L_2}.\nonumber
	\end{align}
	For the first term on the right-hand side, by the Cauchy-Schwarz inequality and $g(\frac{[mn s]}{m}-[n s])\in L^2(0,1)$ we have
	\begin{align}
		&\E\Big[\Big|\big(H_s^{(m,n)}-H^{(m)}_s\big)g(\frac{[mn s]}{m}-[n s])\Big|^2\Big]\nonumber\\&\leq \E\Big[\int_{0}^{t}\Big|H_s^{(m,n)}-H^{(m)}_s\Big|^2\d s\Big]\cdot\int_{0}^{t}|g(\frac{[mn s]}{m}-[n s])|^2\d s\nonumber\\&\leq C\E\Big[\int_{0}^{t}\Big|H_s^{(m,n)}-H^{(m)}_s\Big|^2\d s\Big]\to 0.\nonumber
	\end{align}
	For the remainder term, since $H$ is continuous, according to \eqref{example of k} we have
	\begin{align}
		\Big|\int_{0}^{t}H^{(m)}_s\Big(g(\frac{[mn s]}{m}-[n s])-g_m\Big)\d s\Big|^2\to 0\nonumber
	\end{align}
	holds almost surely. By the Cauchy-Schwarz inequality again we have
	\begin{align}
		\Big|\int_{0}^{t}H^{(m)}_s\Big(g(\frac{[mn s]}{m}-[n s])-g_m\Big)\d s\Big|^2\leq \int_{0}^{t}|H^{(m)}_s|^2\d s\int_{0}^{t}\Big(g(\frac{[mn s]}{m}-[n s])-g_m\Big)^2\d s.\nonumber
	\end{align}
	Since $\E\Big[\int_{0}^{t}|H^{(m)}_s|^2\d s\Big], \int_{0}^{t}|g(\frac{[mn s]}{m}-[n s])|^2\d s$, and $\int_{0}^{t}|g_m|^2\d s$ are bounded, so we have
	\begin{align}
		&\E\Big[\int_{0}^{t}|H^{(m)}_s|^2\d s\int_{0}^{t}\Big(g(\frac{[mn s]}{m}-[n s])-g_m\Big)^2\d s\Big]\nonumber\\&=\E\Big[\int_{0}^{t}|H^{(m)}_s|^2\d s\Big]\int_{0}^{t}\Big(g(\frac{[mn s]}{m}-[n s])-g_m\Big)^2\d s\nonumber\\&\leq C\int_{0}^{t}|g(\frac{[mn s]}{m}-[n s])|^2\d s+CT|g_m|^2<\infty\nonumber.
	\end{align}
	Finally, the DCT with respect to $\mathbb{P}$ gives that 
	\begin{align}
		\E\Big[\int_{0}^{t}H_s\big(g(\frac{[mn s]}{m}-[n s])-g_m\big)\d s\Big]\to 0,\nonumber
	\end{align}
	which concludes the proof.
\end{proof}

\subsection{Limit theorem related to $g(ns-\frac{[mn s]}{m})$}
\begin{lem}
	For any fixed $m\in\mathbb{N}\backslash \{0,1\} $, let $g\in L^1(0,1)$ be either nonnegative or nonpositive and $\int_{0}^{1}g(\frac{r}{m})\d r<\infty$,  let $\{Z^{(m,n)}\}_{n\in\mathbb{N}}$ be
	a sequence of random variables on $[0,t]$
	whose density functions are  
	$$f_{Z^{(m,n)}}(s)=\tilde{C}_{m,n,t}g(ns-\frac{[mn s]}{m}), \quad 0<s<t,$$
	where 
	$$\tilde{C}_{m,n,t}=\Big(\frac{[mnt]}{mn}\int_{0}^{1}g(\frac{r}{m})\d r+\frac{1}{n}\int_{0}^{nt-\frac{[mnt]}{m}}g(u)\d u\Big)^{-1}$$
	is the normalizing constant.
	Then $Z^{(m,n)}$ cconverges in law to the uniform distribution on $[0,t]$ as $n$ goes to infinity.
\end{lem}
\begin{proof}
	Firstly, we confirm that for any fixed $m\in\mathbb{N}\backslash \{0,1\} $, $f_{Z^{(m,n)}}(s)$ is certainly a probability density function. Let $r=mns-j$ and $u=ns-\frac{j}{m}$, we have
	\begin{align}\label{zm-1}
		&\int_{0}^{t}g(ns-\frac{[mn s]}{m})\d s\nonumber\\&=\sum_{j=0}^{[mnt]-1}\int_{\frac{j}{mn}}^{\frac{j+1}{mn}}g(ns-\frac{j}{m})\d s+\int_{\frac{[mnt]}{mn}}^{t}g(ns-\frac{[mn t]}{m})\d s\nonumber\\&=\frac{1}{mn}\sum_{j=0}^{[mnt]-1}\int_{0}^{1}g(\frac{r}{m})\d r+\frac{1}{n}\int_{0}^{nt-\frac{[mnt]}{m}}g(u)\d u\nonumber\\&=\frac{[mnt]}{mn}\int_{0}^{1}g(\frac{r}{m})\d r+\frac{1}{n}\int_{0}^{nt-\frac{[mnt]}{m}}g(u)\d u:=\big(\tilde{C}_{m,n,t}\big)^{-1}.            
	\end{align}
	We now show the convergence of the characteristic function of $Z^{(m,n)}$. For all $x\in\mathbb{R}$ and $i=\sqrt{-1}$, we have
	\begin{align}\label{part of Z}
		&\int_{0}^{t}e^{ixs}f_{Z^{(m,n)}}(s)\d s\nonumber\\&=\int_{0}^{t}e^{ixs}\tilde{C}_{m,n,t}g(ns-\frac{[mn s]}{m})\d s\nonumber\\&=\tilde{C}_{m,n,t}\sum_{j=0}^{[mnt]-1}\int_{\frac{j}{mn}}^{\frac{j+1}{mn}}e^{ixs}g(ns-\frac{j}{m})\d s+\tilde{C}_{m,n,t}\int_{\frac{[mnt]}{mn}}^{t}e^{ixs}g(ns-\frac{[mnt]}{m})\d s\nonumber\\&=\frac{\tilde{C}_{m,n,t}}{mn}\sum_{j=0}^{[mnt]-1}\int_{0}^{1}e^{ix\frac{r+j}{mn}}g(\frac{r}{m})\d r+\frac{\tilde{C}_{m,n,t}}{n}\int_{0}^{nt-\frac{[mnt]}{m}}e^{ix(u+\frac{[mnt]}{m})\frac{1}{n}}g(u)\d u\nonumber\\&:=G_1+G_2.
	\end{align}
	In order to character the convergence of above expression   we first study the convergence of $\tilde{C}_{m,n,t}$ as $n$ tends to infinity. Notice that obviously, 
	\begin{align}
		\tilde{C}_{m,n,t}&=\Big(\frac{[mnt]}{mn}\int_{0}^{1}g(\frac{r}{m})\d r+\frac{1}{n}\int_{0}^{nt-\frac{[mnt]}{m}}g(u)\d u\Big)^{-1}\stackrel{n\to\infty}{\rightarrow }\Big(t\int_{0}^{1}g(\frac{r}{m})\d r\Big)^{-1}.\nonumber
	\end{align}
	So we have the convergence for  the term $G_2$  
	\begin{align}\label{cov of G 2}
		|G_2|\leq \frac{\tilde{C}_{m,n,t}}{n}\int_{0}^{nt-\frac{[mnt]}{m}}|g(u)|\d u\to 0.
	\end{align}
	For the term $G_1$, by the triangle inequality, we have
	\begin{align}
		&\Big|\frac{1}{mn}\sum_{j=0}^{[mnt]-1}e^{ix\frac{j}{mn}}-\int_{0}^{t}e^{ixs}\d s\Big|\nonumber\\&\leq \Big|\frac{1}{mn}\sum_{j=0}^{[mnt]-1}e^{ix\frac{j}{mn}}-\int_{0}^{\frac{[mnt]}{mn}}e^{ixs}\d s\Big|+\int_{\frac{[mnt]}{mn}}^{t}e^{ixs}\d s.\nonumber
	\end{align}
	The last term on the right-hand side vanishes as $n$ goes to infinity. The other term on the right-hand side is evaluated as
	\begin{align}
		&\Big|\frac{1}{mn}\sum_{j=0}^{[mnt]-1}e^{ix\frac{j}{mn}}-\int_{0}^{\frac{[mnt]}{mn}}e^{ixs}\d s\Big|\nonumber\\&=\Big|\frac{1}{mn}\sum_{j=0}^{[mnt]-1}\Big[e^{ix\frac{j}{mn}}-mn\int_{\frac{j}{mn}}^{\frac{j+1}{mn}}e^{ixs}\d s\Big]\Big|\nonumber\\&=\Big|\frac{1}{mn}\sum_{j=0}^{[mnt]-1}\Big[e^{ix\frac{j}{mn}}-\int_{0}^{1}e^{ix(r+j)/mn}\d r\Big]\Big|\nonumber\\&\leq \Big|1-\int_{0}^{1}e^{ix\frac{r}{mn}}\d r\Big|\cdot\frac{1}{mn}\sum_{j=0}^{[mnt]-1}e^{ix\frac{j}{mn}}\nonumber\\&\leq \int_{0}^{1}\Big|1-e^{ix \frac{r}{mn}}\Big|\d r \frac{[mnt]}{mn}\to 0.\nonumber
	\end{align}
	As a consequence, we have
	\begin{align}
		\lim_{n\to\infty}	\frac{1}{mn}\sum_{j=0}^{[mnt]-1}e^{ix\frac{j}{mn}}=\int_{0}^{t}e^{ixs}\d s.\nonumber
	\end{align}
	Then, by the dominated convergence theorem, Fubini's theorem we have
	\begin{align}
		\lim_{n\to\infty}G_1&=\lim_{n\to\infty}\tilde{C}_{m,n,t}\Big(\frac{1}{mn}\sum_{j=0}^{[mnt]-1}e^{ix\frac{j}{mn}}\Big)\int_{0}^{1}e^{ix\frac{r}{mn}}g(\frac{r}{m})\d r\nonumber\\&=\Big(t\int_{0}^{1}g(\frac{r}{m})\d r\Big)^{-1}\int_{0}^{t}e^{ixs}\d s\cdot \int_{0}^{1}g(\frac{r}{m})\d r=\int_{0}^{t}\frac{1}{t}e^{ix s}\d s.\nonumber
	\end{align}
	Indeed, a dominating function is derived as
	\begin{align}
		\Big|\Big(\frac{1}{mn}\sum_{j=0}^{[mnt]-1}e^{ix\frac{j}{mn}}\Big)e^{ix\frac{r}{mn}}g(\frac{r}{m})\Big|\leq \frac{[mnt]-1}{mn}g(\frac{r}{m})\leq T|g(\frac{r}{m})|.\nonumber
	\end{align}
	Thus,
	$$\int_{0}^{t}e^{ixs}f_{Z^{(m,n)}}(s)\d s\to\int_{0}^{t}\frac{1}{t}e^{ix s}\d s\,. $$
	  Since the function $s\mapsto \frac{1}{t}$ is the density function of the uniform distribution on $[0,t]$, this means the convergence of the characteristic functions, which concludes the proof.
\end{proof}
By this lemma, for any continuous function $k$, by the property of convergence in law of $Z^{(m,n)}$ we have
\begin{align}\label{example of k-2}
	\int_{0}^{t}k(s)g(ns-\frac{[mn s]}{m})\d s&=\big(\tilde{C}_{m,n,t}\big)^{-1}\int_{0}^{t}k(s)f_{Z^{(m,n)}}(s)\d s\nonumber\\&\stackrel{n\to\infty}{\rightarrow }t\int_{0}^{1}g(\frac{r}{m})\d r\int_{0}^{t}k(s)\frac{1}{t}\d s=\int_{0}^{1}g(\frac{r}{m})\d r\int_{0}^{t}k(s)\d s.
\end{align}
Similar to the proof of Lemma \ref{limit distribution-1}, we can obtain the following 
lemma, whose proof is omitted.
\begin{lem}\label{limit distribution-2}
	Fix  $m\in\mathbb{N}\backslash \{0,1\} $.  Assume that $h(\cdot)\in L^2(0,1)$ and $\int_{0}^{1}|g(\frac{r}{m})|^2\d r<\infty$. Let $H^{(m,n)}$ and $H^{m}$ be stochastic processes  on $[0,T]$  such that
	$$\E\Big[\int_{0}^{t}\Big|H_s^{(m,n)}-H^{(m)}_s\Big|^2\d s\Big]\to 0$$
	with $H$ being almost surely continuous. Then, for all $t\in[0,T]$,
	$$\int_{0}^{t}H^{(m,n)}_s g(ns-\frac{[mn s]}{m}) \d s\xrightarrow[\text { in } L^2]{n \rightarrow \infty} \int_{0}^{1}g(\frac{r}{m})\d r \int_0^t H^{(m)}_s \mathrm{d} s.$$
\end{lem} 
\begin{remark}\label{rem 3.2}
	In the  other part of our paper,  we shall use the above lemmas for $g(x)=x^{2H}$ with $H\in(0,1/2]$. 
	It is interesting to compute $\int_{0}^{1}g(r)\d r$,  $g_m$   and  $\int_{0}^{1}g(\frac{r}{m})\d r$ for different parameters. 
	
	\renewcommand{\arraystretch}{2}
	\begin{tabular}{|c|@{\extracolsep{2em}}c|c|c|}
		\hline
		Case & $\int_{0}^{1}g(r)\d r$ & $g_m$ & $\int_{0}^{1}g(\frac{r}{m})\d r$\\ \hline
		$H\in(0,1/2),m\in\mathbb{N}\backslash \{0,1\} $ & $\frac{1}{2H+1}$ & $\sum_{j=0}^{m-1}\frac{j^{2H}}{m^{2H+1}}$ & $\frac{1}{(2H+1)m^{2H}}$\\ \hline
		$H=1/2,m\in\mathbb{N}\backslash \{0,1\} $ & $1/2$ & $\frac{m-1}{2m}$ &$\frac{1}{(2H+1)m}$\\ \hline 
		$H\in(0,1/2]$, $m=1$ & $\frac{1}{2H+1}$ & $0$ & $\frac{1}{2H+1}$\\\hline
		$H\in(0,1/2)$, $m\to\infty$ & $\frac{1}{2H+1}$ & $\frac{1}{2H+1}$ & $0$\\\hline
		$H\in(0,1/2]$, $m=1$ & $\frac{1}{2H+1}$ & $0$ & $\frac{1}{2H+1}$\\\hline
		$H=1/2$, $m\to\infty$ & $\frac{1}{2}$ & $\frac{1}{2}$ & $0$\\ \hline
	\end{tabular}
\end{remark}              

	\section*{Conflicts of interests}
	The authors declare no conflict of interests.

	{\bf Acknowledgement} This work was supported by China Scholarship Council, the NSERC discovery fund RGPIN 2024-05941 and a centennial fund of University of Alberta, the NSFC Grant Nos. 12171084, the Jiangsu Provincial Scientific Research Center of Applied Mathematics under Grant No. BK20233002 and  the fundamental Research
	Funds for the Central Universities No. RF1028623037.

	
	\bibliographystyle{amsplain}
	
	\bibliography{ref}
	
\end{document}